\documentclass[11pt]{article}       
\usepackage{geometry}               
\usepackage{graphicx}
\usepackage{caption}
\usepackage{subfigure}
\usepackage{comment}
\usepackage{amsmath,amsfonts,mathrsfs,amssymb,amscd,amsthm,bbm}
\usepackage{array}
\usepackage{extpfeil}
\usepackage{bm}
\usepackage{lipsum}
\usepackage[linesnumbered, ruled]{algorithm2e}
\usepackage{color}
\usepackage{cite}
\usepackage{url}

\newtheorem{theorem}{Theorem}[section]
\newtheorem{assumption}{Assumption}[section]
\newtheorem{lemma}{Lemma}[section]
\newtheorem{corollary}{Corollary}[section]
\newtheorem{proposition}{Proposition}[section]
\newtheorem{remark}{Remark}[section]
\newtheorem{definition}{Definition}[section]
\numberwithin{equation}{section}

\newcommand{\kl}{D_{\mathrm{KL}}}

\newcommand{\bbr}{\mathbb R}

\def \d  {{\rm{d}}}
\def \e  {{\varepsilon}}

\title{Hamilton--Jacobi--Bellman Equations for\\ Maximum Entropy Optimal Control\thanks{This work was supported in part by  the Creative-Pioneering Researchers Program through SNU and  the National Research Foundation of Korea funded by the MSIT(2020R1C1C1009766).}} 

\author{
Jeongho Kim \and Insoon Yang\thanks{Department of Electrical and Computer Engineering, Seoul National University, Seoul 08826, Korea  {\tt\small \{jhkim206, insoonyang\}@snu.ac.kr}.}
}

\date{}

\begin{document}
\maketitle

\pagestyle{myheadings}
\thispagestyle{plain}

\begin{abstract}
Maximum entropy reinforcement learning (RL) methods have been successfully applied to a range of challenging sequential decision-making and control tasks. 
However, most of existing techniques are designed for discrete-time systems. 
As a first step toward their extension to continuous-time systems, 
this paper considers continuous-time deterministic optimal control problems with entropy regularization. 
 Applying the dynamic programming principle, we derive a novel class of Hamilton--Jacobi--Bellman (HJB) equations  and prove that the optimal value function of the maximum entropy control problem corresponds to the unique viscosity solution of the HJB equation. 
Our maximum entropy formulation is  shown to enhance the regularity of the viscosity solution and to be asymptotically consistent as the effect of entropy regularization diminishes. 
A salient feature of the HJB equations is computational tractability. 
Generalized Hopf--Lax formulas can be used to solve the HJB equations in a tractable grid-free manner without the need for numerically optimizing the Hamiltonian. 
We further show that the optimal control is uniquely characterized as Gaussian in the case of control affine systems and that, for linear-quadratic problems, the HJB equation is reduced to a Riccati equation, which can be used to obtain an explicit expression of the optimal control.  
Lastly, we discuss how to extend our results to continuous-time model-free RL by taking an adaptive dynamic programming approach.
To our knowledge, the resulting algorithms are the first data-driven control methods that use an information theoretic exploration mechanism in \emph{continuous time}.
\end{abstract}


\section{Introduction}

The idea of using a stochastic policy  with high entropy has attracted great interest in various sequential decision-making problems over the past decade. 
Such randomized behaviors may encourage the exploration of informative regions of state and action spaces.
In reinforcement learning (RL), 
maximum entropy methods have been recognized as a useful exploration mechanism, effectively balancing the \emph{exploration-exploitation tradeoff}~\cite{Haarnoja2017, Hazan2019}. 
Moreover, maximum entropy policies prescribe all the possible ways of performing a task of interest, instead of  having solely the best way to carry out the task. 
Thus, it has been empirically observed that the resulting policies are robust with respect to perturbations in systems or environments~\cite{Ziebart2010, Haarnoja2018}.
Another benefit of using relative entropy or Kullback-Leibler (KL) regularization is to improve computational tractability in particular settings of Markov decision processes (MDPs)~\cite{todorov2009efficient}.

Maximum entropy optimal control methods have been the best studied in discrete-time RL, where balancing the exploration-exploitation tradeoff is critical.  
Discrete-time MDPs with entropy regularization have been considered 
in~\cite{Ziebart2010, Haarnoja2017}, where it was shown that an associated Bellman equation generalizes its standard counterpart, and the optimal policy is in the form of Boltzmann distributions.\footnote{These results have  been further generalized using the Tsallis entropy in~\cite{Lee2019}.} 
 The Bellman equation has been used to devise variants of value iteration and Q-learning, called soft Q-iteration and soft Q-learning, respectively~\cite{Fox2016, Haarnoja2017}.
Deep RL algorithms based on such maximum entropy formulations have been empirically demonstrated to achieve state-of-the-art performances on several benchmark tasks~\cite{Haarnoja2018}. 
Motivated by the success of maximum entropy RL, 
\cite{wang2019exploration} examined the role of entropy regularization in  continuous-time stochastic control, although a concrete RL or data-driven control method was not proposed.
However, all the existing methods focus on stochastic systems, in which it is natural to use a randomized control policy. 
This motivates us to ask, \emph{Is there an analog of maximum entropy  methods for deterministic (possibly nonlinear) systems?} 

This paper answers the question in the affirmative by deriving and analyzing  novel Hamilton--Jacobi--Bellman (HJB) equations for continuous-time deterministic optimal control problems with entropy regularization.
We adopt a \emph{relaxed control} formulation~\cite{young1942generalized,young1969lectures}
to accommodate randomized control inputs in continuous-time deterministic systems. 
Applying the dynamic programming (DP) principle, we derive the HJB equation and the structure of optimal controls for the maximum entropy control problem. 
Interestingly, our Hamiltonian can be considered as the \emph{soft maximum} of its standard counterpart.
This resembles the structure of the Bellman equation for maximum entropy RL~\cite{Ziebart2010, Haarnoja2017}. 
Another analogy is observed in the form of our optimal control, which is shown to be a  Boltzmann distribution. 
From the perspective of statistical mechanics, 
our Hamiltonian and optimal control can further be interpreted as the negative value of the Helmholtz free energy and the corresponding canonical ensemble, respectively. 
We prove that the optimal value function of our maximum entropy control problem corresponds to the unique \emph{viscosity solution} of the HJB equation. 
A useful byproduct of our maximum entropy formulation is the improved regularity of the value function; specifically, its sub- and super-differentials have at most one element. 
This regularity result is useful in optimal controller synthesis. 
We further show that our value function converges uniformly to the value function of the standard optimal control problem without entropy regularization as the temperature parameter $\alpha$ tends to zero (or, equivalently, as the effect of entropy regularization diminishes). 
This observation confirms the asymptotic consistency of our HJB equations for maximum entropy control. 

An important benefit of our maximum entropy control formulation is 
computational tractability.
In the case of control-affine systems and quadratic control costs, we show that the optimal control is uniquely characterized as a normal distribution with a mean corresponding to the optimal control for the standard problem without entropy regularization. 
Using the structural property and the HJB equation, 
we derive an algebraic Riccati equation and
an explicit expression of the optimal control for maximum entropy linear-quadratic problems. 
When considering fully nonlinear systems and cost functions, 
 generalized Hopf--Lax formulas~\cite{chow2019algorithm} can be used to numerically solve our HJB equation without discretizing the state space.
An important observation is that it is more tractable to use generalized Hopf--Lax formulas in the maximum entropy control case than in the standard case. The reasons are twofold. First, our Hamiltonian can be explicitly computed unlike the standard Hamiltonian involving an optimization problem which is possibly nonconvex. 
Second, our Hamiltonian is differentiable when the vector field and the cost function are differentiable in state as opposed to its standard counterpart. 
Thus, in our maximum entropy setting,
it is tractable to use the characteristic ordinary differential equations (ODEs) for generalized Hopf--Lax formulas.

Returning to the main motivation for using maximum entropy methods,
we discuss how to extend the idea of model-free RL with entropy regularization to the continuous-time setting by employing our HJB framework. 
Specifically, we consider linear-quadratic problems with unknown model parameters and propose maximum entropy methods for data-driven control by taking the \emph{adaptive dynamic programming} approach in~\cite{jiang2014robust}. 
This approach guarantees closed-loop stability during the process of learning as well as convergence to the optimal control under a rank condition.
To the best of our knowledge, these are the first RL-based algorithms
that use an information theoretic exploration mechanism in \emph{continuous time}. 
Unlike conventional continuous-time RL methods that use heuristic exploration mechanisms (e.g., $\epsilon$-greedy, injecting an artificial noise), our algorithms enhance the exploration capability of controls by maximizing their entropy in a principled manner. 
The results of our numerical experiments demonstrate that our maximum-entropy method outperforms its standard counterpart in terms of both learning speed and sample efficiency.
Our numerical studies also confirm the importance of weighting the entropy term in balancing the exploration-exploitation tradeoff.

%
%

The rest of this paper is organized as follows:
In Section~\ref{sec:2}, we introduce the maximum entropy control problem and show the existence of optimal solutions. 
Section~\ref{sec:3} presents the main theoretical results about the HJB equations for maximum entropy control. 
In Section~\ref{sec:4}, we provide the tractable methods for solving the maximum entropy control problems.
In Section~\ref{sec:5}, we discuss RL-based algorithms for learning the maximum entropy optimal control in the linear-quadratic setting without knowing model parameters.
Section~\ref{sec:6} presents the results of numerical experiments to demonstrate the performance of our methods.

\section{Maximum Entropy Optimal Control of Deterministic Continuous-Time Systems}\label{sec:2}

\subsection{Notation}
For any measurable space $X$, we denote the set of all probability measures on $X$ by $\mathcal{P}(X)$. For any bounded set $A$,  let $|A|$ denote its volume. 
 Given $x_0 \in \mathbb{R}^n$ and $R > 0$, we let $B(x_0, R)$ denote the Euclidean ball centered at $x_0$ with radius $R$. 
 For symmetric matrices $A$ and $B$ with the same size,
$A \preceq B$ represents that $B - A$ is a positive semidefinite matrix.

\subsection{Problem Setup}

Consider a deterministic continuous-time dynamical system of the form
\begin{equation}\label{dyn1}
\dot{x}(t)=f(x(t),u(t)),\quad x(t)\in \bbr^n,\quad u(t)\in U\subset\bbr^m,\quad t>0,
\end{equation}
where $x(t)$ and $u(t)$ denote the system state and the control input at time $t$, respectively. Here, $U$ is the set of admissible control actions. Given $\bm{x} \in\bbr^n$ and $t \in [0, T]$, we consider the following cost functional of $u$:
\begin{equation}\label{cost}
J_{\bm{x}, t}(u) := \int_t^T r(x(s),u(s))\,\d s+q(x(T)),\quad x(t)= \bm{x}.
\end{equation}
Here, $r:\bbr^n\times U\to\bbr$ and $q:\bbr^n\to\bbr$ denote a running cost and a terminal cost of interest, respectively.  
Given the initial condition $x(0) = \bm{x}$, the standard finite-horizon optimal control problem can then be formulated as
\begin{equation}\label{opt0}
\min_{u \in \mathcal{U}} \; J_{\bm{x}, 0} (u),
\end{equation}
where 
\[
\mathcal{U}:=\{u:[0,T]\to U \mid \mbox{$u$ is measurable}\}
\]
is the set of admissible controls. 
Throughout the paper, we assume the following standard conditions on $f$, $r$ and $q$:	
	\begin{assumption}\label{assumption}
		\begin{enumerate}
			\item[(i)]  $f:\bbr^n\times U\to \bbr^n$ is continuous.
			\item[(ii)] There exists a constant $C > 0$ such that
			\[|f(\bm{x},\bm{u})|\le C(1+|\bm{x}|+|\bm{u}|).\]
			\item[(iii)] There exists a modulus $\omega_f:[0,+\infty)^2\to[0,+\infty)$ such that \[|f(\bm{x},\bm{u})-f(\bm{y},\bm{u})|\le \omega_f(|\bm{x}-\bm{y}|,R) \quad \forall \bm{u} \in U, \; \forall \bm{x},\bm{y}\in B(0,R) \mbox{ and } \forall R>0.\]
 
			\item[(iv)] For all $\bm{x},\bm{y}\in \bbr^n$ and $\bm{u}\in U$,
			\[(f(\bm{x},\bm{u})-f(\bm{y},\bm{u}))\cdot (\bm{x}-\bm{y})\le L|\bm{x}-\bm{y}|^2.\]
			
			\item[(v)] The function $\bm{u} \mapsto r(\bm{x}, \bm{u})$ is lower semicontinuous for each $\bm{x} \in \mathbb{R}^n$.

			\item[(vi)] $r:\bbr^n\times U\to \bbr^n$ is continuous and there exists a modulus $\omega_r:[0,+\infty)\to[0,+\infty)$ such that
			\[|r(\bm{x},\bm{u})-r(\bm{y},\bm{u})|\le \omega_r(|\bm{x}-\bm{y}|) \quad \forall \bm{x},\bm{y}\in\bbr^n.\]
 			
			\item[(vii)] $q:\bbr^n\to\bbr$ is continuous and there exists a modulus $\omega_q:[0,+\infty)\to[0,+\infty)$ such that
			\[|q(\bm{x})-q(\bm{y})|\le \omega_q(|\bm{x}-\bm{y}|) \quad \forall \bm{x},\bm{y}\in\bbr^n.\]
		\end{enumerate}
\end{assumption}
We note that all these conditions, except $(v)$, are  standard in the literature of HJB equations for optimal control (e.g.,~\cite{Bardi97}). The condition $(v)$ will be used to guarantee 
 the existence of  minimizers to the following optimization problem:
\[\min_{\bm{u}\in U} \left\{\bm{p}\cdot f(\bm{x},\bm{u})+r(\bm{x},\bm{u})\right\}\]
for each $\bm{x}, \bm{p} \in \mathbb{R}^n$ in our analysis of HJB equations.

To consider a maximum entropy variant of the optimal control problem, 
we now generalize the notion of controls by taking the \emph{relaxed control} approach. This approach was first introduced by Young \cite{young1942generalized,young1969lectures}, and then widely applied to calculus of variations~\cite{mcshane1967relaxed,warga1962relaxed}, deterministic optimal control~\cite{artstein1978relaxed,warga2014optimal,williamson1976relaxed} and stochastic optimal control~\cite{brzezniak2013optimal,fleming1980measure,haussmann1990existence}. 
Consider a function $\mu:[0,T] \to \mathcal{P}(U)$.
Given $A \subseteq U$, $\mu (t; A)$ is defined as the probability of $u(t)$ being contained in $A$, i.e.,
\[
\mu (t; A) := P (u(t) \in A), \quad A \subseteq U
\]
for each $t \in [0,T]$. 
The time-dependent probability measure $\mu(t; \cdot)$ can be interpreted as a relaxed version of the original control. 
Employing the relaxed control $\mu$, we consider the following modified version of the original dynamical system~\eqref{dyn1}:
\begin{equation}\label{dyn2}
\dot{x}(t) = \int_U f(x(t), \bm{u})\mu(t; \d \bm{u}).
\end{equation}
In words, the rate of changes in $x(t)$ is   the average of $f(x(t), \bm{u})$ with respect to the probability measure $\mu(t; \cdot)$. One may understand the dynamics \eqref{dyn2} as a generalization of its original counterpart \eqref{dyn1}. 
For any classical control $u:[0,T]\to U$, let the relaxed control action be 
the Diract delta measure concentrated at $u(t)$, i.e.,
\[\mu_0(t; \d \bm{u}) = \delta_{u(t)}(\d \bm{u}).\]
Then,  \eqref{dyn2} is reduced to the original dynamical system:
\[
\dot{x}(t)=\int_{{U}} f(x(t), \bm{u})\delta_{u(t)}(\d \bm{u}) = f(x(t),u(t)).
\]
Another interpretation of the relaxed control system, in terms of differential inclusions, can be found in Appendix~\ref{app:relax}.

We are now ready to define the maximum entropy optimal control problem.
As discussed in the introduction, 
we consider the cost functional as a weighted sum of \eqref{cost} and the entropy of the relaxed control $\mu$. 
Recall that the (differential) entropy of the measure $\mu(t; \cdot)\in\mathcal{P}(U)$ is defined as
\[\mathrm{H}(\mu(t,\cdot)):=\begin{cases}
-\int_{{U}}\frac{\d\mu}{\d \bm{u} }\log\frac{\d\mu}{\d \bm{u}}\,\d \bm{u}&\mbox{if }  \mu\ll \d \bm{u}  \\
-\infty  &\mbox{otherwise}.\end{cases}\]
Our new cost functional for maximum entropy optimal control is  defined as follows:
\begin{equation}\label{cost_soft}
J^\alpha_{\bm{x},t}(\mu):=\int_t^T\left(\int_{U}r(x(s), \bm{u})\,\mu(s; \d \bm{u})-\alpha \mathrm{H}(\mu(s; \cdot))\right)\,\d s+q(x(T)),\quad x(t)= \bm{x}.
\end{equation}
Here, the weight $\alpha \in \mathbb{R}$ is called the \emph{temperature parameter}.  
By minimizing this cost functional, we can find a high entropy-control that keeps the original cost sufficiently small. 
Given the initial condition $x(0) = \bm{x}$, 
the finite-horizon maximum entropy optimal control problem is  formulated as
\begin{equation}\label{opt_me}
\min_{\mu \in \mathcal{M}} \; J_{\bm{x}, 0}^\alpha (\mu).
\end{equation}

The set of admissible relaxed controls $\mathcal{M}$ must be carefully chosen taking into account the following conditions. 
First of all, for the system \eqref{dyn1} and the cost functional \eqref{cost_soft} to be well-defined, the probability measure $\mu(t;\d{\bm u})$ for each fixed time $t\in[0,T]$ needs to satisfy 
\[\int_U |f(\bm{x},\bm{u})|\mu(t;\d \bm{u})<+\infty,\quad\int_U |r(\bm{x},\bm{u})|\mu(t;\d \bm{u})<+\infty \quad \forall(t,\bm{x})\in [0,T]\times\bbr^n.\]
Furthermore, the solution to \eqref{dyn2} exists if 
there exist a function $c\in L^1([0,T]; \mathbb{R})$ and a modulus $\omega_f: [0,+\infty)^2 \to [0, +\infty)$ such that 
\begin{align*}
&\int_U |f(\bm{x},\bm{u})|\mu(t;\d {\bm u}) \d {\bm u}\le c(t)(1+|\bm{x}|) \quad \forall \bm{x} \in \mathbb{R}^n,\\
&\left|\int_U f(\bm{x},\bm{u})\mu(t;\d{\bm u})-\int_U f(\bm{y},\bm{u})\mu(t;\d{\bm u})\right|\le \omega_f(|\bm{x}-\bm{y}|,R) \quad \forall \bm{x},\bm{y}\in B(0,R) \mbox{ and }\forall R > 0.
\end{align*}
By  Assumption \ref{assumption} $(ii)$ and $(iii)$, the first condition is reduced to the condition that $ t\mapsto \int_U |\bm{u}|\mu(t;\d\bm{u})$ is integrable on $[0,T]$, while the second condition automatically holds. Lastly, the cost functional \eqref{cost_soft} does not blow up if the maps $t\mapsto \int_U r(x(t),\bm{u})\mu(t;\d\bm{u})$ and $t\mapsto \mathrm{H}(\mu(t;\cdot))$ are also integrable. Putting these together, we define the set of admissible  controls $\mathcal{M}$ as follows.

\begin{definition}
The set of admissible controls $\mathcal{M}$ is defined as a set of time-dependent measures $\mu:[0,T]\to\mathcal{P}(U)$ that satisfy the following conditions:
\begin{enumerate}
	\item For all $(t,\bm{x})\in[0,T]\times\bbr^n$, 
	\[\int_U |f(\bm{x},\bm{u})|\mu(t;\d \bm{u})<+\infty,\quad\int_U |r(\bm{x},\bm{u})|\mu(t;\d \bm{u})<+\infty.\]
	\item The  maps
	\[t\mapsto \int_U |\bm{u}|\mu(t;\d {\bm u}),\quad t\mapsto \int_U r(x(t),\bm{u})\mu(t;\d\bm{u}),\quad t\mapsto \mathrm{H}(\mu(t;\cdot))\]
	are integrable over $[0,T]$.
\end{enumerate}
\end{definition}

Before studying the existence of an optimal solution, 
we introduce another interpretation of the maximum entropy formulation below. 
  The cost functional \eqref{cost_soft} can  be understood from a different perspective using the Kullback--Leibler divergence or the relative entropy. 
  Recall that, given a separable metric space $X$ and  two probability measures $\mu$ and $\nu$ on $X$, the Kullback--Leibler(KL) divergence from $\nu$ to $\mu$ is defined as
	\[\kl (\mu\|\nu):=\begin{cases}
	\displaystyle\int_X \frac{\d\mu}{\d\nu}\log \frac{\d\mu}{\d\nu} \d\nu=\int_X \log \frac{\d\mu}{\d \nu}\d \mu &\mbox{if}\quad \mu\ll \nu\\
	+\infty & \mbox{otherwise}.
	\end{cases}\]
	When $U$ is compact, the entropy  $\mathrm{H}(\mu)$ can be expressed in terms of KL divergence associated with the uniform probability distribution $\mathscr{U}(\d \bm{u}) = \frac{\d\bm{u}}{|U|}$ as follows:
	\begin{align*}
	\mathrm{H}(\mu) &= -\int_{U} \frac{\d\mu}{\d\bm{u}}\log \frac{\d\mu}{\d\bm{u}} \d {\bm u} = -\int_U \log\frac{\d\mu}{|U|\mathscr{U}(\d {\bm u})}\d\mu=-\int_U \left(\log\frac{\d \mu}{\mathscr{U}(\d \bm{u})}-\log|U|\right)\d\mu\\
	&=-\int_U \log\frac{\d \mu}{\mathscr{U}(\d {\bm u})}\d\mu +\log |U|=-\kl (\mu\|\mathscr{U})+\log|U|.
	\end{align*}
  Therefore, the cost functional \eqref{cost_soft} can be rewritten as
	\begin{equation}\label{cost_kl}
	J^\alpha_{\bm{x}, t}(\mu):=\int_t^T\left(\int_{U}r(x(s), \bm{u})\,\mu(s; \d \bm{u})+\alpha \kl (\mu(s; \cdot)\|\mathscr{U})\right)\,\d s+q(x(T))-\alpha(T-t)\log|U|.
	\end{equation}
	Hence, minimizing the cost functional is equivalent to minimizing a weighted sum of the original cost and the KL divergence from the uniform probability measure $\mathscr{U}$ to the relaxed control $\mu$. 
	In other words, the problem is to find a relaxed control $\mu$ that keeps the expected cumulative cost sufficiently small and is not too far from the uniform distribution. 
We can reinterpret the quantity
	\[\int_{U}r(x(s),\bm{u})\,\mu(s; \d \bm{u})+\alpha \kl (\mu(s; \cdot)\|\mathscr{U})\]
	as the sum of the running cost and the cost of choosing the control $\mu$, where the uniform distribution can be  understood as a passive transition, since the uniform distribution is a neutral control in the absence of any information.  
In a series of research studies~\cite{todorov2007linearly, todorov2009efficient, Dvijotham2011}, a similar KL control cost has been considered in discrete-time MDPs, allowing the transition probabilities to be fully controlled.  
For this class of MDPs, the Bellman equation is linear and thus efficiently solvable. 
This result has been extended to online MDPs~\cite{Guan2014} and ODE methods for MDPs~\cite{Busic2018}.
Moreover, as observed in~\cite{Todorov2008, Theodorou2012},
the continuous-time stochastic control problems that can be efficiently solved using \emph{path integrals} are  a special case of these MDPs with KL control costs.
The path integral control problems consider control-affine systems with a particular covariance condition to reformulate the resulting HJB equations as linear~\cite{Kappen2005, Kappen2005b, Theodorou2010}.\footnote{Interestingly, this class of linearly solvable MDPs turns out to be considered as an inference problem~\cite{Kappen2012}.
The duality between discrete-time stochastic control and inference has been further generalized and used to devise a convergent posterior policy iteration  algorithm~\cite{Rawlik2013}.  }
It is worth emphasizing that our problem setting is more general than the linearly solvable MDPs in the sense that we consider fully nonlinear systems and cost functions. 
However, \cite{todorov2007linearly, todorov2009efficient, Dvijotham2011} consider MDPs with fully controlled transition probabilities and cost functions in a particular form. 
Similarly, path integral control uses control-affine systems and cost functions quadratic in $\bm{u}$ under a special condition on covariance matrices.\footnote{The lifting technique in~\cite{Yang2014} can be used to handle a slightly more general class of cost functions.} 
Since our problem formulation does not assume such particular structures, our HJB equation for maximum entropy control is \emph{nonlinear}, unlike theirs. 
Nevertheless, we will show in Section~\ref{sec:4} that our HJB equation is more tractable to solve compared to its standard counterpart.

\begin{remark}\label{rem:den}
The condition $\mathrm{H}(\mu(t; \cdot))\in L^1([0,T])$ in the set of admissible controls
 implies that $\mathrm{H}(\mu(t; \cdot))<+\infty$ for a.e. $t\in[0,T]$. Therefore, $\mu(t; \d \bm{u}) \ll \d{\bm u} $ for a.e. $t\in[0,T]$,
 and  by the Radon-Nikodym theorem, there exists a measurable function $g:[0,T]\times U\to [0,+\infty)$
 such that 
  $\mu(t; \d \bm{u}) = g (t,\bm{u})\d \bm{u}$ a.e $t\in[0,T]$. Moreover, since $\mu(t;\cdot)$ is a probability measure defined on $U$, we directly have $\int_U g(t,\bm{u})\d\bm{u}=1$ for a.e. $t\in [0,T]$, which implies that $g$ is a probability density function on $U$. Thus, in the remainder of the paper, we interchangeably use relaxed control $\mu(t,\d \bm{u})$ and  its probability density $g(t,\bm{u})$. 
We may reformulate the maximum entropy optimal control problem in terms of the density $g$.  Let us define the differential entropy of the density function $g$ as
  \[\mathrm{H}(g(t,\cdot)) = -\int_U g(t,\bm{u})\log g(t,\bm{u})\d\bm{u}.\]
We also introduce the following density version of the set of admissible controls: 
  \[\mathcal{G} := \left\{g:[0,T]\to L^1_{+,1}(U) ~\mid~ g\d\bm{u}\in\mathcal{M}\right\},\]
  where $L^1_{+,1}(U)$ denotes the set of nonnegative integrable function on $U$ whose integration is 1:
  \[L^1_{+,1}(U):=\left\{\phi:U\to\bbr_+~:~\int_{U}\phi(u)\d u=1\right\}.\] 
  Then, the cost functional \eqref{cost_soft} can be rewritten using the density $g$ as
  \begin{equation}\label{cost_soft_2}
  J^\alpha_{\bm{x},t}(g):=\int_t^T\left(\int_{U}r(x(s), \bm{u})\,g(s,\bm{u})\,\d\bm{u}-\alpha \mathrm{H}(g(s, \cdot))\right)\,\d s+q(x(T)),\quad x(t)= \bm{x},
  \end{equation}
  and the finite-horizon maximum entropy optimal control problem can be expressed as
  \begin{equation}\label{opt_me_2}
  \min_{g\in\mathcal{G}} J_{\bm{x},0}^\alpha(g).
  \end{equation}
  From Section \ref{sec:3}, we mainly use the density version \eqref{opt_me_2} of the maximum entropy optimal control problem, which is equivalent to the original version using measures \eqref{opt_me}.
  \end{remark}

\subsection{Existence of Optimal Controls}\label{sec:existence}

One of the most fundamental questions on the maximum entropy control problem is the existence of optimal controls or, equivalently, the minimizers of \eqref{opt_me}. We first study whether there exists a control $\mu^\star\in\mathcal{M}$ that achieves the infimum of the cost functional.

\begin{theorem}\label{thm:exist_finite}  Suppose that Assumption~\ref{assumption} holds. Moreover, we assume that the control set $U$ is compact and $f,r$ and $q$ are Lipschitz continuous in $\bm{x}$.
Then, for each $(\bm{x}, t) \in \mathbb{R}^n \times [0,T]$,	 there exists $\mu^\star\in\mathcal{M}$ such that
	\[J^{\alpha}_{\bm{x}, t} (\mu^\star) = \inf_{\mu\in\mathcal{M}_t}J^{\alpha}_{\bm{x}, t} (\mu),\]
	where $\mathcal{M}_t:=\left\{\mu|_{[t,T]}:[t,T]\to\mathcal{P}(U) \mid \mu\in \mathcal{M}  \right\}$. 
\end{theorem}
Setting $t = 0$ in the theorem implies the existence of an optimal control.
\begin{proof}
Fix $(\bm{x}, t) \in \mathbb{R}^n \times [0,T]$.
Let $\{\mu_i\}_{i=1}^\infty\subset \mathcal{M}_t$ be a sequence of admissible policies such that
\[\lim_{i\to\infty}J^{\alpha}_{\bm{x}, t}(\mu_i)=\inf_{\mu\in\mathcal{M}_t}J^{\alpha}_{\bm{x}, t} (\mu).
\]
Thus, the costs $J^{\alpha}_{\bm{x}, t}(\mu_i)$ are bounded.
 By using the reformulation \eqref{cost_kl}, we observe that
\[\left|\int_t^T\left(\int_{U}r(x_i(s), \bm{u})\,\mu_i (s; \d \bm{u}) +\alpha \kl(\mu_i(s; \cdot)||\mathscr{U})\right)\,\d s+q(x_i(T))-\alpha(T-t)\log|U|\right|<C\]
for some constant $C$ independent of $i$,
where $x_i$ denotes the system state of \eqref{dyn2} when the control $\mu_i$ is employed. Under Assumption~\ref{assumption}, the state $x_i(s)$ is bounded by a constant $C_0$ depending on $f$ and $T$ by Lemma \ref{stability}. Therefore, there exist  constants $C_r$ and $C_q$ such that
\[|r(x_i(s),\bm{u})|<C_r,\quad |q(x_i(T))|<C_q \quad \forall \bm{u}\in U,\quad i=1,2\ldots,\quad t\le s\le T.\]
Then, we can uniformly bound the integral of the KL divergence term as follows:
\[\left|\int_t^T\kl (\mu_i(s; \cdot)||\mathscr{U})\,\d s\right|< \frac{1}{\alpha}\left[ C+(T-t)C_r +C_q+\alpha(T-t)|\log|U||\right ].\]
On the other hand, for each $\mu_i$, we define a probability measure $\nu_i$ on $[t,T]\times U$ as
\[
 \nu_i(\d s,\d \bm{u}):= \frac{1}{T-t}\mu_i(s; \d \bm{u})\d s,
 \]
i.e., the probability measure $\nu_i$ has a density $\frac{1}{T-t}\mu_i(s;\d\bm{u})$ with respect to $s$, and also consider the probability measure $\mathscr{U}_{[t,T]}(\d s,\d \bm{u}):=\frac{1}{(T-t)}\d s \d\bm{u}$. Then, the KL-divergence from $\mathscr{U}_{[t,T]}$ to $\nu_i$ is uniformly bounded by a constant, independent of $i$, as follows:
\begin{align*}
\left|\kl (\nu_i||\mathscr{U}_{[t,T]})\right|&=\left|\int_t^T\int_U \frac{\frac{1}{T-t}\mu_i(s;\d \bm{u})}{\frac{1}{(T-t)|U|}\d\bm{u}}\log\left(\frac{\frac{1}{T-t}\mu_i(s;\d \bm{u})}{\frac{1}{(T-t)|U|}\d\bm{u}}\right)\frac{1}{(T-t)}\d\bm{u} \d s\right|\\
&=\left|\int_t^T\int_U \frac{\mu_i(s;\d \bm{u})}{\frac{1}{|U|}\d\bm{u}}\log\left(\frac{\mu_i(s; \d \bm{u})}{\frac{1}{|U|}\d\bm{u}}\right)\frac{1}{(T-t)}\d\bm{u} \d s\right|\\
&\le\frac{1}{T-t}\left|\int_t^T \kl (\mu_i(s; \cdot)||\mathscr{U})\d s\right|\\
&\le \frac{1}{\alpha (T-t)}(C+(T-t) C_r+C_q+\alpha(T-t)\log|U|)=:M.
\end{align*}

It follows from Lemma \ref{LB.3} that   $\{\nu \in \mathcal{P}([t,T]\times U) \mid \kl(\nu||\mathscr{U}_{[t,T]})\le M\}$ is a compact subset of $\mathcal{P}([t,T]\times U)$, and therefore there exists a subsequence $\{\nu_{i_k}\}_{k=1}^\infty$ such that $\nu_{i_k}\xrightharpoonup{*} \nu\in \mathcal{P}([t,T]\times {U})$. Then, by Lemma \ref{LB.4}, there exists a family of probability measures $\{\rho (s; \d \bm{u})\}_{s\in[t,T]}$, $\rho (s; \cdot) \in \mathcal{P} (U)$, such that  for any measurable function $\phi:[t,T]\times U\to \bbr$,
\[\int_{[t,T]\times {U}} \phi(s, \bm{u})\,\nu(\d s,\d \bm{u}) = \int_t^T\left(\int_{{U}}\phi(s, \bm{u})\rho (s; \d \bm{u})\right)\nu(\pi_1^{-1}(\d s)),\]
where $\pi_1:[t,T]\times U\to[t,T]$ is the projection with respect to the first argument, i.e., $\pi_1(s, \bm{u}) = s$. Since the marginal  of $\nu_i$ on $[t,T]$ is identical to $\frac{1}{T-t}\d s$, the marginal of the weak-$*$ limit $\nu$ should also be the uniform probability measure, i.e., $\nu(\pi_1^{-1}(\d s)) = \frac{1}{T-t} \d s$. Therefore, we have
\[\int_{[t,T]\times {U}} \phi(s, \bm{u})\nu(\d s,\d \bm{u}) = \frac{1}{T-t}\int_t^T\left(\int_{{U}}\phi(s, \bm{u}) \rho (s; \d \bm{u})\right)\d s,\]
which implies that $\frac{1}{T-t}\rho(s;\d u)$ is the density of $\nu$ with respect to the $s$-variable. Therefore, we can express $\nu$  as $\nu(\d s,\d \bm{u}) =\frac{1}{T-t} \rho(s ;\d \bm{u})\,\d s$. Moreover, since $\kl (\nu||\mathscr{U}_{[t,T]})\le M$, we have
\begin{align*}
\int_t^T |\mathrm{H}(\rho (s; \cdot))|\d s &\le \int_t^T \kl (\rho (s; \cdot)||\mathscr{U})\,\d s +(T-t)|\log|U|| \\
&=(T-t) \kl (\nu ||\mathscr{U}_{[t,T]})+(T-t)|\log|U||<+\infty,
\end{align*}
which implies $\rho \in \mathcal{M}_t$. Finally, we show that $\rho$ indeed minimizes $J_{\bm{x}, t}^\alpha$ over $\mathcal{M}_t$. It follows from the choice of $\mu_i$ that
\begin{align*}
&\inf_{\mu\in\mathcal{M}_t}J^{\alpha}_{\bm{x},t} (\mu)\\
&=\liminf_{i\to \infty}\left\{\int_t^T \left(\int_{U} r(x_i(s), \bm{u})\,\mu_i(s; \d \bm{u}) +\alpha \kl (\mu_i(s;\cdot)||\mathscr{U})\right) \d s +q(x_i(T))-\alpha(T-t)\log|U|\right\}\\
&\ge\liminf_{i\to\infty}\left\{\int_t^T \int_{U} r(x_i(s), \bm{u})\mu_i(s;\d \bm{u})\d t\right\} + \alpha\liminf_{i\to\infty} \int_t^T \kl (\mu_i(s;\cdot)||\mathscr{U}) \d s\\
&\quad +\liminf_{i\to\infty} q(x_i(T))-\alpha(T-t)\log|U|\\
&=\int_t^T \int_{U} r(x(s),\bm{u}) \rho(s;\d \bm{u} )\d s+q(x(T))+\alpha(T-t)\liminf_{i\to\infty} \kl (\nu_i||\mathscr{U}_{[t,T]})-\alpha(T-t)\log|U|,
\end{align*}
where the last equality comes from the weak-$*$ convergence of $\nu_i$ to $\nu$ and 
\[\int_t^T \kl(\mu_i(s;\cdot)||\mathscr{U})\,\d s = (T-t)\kl(\nu_i||\mathscr{U}_{[t,T]}).\]
By the lower semicontinuity of the KL divergence (Lemma \ref{LA.2}), we have
\[\liminf_{i\to\infty} \kl (\nu_i||\mathscr{U}_{[t,T]})\ge \kl (\nu||\mathscr{U}_{[t,T]}).\]
Therefore, we conclude that
\begin{align*}
\inf_{\mu\in\mathcal{M}_t}J^{\alpha}_{\bm{x}, t}(\mu) &\ge \int_t^T \int_{U} r(x(s),\bm{u})\rho(s; \d \bm{u})\d s+q(x(T))+\alpha (T-t)\kl (\nu ||\mathscr{U}_{[t,T]})-\alpha(T-t)\log|U|  \\
&=\int_t^T\left(\int_U r(x(s),\bm{u})\rho(s;\d u)+\alpha \kl(\rho(s;\cdot)||\mathscr{U})\right)\d s+q(x(T))-\alpha(T-t)\log|U|\\
&=J^\alpha_{\bm{x}, t}(\rho).
\end{align*}
This implies that $\rho$ is a minimizer of $J_{\bm{x}, t}^\alpha$ over $\mathcal{M}_t$.
\end{proof}

\subsection{Discrete-Time Approximation}\label{sec:dt}

Before introducing an HJB-based method for solving the maximum entropy control problem, we discuss practical issues in implementing relaxed controls to continuous-time dynamical systems. 
There are two major issues. 
First, it is unclear how to use a probability distribution as a control action. 
Second, in practical systems, it may be infeasible to continuously exert the control actions. 
As a means of addressing these practical issues, 
we introduce the following discrete-time stochastic system with sampling interval $\Delta t$:
\begin{equation}\label{sys_dt}
x_{k+1} = x_k + (\Delta t)^2 f(x_k, u_k),\quad u_k\sim \mu(k(\Delta t)^2,\cdot),
\end{equation}
where $\Delta t$ is set to be $\frac{1}{N}$ for some fixed natural number $N$ for convenience. We show that the state of this discrete-time system converges to that of the original relaxed control system, as $\Delta t$ tends to zero.

\begin{proposition}
Suppose that Assumption~\ref{assumption} holds.
We further assume that the control set $U$ is compact, $f$ is Lipschitz continuous with respect to space variable and $\mu:[0,T]\to\mathcal{P}(U)$ is continuous in time in the sense that for any continuous bounded function $\phi:U\to \bbr$,
	\begin{equation}\label{close}
	\left|\int_U \phi(\bm{u}) \mu(t;\d \bm{u})-\int_U \phi(\bm{u})\mu(s; \d \bm{u})\right|\to 0 \quad\mbox{as } |t-s|\to0.
	\end{equation}
	Then, for any natural number $K$, we have
	\[|x(T)-x_{K}|\le |x(0)-x_0|+o(1) \mbox{ a.s.} \quad \mbox{as} \quad\Delta t\to 0,\]
	where $T = K(\Delta t)^2$.
	In particular, if $x(0) = x_0$, then
	\[|x(T)-x_K|=o(1) \mbox{ a.s.} \quad \mbox{as} \quad  \Delta t\to 0.
	\]
\end{proposition}

\begin{proof}
Fix an arbitrary natural number $N$ and let $\Delta t :=\frac{1}{N}$.
We compare the two states at time $\Delta t=N(\Delta t)^2$ as 
\begin{align*}
|x(\Delta t) - x_N| &= \left|x(0)+\int_0^{\Delta t} \int_{U}f(x(s), \bm{u})\,\mu(s; \d \bm{u})\d s-x_0-(\Delta t)^2\sum_{k=0}^{N-1} f(x_k,u_k)\right|\\
&\le \left|x(0)+\int_0^{\Delta t} \int_{U}f(x(s),\bm{u})\,\mu(s; \d \bm{u})\d s-x(0)-\Delta t\int_{U}f(x(0),\bm{u})\,\mu(0; \d \bm{u}) \right|\\
&\quad + \left|x(0)+\Delta t\int_{U}f(x(0),\bm{u})\,\mu(0; \d \bm{u})-x_0-(\Delta t)^2\sum_{k=0}^{N-1} f(x_k, u_k)\right|\\
&\le |x(0)-x_0| +\Delta t\underbrace{\left| \int_{U}f(x(0), \bm{u})\,\mu(0;\d \bm{u}) - \frac{1}{N}\sum_{k=0}^{N-1}f(x_k, u_k)\right|}_{=:\mathcal{I}}+o(\Delta t).
\end{align*}
We now choose i.i.d. random variables $v_k$ whose distribution follow $\mu(0; \cdot)$ and then further estimate the term $\mathcal{I}$ as follows:
\begin{align*}
	\mathcal{I}&\le\left|\int_U f(x(0), \bm{u})\mu(0; \d \bm{u}) -\frac{1}{N}\sum_{k=0}^{N-1}f(x_0, v_k)\right|+\left|\frac{1}{N}\sum_{k=0}^{N-1}f(x_0,v_k)-\frac{1}{N}\sum_{k=0}^{N-1}f(x_k,v_k)\right|\\
	&\quad+\left|\frac{1}{N}\sum_{k=0}^{N-1}f(x_k,v_k)-\frac{1}{N}\sum_{k=0}^{N-1}f(x_k, u_k)\right|=o(1) +\left|\frac{1}{N}\sum_{k=0}^{N-1}(f(x_k,v_k)-f(x_k, u_k))\right|,
\end{align*}
where we used the law of large numbers, Lipschitz continuity of $f$ and the fact that $|x_k-x_0|=o(1)$. We now let $Z_k:=f(x_k,v_k)-f(x_k,u_k)$. Then, since $f$ is continuous and the control set $U$ is compact, it is easy to see that the variance  $\mathrm{Var}[Z_k]<C$ for some constant $C$ independent of $k$.
Therefore, by Kolmogorov's strong law of large numbers \cite{feller2008introduction}, we have
\[
\left(\frac{1}{N}\sum_{k=0}^{N-1}Z_k-\frac{1}{N}\sum_{k=0}^{N-1}\mathbb{E}[ Z_k] \right)\to 0\quad\mbox{a.s.}
\]
On the other hand, since  $\mu(0; \cdot)$ and $\mu(k(\Delta t)^2; \cdot)$ are close each other in the sense of \eqref{close}, we have $\mathbb{E}[Z_k] \to0$ as $\Delta t=\frac{1}{N}\to0$. Thus, we arrive at the following convergence result:
\[
\frac{1}{N}\sum_{k=0}^{N-1}Z_k \to 0\quad \mbox{a.s.}
\]
This implies that $\mathcal{I}\le o(1)$, and consequently
\[|x(\Delta t)-x_N|\le|x(0)-x_0|+o(\Delta t) \mbox{ a.s.} \quad \mbox{as $\Delta t\to0$.}\]
Hence, until the finite time $T = K(\Delta t)^2$, we have
\[|x(T)-x_{K}|=|x(T)-x_{TN^2}|\le |x(0)-x_0|+TN \times o(\Delta t)=|x(0)-x_0|+o(1) \mbox{ a.s.} \quad \mbox{as $\Delta t\to0$}.
\]
\end{proof}

To numerically demonstrate the performance and the utililty of our maximum entropy methods, we use the discrete-time approximation~\eqref{sys_dt} in Section~\ref{sec:6}.

\section{Soft Hamilton--Jacobi--Bellman Equations}\label{sec:3}

In this section, we derive the HJB equations for the maximum entropy control problems. 
We then show that the optimal value function, defined by
\[
V_\alpha(t,\bm{x}):=\inf_{g\in \mathcal{G}} J^\alpha_{\bm{x},t}(g),
\]
 corresponds to the unique viscosity solution of the HJB equation. 
We further study some properties of the HJB equations.

\subsection{Dynamic Programming and Soft HJB Equations}

To begin, we apply the dynamic programming principle to the maximum entropy control problem to obtain the following equality, which will be used in deriving the HJB equation.
\begin{lemma}\label{L3.1}
	Suppose that Assumption \ref{assumption} holds. Then, the value function $V_\alpha$ satisfies the following equality:
	\begin{equation}\label{dpp}
	V_\alpha(t,\bm{x})=\inf_{g\in \mathcal{G}} \left\{\int_t^{t+h}\left(\int_{U}r(x(s),\bm{u})g(s,\bm{u})\,\d\bm{u}-\alpha \mathrm{H}(g(s,\cdot))\right)\,\d s+ V_\alpha(t+h,x(t+h))\right\},
	\end{equation}
	where $x(s)$ is the solution to \eqref{dyn2} with control $g\in \mathcal{G}$ and  initial condition $x(t)=\bm{x}$.	
\end{lemma}
\begin{proof}
	Although the proof is almost identical to the standard optimal control case~\cite[Proposition 3.2, Section 3]{Bardi97}, we provide the full proof for the completeness of the paper. 
	Fix an arbitrary $g\in \mathcal{G}$. It follows from the definition of $J^\alpha_{\bm{x},t}(g)$ that
	\begin{align*}
	J^\alpha_{{\bm x},t}(g) &= \int_t^T\left(\int_{U}r(x(s),\bm{u})g(s,\bm{u})\,\d\bm{u}-\alpha \mathrm{H}(g(s,\cdot))\right)\,\d s+q(x(T))\\
	&= \int_t^{t+h}\left(\int_{U}r(x(s),\bm{u})g(s,\bm{u})\,\d\bm{u}-\alpha \mathrm{H}(g(s,\cdot))\right)\,\d s\\
	&\quad +\int_{t+h}^T\left(\int_{U}r(x(s),\bm{u})g(s,\bm{u})\,\d\bm{u}-\alpha \mathrm{H}(g(s,\cdot))\right)\,\d s+q(x(T))\\
	&\ge \int_t^{t+h}\left(\int_{U}r(x(s),\bm{u})g(s,\bm{u})\,\d\bm{u}-\alpha \mathrm{H}(g(s,\cdot))\right)\,\d s + V_\alpha(t+h,x(t+h)).
	\end{align*}
	Now, taking infimum with respect to $g\in\mathcal{G}$ on both sides yields
	\[V_\alpha(t,x) \ge \inf_{g\in \mathcal{G}}\left\{\int_t^{t+h}\left(\int_{U}r(x(s),\bm{u})g(s,\bm{u})\,\d\bm{u}-\alpha \mathrm{H}(g(s,\cdot))\right)\,\d s+ V_\alpha(t+h,x(t+h))\right\}.\]
	To obtain the reverse direction of the inequality, we fix an arbitrary $g\in \mathcal{G}$ and $\e>0$ and let $x(s)$ be the solution to \eqref{dyn2} with  control $g$ for $t\le s\le t+h$. Choose a control $g'\in \mathcal{G}$ satisfying
	\[V_\alpha(t+h,x(t+h))\ge J^\alpha_{x(t+h),t+h}(g')-\e.\]
	We construct another control $\tilde{g}\in \mathcal{G}$  as
	\[\tilde{g}(s)=\begin{cases} g(s) & \mbox{for } t\le s< t+h,\\
	g'(s) & \mbox{for } t+h\le s\le T. \end{cases}\]
	We define $\tilde{x}(s)$ be a solution to \eqref{dyn2} with the control $\tilde{g}$ for the time interval $t\le s\le T$. In particular, $\tilde{x}(s) = x(s)$ for $t\le s\le t+h$. Then, we have 
	\begin{align*}
	V_{\alpha}(t,x)\le J^{\alpha}_{\bm{x},t}(\tilde{g})&= \int_t^{t+h}\left(\int_{U}r(x(s),\bm{u})g(s,\bm{u})\,\d\bm{u}-\alpha \mathrm{H}(g(s,\cdot))\right)\,\d s\\
	&\quad +\int_{t+h}^T\left(\int_{U}r(\tilde{x}(s),\bm{u})g'(s,\bm{u})\,\d\bm{u}-\alpha \mathrm{H}(g'(s,\cdot))\right)\,\d s+q(\tilde{x}(T))\\ 
	&= \int_{t}^{t+h}\left(\int_{U}r(x(s),\bm{u})g(s,\bm{u})\,\d\bm{u}-\alpha \mathrm{H}(g(s,\cdot))\right)\,\d s+J^\alpha_{x(t+h),t+h}(g')\\
	&\le \int_{t}^{t+h}\left(\int_{U}r(x(s),\bm{u})g(s,\bm{u})\,\d\bm{u}-\alpha \mathrm{H}(g(s,\cdot))\right)\,\d s+V_{\alpha}(t+h,x(t+h))+\e. 
	\end{align*}
	We now take an infimum over $g\in \mathcal{G}$ to obtain
	\[V_\alpha(t,x) \le \inf_{g\in \mathcal{G}}\left\{\int_t^{t+h}\left(\int_{U}r(x(s),\bm{u})g(s,\bm{u})\,\d\bm{u}-\alpha \mathrm{H}(g(s,\cdot))\right)\,\d s+ V_\alpha(t+h,x(t+h))\right\}+\e.\]
	Since $\e$ was arbitrarily chosen, the result follows.
\end{proof}

To formally derive the HJB equation that the value function $V_\alpha$ should satisfy, we now temporally assume that  $V_\alpha$ is smooth. 
This assumption will be relaxed in the next subsection by using the viscosity solution framework. 
Rearranging \eqref{dpp}, dividing it by $h$ and letting $h\to 0$, we formally have
\begin{align}
\begin{aligned}\label{HJB}
0&=\lim_{h\to0}\inf_{g\in \mathcal{G}}\left\{\frac{V_\alpha(t+h,x(t+h))-V_\alpha(t,\bm{x})}{h}+\frac{1}{h}\int_t^{t+h}\left(\int_U r(x(s),\bm{u})g(s,\bm{u})\d\bm{u}-\alpha \mathrm{H}(g(s,\cdot))\right)\,\d s\right\}\\
&= \inf_{g\in \mathcal{G}}\left\{\partial_t V_\alpha + \nabla_{\bm{x}} V_\alpha \cdot \dot{x}(t) +\int_U r(\bm{x},\bm{u})g(t,\bm{u})\d\bm{u}-\alpha \mathrm{H}(g(t,\cdot))\right\}\\
&=\partial_t V_\alpha +\inf_{g\in\mathcal{G}} \left\{\nabla_{\bm{x}} V_\alpha \cdot \left(\int_U f(\bm{x},\bm{u})g(t,\bm{u})\d\bm{u}\right)+\int_U r(\bm{x},\bm{u})g(t,\bm{u})\d\bm{u}-\alpha \mathrm{H}(g(t,\cdot))\right\}\\
&=\partial_t V_\alpha(t,\bm{x})+\inf_{g\in \mathcal{G}}\left\{\int_{U}\left(\nabla_{\bm{x}}V_\alpha(t,\bm{x})\cdot f(\bm{x},\bm{u})+r(\bm{x},\bm{u})\right)g(t,\bm{u})\,\d\bm{u}-\alpha \mathrm{H}(g(t,\cdot))\right\}.
\end{aligned}
\end{align}
We can further simplify the HJB equation in a more explicit form,
using the entropy term $\mathrm{H}(g(t,\cdot))$. 
Let
\[L(\bm{x}, \bm{p}, \bm{u}):=   \bm{p} \cdot f(\bm{x},\bm{u})+r(\bm{x},\bm{u}),\]
where $\bm{p} \in \mathbb{R}^n$.
Then, the minimization problem in \eqref{HJB} is of the form
\begin{equation}\label{minimization}
\inf_{g\in \mathcal{G}} \left\{\int_{U} L(\bm{x}, \bm{p}, \bm{u})g(t,\bm{u})\d\bm{u}-\alpha \mathrm{H}(g(t,\cdot))\right\}.
\end{equation}
We show that it admits a closed-form optimal solution, which is in the form of Boltzmann distributions. 
\begin{lemma}\label{L3.2}
	Suppose that $U$ is compact. Then, the unique optimal solution of the minimization problem \eqref{minimization} is given by
	\[g_\alpha^*(\bm{x},\bm{p},\bm{u}):=\frac{\exp(-\frac{1}{\alpha} L(\bm{x}, \bm{p}, \bm{u}))}{\int_{U}\exp\left(-\frac{1}{\alpha} L(\bm{x},\bm{p},\bm{u})\right)\,\d\bm{u}}=\frac{\exp(-\frac{1}{\alpha}(\bm{p}\cdot f(\bm{x},\bm{u})+r(\bm{x},\bm{u})))}{\int_{U}\exp\left(-\frac{1}{\alpha}(\bm{p} \cdot f(\bm{x},\bm{u})+r(\bm{x},\bm{u}))\right)\,\d\bm{u}}.\]
Furthermore, the optimal value of \eqref{minimization} is 
	\[\inf_{g\in \mathcal{G}} \left\{\int_{U} L(\bm{x},\bm{p},\bm{u})g(t,\bm{u})\d\bm{u}-\alpha \mathrm{H}(g(t,\cdot))\right\}=-\alpha\log\int_U \exp\left(-\frac{1}{\alpha} L(\bm{x},\bm{p},\bm{u})\right) \d\bm{u}.\]
\end{lemma}
\begin{proof}
	The optimization problem~\eqref{minimization} can be reformulated as
	\begin{align*}
	\inf_{g\in\mathcal{G}} &\left\{\int_{U}L(\bm{x},\bm{p},\bm{u})g(t,\bm{u})\d\bm{u}-\alpha \mathrm{H}(g(t,\cdot))\right\}\\
	&=\alpha\inf_{g\in\mathcal{G}}\left\{\int_{U} \frac{L(\bm{x},\bm{p},\bm{u})}{\alpha}g(t,\bm{u})\d\bm{u}-\mathrm{H}(g(t,\cdot))\right\}\\
	&=\alpha\inf_{g\in\mathcal{G}}\left\{\int_{U} \frac{L(\bm{x},\bm{p},\bm{u})}{\alpha} g(t,\bm{u})\d\bm{u}+\kl(g(t,\cdot)\d\bm{u}||\mathscr{U})\right\}-\alpha\log|U|,
	\end{align*}
	where $\mathscr{U}$ denotes the uniform probability measure defined by $\mathscr{U}(\d \bm{u}) = \frac{\d\bm{u}}{|U|}$.
Using Lemma \ref{LA.1} with $\phi = \frac{L(\bm{x},\bm{p},\bm{u})}{\alpha}$, $\mu = g(t,\bm{u})\d\bm{u}$ and  $\gamma = \mathscr{U}$, we conclude that the unique optimal solution $g_\alpha^*$ of the minimization problem above is given by
	\[g_\alpha^*(\bm{x},\bm{p},\bm{u}):=\frac{\exp(-\frac{1}{\alpha} L(\bm{x},\bm{p},\bm{u}))}{\int_{U}\exp\left(-\frac{1}{\alpha} L(\bm{x},\bm{p},\bm{u})\right)\,\d\bm{u}},\]
	and the corresponding optimal value  is obtained as
	\[-\alpha\log\int_U \exp\left(-\frac{1}{\alpha}L( \bm{x},\bm{p},\bm{u})\right)\mathscr{U}(\d\bm{u})-\alpha\log|U|=-\alpha\log\int_U \exp\left(-\frac{1}{\alpha}L(\bm{x},\bm{p},\bm{u})\right)\d\bm{u}.\]
\end{proof}

By Lemma~\ref{L3.2}, we can substitute the minimizer $g_\alpha^*$ into the HJB equation \eqref{HJB} to obtain
\begin{align*}
0&=\partial_t V_\alpha(t,\bm{x})+\left(\int_{U} L( \bm{x},\nabla_{\bm{x}}V_\alpha,\bm{u})g_\alpha^*(t,u)\,\d\bm{u}-\alpha \mathrm{H}(g_\alpha^*(t,\cdot))\right)\\
&=\partial_t V_\alpha(t,\bm{x})-\alpha \log\int_{U}\exp\left(-\frac{1}{\alpha}L( \bm{x},\nabla_{\bm{x}}V_\alpha,\bm{u})\right)\,\d\bm{u}\\
&=\partial_t V_\alpha(t,\bm{x})-\alpha \log \int_{U}\exp \left(-\frac{\nabla_{\bm{x}}V_\alpha\cdot f(\bm{x},\bm{u})+r(\bm{x},\bm{u})}{\alpha}\right)\,\d\bm{u}.
\end{align*}
Thus, the HJB equation for maximum entropy control can be written as
\begin{equation}\label{HJB_soft}
\left \{
\begin{array}{ll}
\partial_t V_\alpha(t,\bm{x})-H_\alpha(\bm{x},\nabla_{\bm{x}}V_\alpha)=0 &  \mbox{ on } (0, T) \times \mathbb{R}^n\\
V_\alpha(T,\bm{x}) = q(\bm{x}) & \mbox{ in } \{t = T\} \times \mathbb{R}^n, 
\end{array}
\right.
\end{equation}
where the Hamiltonian is given by
\begin{equation}\label{Ham_soft}
H_\alpha(\bm{x},\bm{p}):=\alpha\log \int_{U}\exp\left(-\frac{\bm{p}\cdot f(\bm{x},\bm{u})+r(\bm{x},\bm{u})}{\alpha}\right)\,\d\bm{u}.
\end{equation}
\begin{remark}
It is worth comparing this Hamiltonian with its standard counterpart 
\[
H_0(\bm{x}, \bm{p}) := \max_{\bm{u} \in U} \{ -\bm{p} \cdot f(\bm{x}, \bm{u}) - r(\bm{x}, \bm{u})\}.
\]
Note that $H_\alpha (\bm{x}, \bm{p})$ can be interpreted as a version of the \emph{soft maximum value} of $-\bm{p} \cdot f(\bm{x}, \bm{u}) - r(\bm{x}, \bm{u})$, which is the objective function in the standard Hamiltonian. 
Motivated by this observation, we refer to the Hamiltonian~\eqref{Ham_soft} and
the HJB equation~\eqref{HJB_soft} as the \emph{soft Hamiltonian} and the
\emph{soft HJB equation}, respectively.
\end{remark}

Before studying the viscosity solution of the soft HJB equation, 
we discuss several properties of the soft Hamiltonian $H_\alpha$.

\begin{proposition}\label{prop:convex}
	Suppose that Assumption \ref{assumption} and  the control set $U$ is compact. Then, the Hamiltonian $H_\alpha(\bm{x},\bm{p})$ has a finite value and is convex in $\bm{p}$.
\end{proposition}
\begin{proof}
	First, we show that $H_\alpha$ is finite, i.e., $H_\alpha(\bm{x},\bm{p})<+\infty$ for all $(\bm{x},\bm{p})\in\bbr^n\times\bbr^n$. By  Assumption~\ref{assumption} $(i)$ and $(v)$, for each $(\bm{x},\bm{p})\in\bbr^n\times\bbr^n$, 
\[
\min_{\bm{u} \in U} \{\bm{p}\cdot f(\bm{x},\bm{u})+r(\bm{x},\bm{u})\}
\]
admits an optimal solution and has a finite optimal value, say $C=C(\bm{x},\bm{p})$. Therefore, we have
	\[\exp\left(-\frac{\bm{p}\cdot f(\bm{x},\bm{u})+r(\bm{x},\bm{u})}{\alpha}\right)\le \exp\left(-\frac{C}{\alpha}\right),\]
	which implies that $H_\alpha(\bm{x},\bm{p})\le -C(\bm{x},\bm{p})+\alpha\log|U|$. 
	
	We now show that the map $\bm{p}\mapsto H_\alpha(\bm{x},\bm{p})$ is convex. Fix $\bm{x}\in\bbr^n$ and a constant $0<\lambda <1$.
	By the H\"older's inequality for pair $\left(\frac{1}{\lambda},\frac{1}{1-\lambda}\right)$, we obtain
	\begin{align*}
	H_{\alpha}&(\bm{x},\lambda \bm{p}_1+(1-\lambda)\bm{p}_2)\\
	&=\alpha\log\int_{U}\exp\left(-\frac{(\lambda \bm{p}_1+(1-\lambda)\bm{p}_2)\cdot f(\bm{x},\bm{u})+r(\bm{x},\bm{u})}{\alpha}\right)\,\d\bm{u}\\
	&=\alpha\log\int_{U}\exp\left(-\left(\lambda\frac{ \bm{p}_1\cdot f(\bm{x},\bm{u})+r(\bm{x},\bm{u})}{\alpha}+(1-\lambda)\frac{\bm{p}_2\cdot f(\bm{x},\bm{u})+r(\bm{x},\bm{u})}{\alpha}\right)\right)\,\d\bm{u}\\
	&=\alpha \log\int_{U}\exp\left(-\frac{ \bm{p}_1\cdot f(\bm{x},\bm{u})+r(\bm{x},\bm{u})}{\alpha}\right)^\lambda\exp\left(-\frac{\bm{p}_2\cdot f(\bm{x},\bm{u})+r(\bm{x},\bm{u})}{\alpha}\right)^{1-\lambda}\,\d\bm{u}\\
	&\le \alpha \log \left[\left(\int_{U}\exp\left(-\frac{\bm{p}_1\cdot f(\bm{x},\bm{u})+r(\bm{x},\bm{u})}{\alpha}\right)\d\bm{u}\right)^\lambda\left(\int_{U}\exp\left(-\frac{\bm{p}_2\cdot f(\bm{x},\bm{u})+r(\bm{x},\bm{u})}{\alpha}\right)\d\bm{u}\right)^{1-\lambda}\right]\\
	&=\lambda\left(\alpha\log \int_U \exp\left(-\frac{\bm{p}_1\cdot f(\bm{x},\bm{u})+r(\bm{x},\bm{u})}{\alpha}\right)\,\d\bm{u}\right)\\
	&\quad +(1-\lambda)\left(\alpha\log \int_U \exp\left(-\frac{\bm{p}_2\cdot f(\bm{x},\bm{u})+r(\bm{x},\bm{u})}{\alpha}\right)\,\d\bm{u}\right)\\
	&=\lambda H_\alpha(\bm{x},\bm{p}_1)+(1-\lambda)H_\alpha(\bm{x},\bm{p}_2).
	\end{align*}
	Thus, $\bm{p}\mapsto H_\alpha(\bm{x},\bm{p})$ is convex for each $\bm{x} \in \mathbb{R}^n$.
\end{proof}

	We discuss a few notable aspects regarding the convexity of the soft Hamiltonian $H_\alpha$.
	\begin{enumerate}
		\item Instead of using the H\"older's inequality, we can directly calculate the gradient and the Hessian of $H_\alpha$ with respect to $\bm{p}$ as
		\begin{align*}
		\nabla_{\bm{p}} H_\alpha&=\nabla_{\bm{p}}\left(\alpha\log\int_U \exp\left(-\frac{\bm{p}\cdot f+r}{\alpha}\right)\d\bm{u}\right)=\alpha \frac{\nabla_{\bm{p}}\int_{U}\exp\left(-\frac{\bm{p}\cdot f+r}{\alpha}\right)\d\bm{u}}{\int_{U}\exp\left(-\frac{\bm{p}\cdot f+r}{\alpha}\right)\d\bm{u}}\\
		&=-\frac{\int_{U}\exp\left(-\frac{\bm{p}\cdot f+r}{\alpha}\right)f\d\bm{u}}{\int_{U}\exp\left(-\frac{\bm{p}\cdot f+r}{\alpha}\right)\d\bm{u}},
		\end{align*}
		and
		\begin{align*}
		\nabla_{\bm{p}}^2H_\alpha&=\frac{1}{\alpha}\bigg(\left(\int_{U}\exp\left(-\frac{\bm{p}\cdot f+r}{\alpha}\right)f\otimes f\d\bm{u}\right)\left(\int_{U}\exp\left(-\frac{\bm{p}\cdot f+r}{\alpha}\right)\d\bm{u}\right)&\\
		&\hspace{2cm}-\left(\int_{U}\exp\left(-\frac{\bm{p}\cdot f+r}{\alpha}\right)f\d\bm{u}\right)\otimes \left(\int_{U}\exp\left(-\frac{\bm{p}\cdot f+r}{\alpha}\right)f\d\bm{u}\right)\bigg)\\
		&\quad\times \left(\int_{U}\exp\left(-\frac{\bm{p}\cdot f+r}{\alpha}\right)\d\bm{u}\right)^{-2}.
		\end{align*}
Let
		\[g_{\bm{p}}(\bm{u};\bm{x}):=\frac{1}{Z_{\bm{p}}}\exp\left(-\frac{\bm{p}\cdot f(\bm{x},\bm{u})+r(\bm{x},\bm{u})}{\alpha}\right),\quad Z_{\bm{p}} := \int_U \exp\left(-\frac{\bm{p}\cdot f+r}{\alpha}\right)\d\bm{u},\]
		and 
		\[\bar{f}_{\bm{p}}(\bm{x}):=\frac{\int_U \exp(-\frac{\bm{p}\cdot f+r}{\alpha})f(\bm{x},\bm{u})\d\bm{u}}{Z_{\bm{p}}}=\mathbb{E}_{g_{\bm{p}}}[f(\bm{x},\cdot)],\]
		where $\mathbb{E}_{g_{\bm{p}}}[f(\bm{x},\cdot)]$ denotes the expectation of $\bm{u}\mapsto f(\bm{x},\bm{u})$ with respect to the probability density $g_{\bm{p}}(\bm{u};\bm{x})$.
		Then, the gradient and the Hessian of $H_\alpha$ with respect to $\bm{p}$ can be expressed as
		\[\nabla_{\bm{p}} H_\alpha = - \bar{f}_{\bm{p}}(\bm{x}) = -\mathbb{E}_{g_{\bm{p}}}(f(\bm{x},\bm{u})),\]
		and
		\[\nabla_{\bm{p}}^2 H_\alpha = \frac{1}{\alpha}\frac{\int_U(f-\bar{f}_{\bm{p}})\otimes (f-\bar{f}_{\bm{p}})\exp\left(-\frac{\bm{p}\cdot f +r}{\alpha}\right)\d\bm{u}}{Z_{\bm{p}}}=\frac{1}{\alpha}\textup{Cov}_{g_{\bm{p}}}[f(\bm{x},\cdot)],\]
		where $\textup{Cov}_{g_{\bm{p}}}[f(\bm{x},\cdot)]$ denotes the covariance matrix of the map $\bm{u}\mapsto f(\bm{x},\bm{u})$ with respect to the probability density $g_{\bm{p}}$. Therefore, we conclude that $\nabla_{\bm{p}}^2 H_\alpha\succeq 0$. By using this formulation, we may interpret the gradient and the Hessian of the soft Hamiltonian $H_\alpha$ as the expected value and the covariance matrix of $\bm{u}\mapsto f(\bm{x},\bm{u})$ with respect to the probability density $g_{\bm{p}}(\bm{u};\bm{x})$. The probability density $g_{\bm{p}}$ can be interpreted as a Boltzmann distribution  if we consider the quantity $\bm{p}\cdot f(\bm{x},\bm{u})+r(\bm{x},\bm{u})$ as  ``energy" (c.f. Remark \ref{R3.2}). 
		
		\item Furthermore, if we assume that there exists $r>0$ such that\footnote{This condition can be interpreted as the condition that the dynamics at any state $\bm{x}\in\bbr^n$ can be driven to any direction.}
		\begin{equation}\label{control}
		B(0,r)\subset \overline{\mathrm{conv}[f(\bm{x},U)]} \quad \forall \bm{x} \in \mathbb{R}^n,
		\end{equation}
		where $\mathrm{conv}[f(\bm{x},U)]$ denotes the convex hull of $f(\bm{x},U) := \{ f(\bm{x}, \bm{u}) \in \mathbb{R}^n \mid \bm{u} \in U\}$, then we can show that $H_\alpha$ is strictly convex in $\bm{p}$. 
		To see this, suppose not. 
Then, the H\"older's inequality in the proof of Proposition~\ref{prop:convex} holds with equality for some $\bm{x},\bm{p}_1$ and $\bm{p}_2$ in $\mathbb{R}^n$. However, the equality condition for the H\"older's inequality implies that
		\[\exp\left(-\frac{(\bm{p}_1-\bm{p}_2)\cdot f(\bm{x},\bm{u})}{\alpha}\right)\equiv C \quad \forall \bm{u}\in U\]
		for some constant $C$,
		or equivalently
		\[(\bm{p}_1-\bm{p}_2)\cdot f(\bm{x},\bm{u}) \equiv -\alpha\log C =: \tilde{C} \quad \forall \bm{u}\in U,\]
		where $\tilde{C}$ is a constant independent of $\bm{u}$. Therefore, the set $f(\bm{x},U)$ is a subset of the hyperplane
		\[\{\bm{z}\in \bbr^n ~\mid~ (\bm{p}_1-\bm{p}_2)\cdot \bm{z} = \tilde{C}\},\]
		which  contradicts the assumption that  $B(0,r)\subset \textup{conv}~f(\bm{x},U)$. Thus, the H\"older's inequality strictly holds, and $H_\alpha$ is strictly convex in $\bm{p}$. 
	\end{enumerate}

The following proposition shows some regularity of the soft Hamiltonian, which will be used to guarantee the uniqueness of the viscosity solution of HJB equation \eqref{HJB_soft}.

\begin{proposition}\label{P3.2}
	Suppose that Assumption \ref{assumption} holds and   the control set $U$ is compact. Then, the soft Hamiltonian $H_\alpha$ satisfies the following conditions:
	\begin{equation}\nonumber
	\begin{split}
	|H_{\alpha}(\bm{x},\bm{p})-H_{\alpha}(\bm{y},\bm{p})|&\le |\bm{p}|\omega_f(|\bm{x}-\bm{y}|,R)+\omega_r(|\bm{x}-\bm{y}|) \quad \forall \bm{x}, \bm{y} \in \mathbb{R}^n \mbox{ s.t. } |\bm{x}|,|\bm{y}|<R\\
|H_{\alpha}(\bm{x},\bm{p})-H_{\alpha}(\bm{x},\bm{q})|&\le C\left(1+|\bm{x}|+\sup_{\bm{u}\in U}|\bm{u}|\right) |\bm{p}-\bm{q}|.
\end{split}
\end{equation}	
\end{proposition}
\begin{proof}
	It follows from the definition of $H_\alpha$ that
	\begin{align*}
	|H_\alpha(\bm{x},\bm{p})-H_\alpha(\bm{y},\bm{p})|=\alpha\left|\log \frac{\int_{U}\exp\left(-\frac{\bm{p}\cdot f(\bm{x},\bm{u})+r(\bm{x},\bm{u})}{\alpha}\right)\d\bm{u}}{\int_{U}\exp\left(-\frac{\bm{p}\cdot f(\bm{y},\bm{u})+r(\bm{y},\bm{u})}{\alpha}\right)\d\bm{u}}\right|.
	\end{align*}
	However, note that for any measurable functions $F,G\ge0$, we have
	\[\int_U F(\bm{u})\,\d\bm{u}\le  \max_{\bm{u}\in U}\left\{F(\bm{u})/G(\bm{u})\right\}\int_{U} G(\bm{u})\, \d\bm{u}.\]
	Thus, we further estimate $|H_\alpha(\bm{x},\bm{p})-H_\alpha(\bm{y},\bm{p})|$ as
	\[|H_\alpha(\bm{x},\bm{p})-H_{\alpha}(\bm{y},\bm{p})|\le \alpha\left|\log \max_{\bm{u}\in U}\left\{\exp\left(-\frac{\bm{p}\cdot(f(\bm{x},\bm{u})-f(\bm{y},\bm{u}))+r(\bm{x},\bm{u})-r(\bm{y},\bm{u})}{\alpha}\right)\right\}\right|.\]
	Since the exponential function is an increasing function, we can move the maximum operator inside the exponential. Therefore, for any $\bm{x},\bm{y}\in\bbr^n$ such that $|\bm{x}|,|\bm{y}|<R$,  we have
	\begin{align*}
	|H_\alpha(\bm{x},\bm{p})-H_{\alpha}(\bm{y},\bm{p})|&\le \alpha\left| \max_{\bm{u}\in U}\left\{-\frac{\bm{p}\cdot(f(\bm{x},\bm{u})-f(\bm{y},\bm{u}))+r(\bm{x},\bm{u})-r(\bm{y},\bm{u})}{\alpha}\right\}\right|\\
	&=\max_{\bm{u}\in U}\left\{\bm{p}\cdot(f(\bm{y},\bm{u})-f(\bm{x},\bm{u}))+r(\bm{y},\bm{u})-r(\bm{x},\bm{u})\right\}\\
	&\le |\bm{p}|\omega_f(|\bm{x}-\bm{y}|,R)+\omega_r(|\bm{x}-\bm{y}|).
	\end{align*}
	The second assertion can be shown in a similar way as follows:
	\begin{align*}
	|H_\alpha&(\bm{x},\bm{p})-H_\alpha(\bm{x},\bm{q})|\\
	&=\alpha\left|\log \frac{\int_{U}\exp\left(-\frac{\bm{p}\cdot f(\bm{x},\bm{u})+r(\bm{x},\bm{u})}{\alpha}\right)\d\bm{u}}{\int_{U}\exp\left(-\frac{\bm{q}\cdot f(\bm{x},\bm{u})+r(\bm{x},\bm{u})}{\alpha}\right)\d\bm{u}}\right|\le \alpha\left|\log \max_{\bm{u}\in U}\left\{\exp\left(-\frac{(\bm{p}-\bm{q})\cdot f(\bm{x},\bm{u})}{\alpha}\right)\right\}\right|\\
	&=\max_{\bm{u}\in U}\left\{(\bm{q}-\bm{p})\cdot f(\bm{x},\bm{u})\right\}\le |\bm{p}-\bm{q}|\max_{\bm{u}\in U}|f(\bm{x},\bm{u})|=C|\bm{p}-\bm{q}|\left(1+|\bm{x}|+\max_{\bm{u}\in U}|\bm{u}|\right).
	\end{align*}
\end{proof}

\begin{remark}\label{R3.2}
	In the soft HJB equation \eqref{HJB}, the idea of the maximum entropy optimal control is translated into minimizing $\nabla_{\bm{x}} V_\alpha(t,\bm{x})\cdot f(\bm{x},\bm{u})+r(\bm{x},\bm{u})-\alpha \mathrm{H}(g)=L(\bm{x}, \nabla_{\bm{x}} V_\alpha, \bm{u})-\alpha \mathrm{H}(g)$, i.e.,
	\[\min_{g\in\mathcal{G}}\left\{\int_{U}L(\bm{x},\nabla_{\bm{x}} V_\alpha, \bm{u})g(\bm{u})\d\bm{u}-\alpha \mathrm{H}(g)\right\}.\]
	It is remarkable that this optimization problem resembles the minimization of the Helmholtz free energy $F$ in physics, defined as
	\[F:=U-TS,\quad \mbox{$U$: internal energy, $T$: Temperature, $S$: Entropy}.\]
	Thus, if we interpret $L$ as ``internal energy" and the temperature parameter $\alpha$ as a physical temperature, the quantity to be minimized is exactly the same as the Helmholtz free energy. Moreover, the optimal control $g_\alpha^*$ corresponds to the canonical ensemble of the given $L( \bm{x},\nabla_{\bm{x}} V_\alpha, \bm{u})$:
	\[g^*(\bm{x}, \nabla_{\bm{x}} V_\alpha, \bm{u}):=\frac{\exp\left(-\frac{1}{\alpha}L( \bm{x},\nabla_{\bm{x}} V_\alpha, \bm{u})\right)}{\int_{U}\exp\left(-\frac{1}{\alpha}L(\bm{x},\nabla_{\bm{x}} V_\alpha, \bm{u})\right)\,\d\bm{u}}=\frac{\exp\left(-\beta L( \bm{x},\nabla_{\bm{x}} V_\alpha, \bm{u})\right)}{Z},\]
	where $\beta=\frac{1}{\alpha}$ is an inverse temperature in statistical mechanics and $Z$ is a partition function. Then, with the canonical ensemble, the corresponding Helmholtz free energy is given by
	\[F=-\alpha\log Z = -H_\alpha(\bm{x},\nabla_{\bm{x}}V_\alpha),\]
	which is the exact minimum value of the quantity $\int_U L( \bm{x},\nabla_{\bm{x}} V_\alpha, \bm{u})g(\bm{u})\d\bm{u}-\alpha \mathrm{H}(g)$ in Lemma \ref{L3.2}. 
	This observation suggests a connection between maximum entropy optimal control and statistical mechanics. 
	A deeper connection may provide more insights into maximum entropy optimal control from the perspective of statistical mechanics. 
	We leave this as  future research.
\end{remark}

\subsection{Viscosity Solutions}

The soft HJB equation has been derived under the assumption that the value function $V_\alpha$ is continuously differentiable. 
We now relax this assumption and show that $V_\alpha$ 
satisfies the soft HJB equation in the sense of viscosity solutions~\cite{CrandallLions83, CrandallEvansLions84}. 
Recall the definition of the viscosity solution of the HJB equation for the terminal-value problem.
\begin{definition}\label{D3.1}
	A continuous function $V: [0,T] \times \mathbb{R}^n \to \mathbb{R}$ is a viscosity solution of the HJB equation \eqref{HJB_soft} if the following conditions hold:
	\begin{enumerate}
		\item $V(T,\bm{x})=q(\bm{x})$ for all $\bm{x}\in \bbr^n$.
		\item (Subsolution) For any $\phi\in C^1([0,T]\times\bbr^n)$ such that $V-\phi$ has a local maximum at $(t_0,\bm{x}_0)$,
		\[\partial_t \phi(t_0,\bm{x}_0)-H_\alpha(\bm{x}_0,\nabla_{\bm{x}}\phi(t_0,\bm{x}_0))\ge0.\]
		\item (Supersolution) For any $\phi\in C^1([0,T]\times\bbr^n)$ such that $V-\phi$ has a local minimum at $(t_0,\bm{x}_0)$,
		\[\partial_t \phi(t_0,\bm{x}_0)-H_\alpha(\bm{x}_0,\nabla_{\bm{x}}\phi(t_0,\bm{x}_0))\le0.\]
	\end{enumerate}
\end{definition}
Note that the inequalities in the definition of sub- and supersolutions are reversed compared to the ones in their standard definition. This is because our HJB equation is a terminal value problem as opposed to the one in the standard definition. The following theorem states that
the soft HJB equation~\eqref{HJB_soft} has a unique viscosity solution, which corresponds to
 the value function $V_\alpha$ of the maximum entropy control problem.
\begin{theorem}\label{thm:viscosity}
	Suppose that Assumption \ref{assumption} holds and  the control set $U$ is compact. Then, the value function $V_\alpha$ is the unique viscosity solution of the HJB equation \eqref{HJB_soft}.
\end{theorem}
\begin{proof}
The idea of our proof is adopted from the proof of \cite[Theorem 2, Section 10.3]{Evans2010}.
	We first show that $V_\alpha$ satisfies the two conditions in the definition of viscosity solutions. 
	
	\noindent $\bullet$ (Supersolution): Suppose there exists $\phi\in C^1([0,T]\times\bbr^n)$ such that $V_\alpha-\phi$ has a local minimum at $(t_0,\bm{x}_0)$. We need to show that
	\[\partial_t\phi(t_0,\bm{x}_0)-H_\alpha(\bm{x}_0,\nabla_{\bm{x}}\phi(t_0,\bm{x}_0)))\le0.\]
	Suppose the inequality fails to hold. Then, there exists a neighborhood $\mathcal N$ of $(t_0,\bm{x}_0)$ such that
	\begin{equation}\label{cont1}
	\partial_t\phi(t,\bm{x})-H_\alpha(\bm{x},\nabla_{\bm{x}}\phi(t,\bm{x})))\ge \eta>0 \quad  \forall (t,\bm{x})\in \mathcal N.
	\end{equation}
	Then, there exists $\delta>0$ such that for any $(t,\bm{x})\in \mathcal N$ satisfying $|(t,\bm{x})-(t_0,\bm{x}_0)|<\delta$, 
	\[\phi(t,\bm{x})-\phi(t_0,\bm{x}_0)\le V_\alpha(t,\bm{x})-V_\alpha(t_0,\bm{x}_0).\]
	On the other hand, it follows from Lemma \ref{stability}, we can choose a small time interval $h$ such that 
	\[|x(s)-\bm{x}_0|+|s-t_0|<\delta \quad \forall x \in [t_0, t_0+h]
	\]
	for any choice of $g$, where $x(s)$ is a solution to \eqref{dyn2} with  control $g\in \mathcal{G}$. Now, it follows from the definition of the value function, there exists a control $g_1\in\mathcal{G}$ such that
	\[V_\alpha(t_0,\bm{x}_0)+\frac{\eta h}{2} >\int_{t_0}^{t_0+h}\left(\int_{U}r(x(s),\bm{u})g_1(s,\bm{u})\,\d\bm{u}-\alpha \mathrm{H}(g_1(s,\cdot))\right)\,\d s+V_\alpha(t_0+h,x(t_0+h)).\]
We then have
	\begin{align*}
	\frac{\eta h}{2}&>\int_{t_0}^{t_0+h}\left(\int_{U}r(x(s),\bm{u})g_1(s,\bm{u})\,\d\bm{u}-\alpha \mathrm{H}(g_1(s,\cdot))\right)\,\d s+V_\alpha(t_0+h,x(t_0+h))-V_\alpha(t,\bm{x})\\
	&\ge \int_{t_0}^{t_0+h}\left(\int_{U}r(x(s),\bm{u})g_1(s,\bm{u})\,\d\bm{u}-\alpha \mathrm{H}(g_1(s,\cdot))\right)\,\d s+\phi(t_0+h,x(t_0+h))-\phi(t,\bm{x})\\
	&=\int_{t_0}^{t_0+h}\left(\int_{U}r(x(s),\bm{u})g_1(s,\bm{u})\,\d\bm{u}-\alpha \mathrm{H}(g_1(s,\cdot))+\frac{\d}{\d s}\phi(s,x(s))\right)\,\d s\\
	&=\int_{t_0}^{t_0+h}\bigg(\int_{U}r(x(s),\bm{u})g_1(s,\bm{u})\,\d\bm{u}-\alpha \mathrm{H}(g_1(s,\cdot))+\partial_s\phi(s,x(s))\\
	&\hspace{2cm}+\nabla_x\phi(s,x(s))\cdot\left(\int_{U}f(x(s),\bm{u})g(s,\bm{u})\,\d\bm{u}\right)\bigg)\,\d s\\
	&=\int_{t_0}^{t_0+h}\left(\partial_s\phi(s,x(s))+\int_{U}(\nabla_{\bm{x}}\phi(s,x(s))\cdot f(x(s),\bm{u})+r(x(s),\bm{u}))g_1(s,\bm{u})\,\d\bm{u}-\alpha \mathrm{H}(g_1(s,\cdot))\right)\,\d s.
	\end{align*}
Lemma \ref{L3.2} implies that the optimal value of
	\[
	\min_{g \in \mathcal{G}} \bigg \{ \int_{U} \nabla_{\bm{x}} \phi(s,x(s))\cdot f(x(s),\bm{u})+r(x(s),\bm{u}))g(s,\bm{u})\,\d\bm{u}-\alpha \mathrm{H}(g(s,\cdot)) \bigg \}\]
	is equal to $-H_\alpha(x(s),\nabla_{\bm{x}}\phi(s,x(s)))$, which is achieved at	\[g(s,\bm{u})=\frac{\exp(-\frac{1}{\alpha}(\nabla_{\bm{x}}\phi(t,x(s))\cdot f(x(s),\bm{u})+r(x(s),\bm{u})))}{\int_{U}\exp\left(-\frac{1}{\alpha}(\nabla_{\bm{x}} \phi(t,x(s))\cdot f(x(s),\bm{u})+r(x(s),\bm{u}))\right)\,\d\bm{u}}.\]
	Thus, it follows from \eqref{cont1} that
	\begin{align*}
	\frac{\eta h}{2}>\int_{t_0}^{t_0+h}(\partial_s\phi(s,x(s))-H_\alpha(x(s),\nabla_{\bm{x}}\phi(s,x(s))))\,\d s\ge \eta h,
	\end{align*}
	which is  a contradiction. Therefore, we conclude that
	\[\partial_t\phi(t_0,\bm{x}_0)-H_\alpha(\bm{x}_0,\nabla_{\bm{x}}\phi(t_0,\bm{x}_0))\le0.\]

	\noindent $\bullet$ (Subsolution) Similarly, let $\phi\in C^1([0,T]\times \bbr^n)$ such that $V_\alpha-\phi$ has a local maximum at $(t_0,\bm{x}_0)$. Then, for $(t,\bm{x})$ close enough to $(t_0,\bm{x}_0)$,
	\[\phi(t,\bm{x})-\phi(t_0,\bm{x}_0)\ge V_\alpha(t,\bm{x})-V_\alpha(t_0,\bm{x}_0).\]
Choose an arbitrary constant control $g(t,\bm{u}) \equiv g_2(\bm{u})$ and let $x(s)$ be a solution to \eqref{dyn2} with control $g_2$ for $t_0\le s\le t_0+h$. By the dynamic programming equation \eqref{dpp} for $V_\alpha$, 
	\[V_\alpha(t_0,\bm{x}_0)\le \int_{t_0}^{t_0+h}\left(\int_{U}r(x(s),\bm{u})g_2(\bm{u})\,\d\bm{u}-\alpha \mathrm{H}(g_2(\cdot))\right)\,\d s+V_\alpha(t_0+h,x(t_0+h)).\]
	Thus, we have
	\begin{align*}
	0&\le \int_{t_0}^{t_0+h}\left(\int_{U}r(x(s),\bm{u})g_2(\bm{u})\,\d\bm{u}-\alpha \mathrm{H}(g_2(\cdot))\right)\,\d s +V_\alpha(t_0+h,x(t_0+h))-V_\alpha(t_0,\bm{x}_0)\\
	&\le \int_{t_0}^{t_0+h}\left(\int_{U}r(x(s),\bm{u})g_2(\bm{u})\,\d\bm{u}-\alpha \mathrm{H}(g_2(\cdot))\right)\,\d s+\phi(t_0+h,x(t_0+h))-\phi(t_0,\bm{x}_0)\\
	&=\int_{t_0}^{t_0+h}\left(\int_{U}r(x(s),\bm{u})g_2(\bm{u})\,\d\bm{u}-\alpha \mathrm{H}(g_2(\cdot))+\frac{\d}{\d s}\phi(s,x(s))\right)\,\d s\\
	&=\int_t^{t+h}\Bigg(\int_{U}r(x(s),\bm{u})g_2(\bm{u})\,\d\bm{u}-\alpha \mathrm{H}(g_2(\cdot))+\partial_s \phi(s,x(s))\\
	&\hspace{3cm}+\nabla_x\phi(s,x(s))\cdot \left(\int_{U} f(x(s),\bm{u})g_2(\bm{u})\,\d\bm{u}\right)\Bigg)\,\d s.
	\end{align*}
Dividing both sides of inequality by $h$ and letting $h\to0$, we obtain
	\begin{align*}
	0&\le \partial_t \phi(t_0,\bm{x}_0)+\int_{U}(\nabla_{\bm{x}}\phi(t_0,\bm{x}_0)\cdot f(\bm{x}_0,\bm{u})+r(\bm{x}_0,\bm{u}))g_2(\bm{u})\,\d\bm{u} -\alpha \mathrm{H}(g_2(\cdot)).
	\end{align*}
Taking the infimum of both sides with respect to $g_2$ yields
	\[0\le \partial_t\phi(t_0,\bm{x}_0)-H_\alpha(\bm{x},\nabla_{\bm{x}}\phi(t_0,\bm{x}_0))\] 
	by Lemma \ref{L3.2}.
	Therefore, the value function $V_\alpha$ satisfies the two conditions in the definition of viscosity solutions, and it also satisfies the terminal condition. 
This suggests that the value function $V_\alpha$ is a viscosity solution of the soft HJB equation.

By the regularity of the soft Hamiltonian $H_\alpha(\bm{x},\bm{p})$ in Proposition \ref{P3.2}, the uniqueness of the viscosity solution directly follows from the comparison principle of HJB equations (for example, \cite[Theorem 3.15, Section III]{Bardi97}). 
\end{proof}

 Finally, we show that the value function $V_\alpha$ is also a viscosity supersolution of the HJB equation that has the opposite sign compared to \eqref{HJB_soft}. The following proposition will be used in the next subsection, where we present the optimality condition in Proposition \ref{P3.3}.
	
	\begin{proposition}\label{prop:bilateral}
		Suppose that Assumption \ref{assumption} holds and  the control set $U$ is compact. Then, the value function $V_\alpha$ is a viscosity supersolution of the following HJB equation:
		\begin{equation}\label{neg_HJB_soft}
		-\partial_t V_\alpha +H_\alpha(\bm{x},\nabla_{\bm{x}}V_\alpha)=0.
		\end{equation}
	\end{proposition} 

\begin{proof}
	 Let $\phi\in C^1([0,T]\times \bbr^n)$ such that $V_\alpha-\phi$ has a local minimum at $(t_0,\bm{x}_0)\in(0,T)\times \bbr^n$, and fix an arbitrary $g\in \mathcal{G}$. Let $x(s)$ be a solution to \eqref{dyn2} with control $g$ and $x(t_0)=\bm{x}_0$. Then, by the dynamic programming principle, we have
	\begin{equation}\nonumber
	V_\alpha(t_0-h,x(t_0-h))\le V_\alpha(t_0,\bm{x}_0)+\int_{t_0-h}^{t_0} \left(\int_U r(x(s),\bm{u})g(s,\bm{u})\d \bm{u} -\alpha \mathrm{H}(g(s,\cdot))\right)\d s.
	\end{equation} 
We can then deduce that
	\begin{align*}
	\phi(t_0,\bm{x}_0)-\phi(t_0-h,x(t_0-h))&\ge V_\alpha(t_0,\bm{x}_0)-V_\alpha(t_0-h,x(t_0-h))\\
	&\ge -\int_{t_0-h}^{t_0}\left(\int_U r(x(s),\bm{u})g(s,\bm{u})\d \bm{u} -\alpha \mathrm{H}(g(s,\cdot))\right)\d s.
	\end{align*}
Dividing both sides by $h$ and letting $h\to0$, we have
	\[\partial_t \phi(t_0,\bm{x}_0) +\nabla_{\bm{x}}\phi(t_0,\bm{x}_0)\cdot \int_U f(\bm{x}_0,\bm{u})g(t_0,\bm{u})\d \bm{u}+\int_U r(\bm{x}_0,\bm{u})g(t_0,\bm{u})\d \bm{u} -\alpha \mathrm{H}(g(t_0,\cdot))\ge0.\]
Minimizing both sides with respect to $g(t_0,\cdot)$ and using Lemma \ref{L3.2} again, we conclude that
	\[\partial_t \phi(t_0,\bm{x}_0) - H_\alpha(\bm{x}_0,\nabla_{\bm{x}}\phi(t_0,\bm{x}_0))\ge 0,\]
	which implies that $V_\alpha$ is a supersolution of \eqref{neg_HJB_soft}.
\end{proof}

\subsection{Conditions for Optimality and Optimal Control Synthesis}\label{sec:opt_con}

We now provide necessary and sufficient conditions for the optimality of control $g\in\mathcal{G}$ in terms of the generalized derivatives of  $V_\alpha$ as in \cite{Bardi97}. For any given $\bm{x}$ and $g\in\mathcal{G}$, we let
\begin{equation}\label{eta}
\eta(t) := V_\alpha(t,x(t)) +\int_0^t \left(\int_Ur(x(s),\bm{u}) g(s,\bm{u})\d\bm{u} -\alpha \mathrm{H}(g(s,\cdot))\right)\d s,
\end{equation}
where $x(t)$ is a solution to \eqref{dyn2} governed by control $g$ with initial data $x(0)=\bm{x}$. Then, it follows from   Lemma \ref{L3.1} that for any $h>0$,
\begin{align*}
	\eta(t+h) -\eta(t) &= \int_t^{t+h}\left(\int_U r(x(s),\bm{u})g(s,\bm{u})\d\bm{u}-\alpha \mathrm{H}(g(s,\cdot))\right)\d s\\
	&\quad+V_\alpha(t+h,x(t+h))-V_\alpha(t,x(t))\ge 0.
\end{align*} 
Therefore, the map $t\mapsto \eta(t)$ is   non-decreasing and  it is constant if and only if  $g\in\mathcal{G}$ is optimal. Suppose for a moment that $V_\alpha$ is continuously differentiable. Then, $\eta$ is constant if and only if
\[\frac{\d\eta}{\d t} =\partial_t V_\alpha(t,x(t))+\nabla_{\bm{x}}V_\alpha (t,x(t))\cdot \int_U f(x(t),\bm{u})g(t,\bm{u})\d\bm{u} +\int_U r(x(t),\bm{u})g(t,\bm{u})\d\bm{u}-\alpha \mathrm{H}(g(t,\cdot))=0.\] 
By the HJB equation \eqref{HJB_soft}, we have
\begin{align*}
\nabla_{\bm{x}}&V_\alpha(t,x(t))\cdot\int_U f(x(t),\bm{u})g(t,\bm{u})\d\bm{u} +\int_U r(x(t),\bm{u})g(t,\bm{u})\d\bm{u}-\alpha \mathrm{H}(g(t,\cdot))\\
&=-H_\alpha(x(t),\nabla_{\bm{x}}V_\alpha(t,x(t))) = -\alpha\log\int_U \exp\left(-\frac{1}{\alpha} \nabla_{\bm{x}}V_\alpha(t,x(t))\cdot f(x(t),\bm{u})+r(x(t),\bm{u})\right)\d \bm{u}.
\end{align*}
Therefore, it follows from Lemma \ref{L3.2} that the optimal control $g$ should be given as
\[g(t,\bm{u}) = \frac{\exp\left(-\frac{1}{\alpha} \nabla_{\bm{x}}V_\alpha(t,x(t))\cdot f(x(t),\bm{u})+r(x(t),\bm{u})\right)}{\int_U\exp\left(-\frac{1}{\alpha} \nabla_{\bm{x}}V_\alpha(t,x(t))\cdot f(x(t),\bm{u})+r(x(t),\bm{u})\right)\d\bm{u}}.\]
However, since the value function is not continuously differentiable in general, we need the following generalized notion of derivatives to characterize the optimality condition. 

\begin{definition}[\cite{Bardi97}]
	Let $V:[0,T]\times\bbr^d\to\bbr$ be a continuous function. Then, the superdifferential $D^+V$ and subdifferential $D^-V$ at $(t_0,\bm{x}_0)$ are defined as
	\begin{align*}
	&D^+V(t_0,\bm{x}_0)\\
	&:=\left\{(\bm{p}_t,\bm{p}_x)\in \bbr^+\times\bbr^n~\Bigg|~ \limsup_{(t,\bm{x})\to(t_0,\bm{x}_0)}\frac{V(t,\bm{x})-V(t_0,\bm{x}_0)-(\bm{p}_t(t-t_0)+\bm{p}_x\cdot(\bm{x}-\bm{x}_0))}{(|t-t_0|+|\bm{x}-\bm{x}_0|)}\le0\right\},
	\end{align*}
	and
	\begin{align*}
	&D^-V(t_0,\bm{x}_0)\\
	&:=\left\{(\bm{p}_t,\bm{p}_x)\in \bbr^+\times\bbr^n~\Bigg|~ \liminf_{(t,\bm{x})\to(t_0,\bm{x}_0)}\frac{V(t,\bm{x})-V(t_0,\bm{x}_0)-(\bm{p}_t(t-t_0)+\bm{p}_x\cdot(x-\bm{x}_0))}{(|t-t_0|+|\bm{x}-\bm{x}_0|)}\ge0\right\}.
	\end{align*}
	Moreover, the (generalized) Dini directional derivatives $\partial^\pm V$ at $(t_0,\bm{x}_0)$ with the direction $(s_0,\bm{y}_0)$ are defined as
	\begin{align*}
		&\partial^+V(t_0,\bm{x}_0;s_0,\bm{y}_0) :=\limsup_{\stackrel{\varepsilon\to0^+}{(s,\bm{y})\to(s_0,\bm{y}_0)}}\frac{V(t_0+\varepsilon s,\bm{x}_0+\varepsilon \bm{y})-V(t_0,\bm{x}_0)}{\varepsilon},\\
		&\partial^-V(t_0,\bm{x}_0;s_0,\bm{y}_0) :=\liminf_{\stackrel{\varepsilon\to0^+}{(s,\bm{y})\to(s_0,\bm{y}_0)}}\frac{V(t_0+\varepsilon s,\bm{x}_0+\varepsilon \bm{y})-V(t_0,\bm{x}_0)}{\varepsilon}.
	\end{align*} 
\end{definition}
Intuitively, the Dini derivatives provide upper and lower bounds on the infinitesimal directional change, particularly when the function $V$ is not differentiable. 
It is well-known  that $(\bm{p}_t,\bm{p}_x)\in D^+V(t_0,\bm{x}_0)$ ($D^-V(t_0,\bm{x}_0)$, respectively) if and only if there exists $\phi\in C^1(\bbr\times\bbr^n)$ such that $\nabla_{(t,\bm{x})}\phi=(\bm{p}_t,\bm{p}_x)$ and $V-\phi$ attains a local maximum (minimum, respectively) at $(t_0,\bm{x}_0)$~\cite{Bardi97}. Therefore, if $V$ is a viscosity solution of the soft HJB equation \eqref{HJB_soft}, we have
\[\bm{p}_t-H_\alpha(\bm{x}_0,\bm{p}_x)\ge 0 \quad \forall(\bm{p}_t,\bm{p}_x)\in D^+V(t_0,\bm{x}_0)\]
and 
\[\bm{p}_t-H_\alpha(\bm{x}_0,\bm{p}_x)\le 0 \quad \forall(\bm{p}_t,\bm{p}_x)\in D^-V(t_0,\bm{x}_0).\]
We record the following properties of the super-, subderivatives and Dini derivatives that will be used in this subsection.

\begin{lemma}[\cite{Bardi97}]\label{L3.3}
	Suppose the map $x:[0,T]\to\bbr^n$ is differentiable at $t$ and $V:[0,T]\times\bbr^n\to\bbr$ is a continuous function. Let $\textup{id}:[0,T]\to[0,T]$ be an identity map, i.e., $\textup{id}(t)=t$ and let the map $V\circ(\textup{id},x):[0,T]\to \bbr$ be defined by $(V\circ(\textup{id},x))(t) = V(t,x(t))$.
	Then,
	\begin{align*}
		&\partial^-(V\circ (\textup{id},x))(t;1) \ge \partial^-V(t,x(t);1,\dot{x}(t)),\\
		&\partial^+(V\circ (\textup{id},x))(t;1) \le \partial^+V(t,x(t);1, \dot{x}(t)).
	\end{align*}
	If $V$ is locally Lipschitz continuous, then both of the inequalities hold with equality.
\end{lemma}

\begin{lemma}[\cite{Bardi97}]\label{L3.4}
	Let $V:[0,T]\times \bbr^n\to\bbr$ be a continuous function. Then,
	\begin{align*}
		& D^-V(t_0,x_0) = \left\{(\bm{p}_t,\bm{p}_x)~:~\bm{p}_ts_0+\bm{p}_x\cdot\bm{y}_0\le \partial^-V(t_0,\bm{x}_0;s_0,\bm{y}_0)\quad \forall(s_0,\bm{y}_0)\in[0,T]\times \bbr^d\right\},\\
		& D^+V(t_0,x_0) = \left\{(\bm{p}_t,\bm{p}_x)~:~\bm{p}_ts_0+\bm{p}_x\cdot\bm{y}_0\ge \partial^+V(t_0,\bm{x}_0;s_0,\bm{y}_0)\quad \forall(s_0,\bm{y}_0)\in[0,T]\times\bbr^d\right\}.
	\end{align*}
\end{lemma}

We now present the necessary conditions for optimality. The following proposition is a variation of the necessary conditions presented in \cite[Theorem 3.37, Section 3]{Bardi97} for the standard optimal control problems.

\begin{proposition}\label{P3.3}
	Suppose that Assumption \ref{assumption} holds and  the control set $U$ is compact. We further assume that  control $g \in \mathcal{G}$ is an optimal solution to the maximum entropy control problem~\eqref{opt_me} with initial data $\bm{x}$, and let $x(t)$ be the system trajectory governed by control $g$ with $x(0) = \bm{x}$. Then, we have
	\begin{enumerate}
		\item for a.e. $0\le t\le T$,
		\[\partial^-V_\alpha(t,x(t);1, \dot{x}(t))+\int_U r(x(t),\bm{u})g(t,\bm{u})\d\bm{u} -\alpha \mathrm{H}(g(t,\cdot))\le0;\]
		\item for a.e. $0\le t\le T$,
		\[-\partial^+V_\alpha(t,x(t);-1,- \dot{x}(t))+\int_U r(x(t),\bm{u})g(t,\bm{u})\d\bm{u} -\alpha \mathrm{H}(g(t,\cdot))\le0;\]
		\item for a.e. $0\le t\le T$ and all $(\bm{p}_t,\bm{p}_x)\in D^{\pm}V_\alpha(t,x(t))$,
		\[\bm{p}_t+\int_U \left(\bm{p}_x\cdot f(x(t),\bm{u})+r(x(t),\bm{u})\right)g(t,\bm{u})\d\bm{u}  -\alpha \mathrm{H}(g(t,\cdot))= 0;\]
		\item for a.e. $0\le t\le T$ and all $(\bm{p}_t,\bm{p}_x)\in D^{\pm}V_\alpha(t,x(t))$,
		\begin{equation}\label{optimal_synthesis}
		g(t,\bm{u}) = \frac{\exp\left(-\frac{1}{\alpha}(\bm{p}_x\cdot f(x(t),\bm{u})+r(x(t),\bm{u}))\right)}{\int_U \exp\left(-\frac{1}{\alpha}(\bm{p}_x\cdot f(x(t),\bm{u})+r(x(t),\bm{u}))\right)\d \bm{u}}. 
		\end{equation}
	\end{enumerate}
\end{proposition}
\begin{proof}
	Since $g$ is optimal, $\eta$ in \eqref{eta} is a constant function. We note that $\eta$ can be represented as
	\[\eta(t) = (V_\alpha\circ(\textup{id},x))(t) +\int_0^t\left(\int_U r(x(s),\bm{u})g(s,\bm{u})\d\bm{u}-\alpha \mathrm{H}(g(s,\cdot))\right)\d s.\] By Lemma \ref{L3.3}, we obtain
	\begin{align*}
	0 &= \frac{\d\eta(t)}{\d t} = \partial^-(V_\alpha\circ(\textup{id},x))(t;1)+\int_U r(x(t),\bm{u})g(t,\bm{u})\d\bm{u}-\alpha \mathrm{H}(g(t,\cdot))\\
	&\ge \partial^-V_\alpha(s,x(t);1, \dot{x}(t))+\int_U r(x(t),\bm{u})g(t,\bm{u})\d\bm{u}-\alpha \mathrm{H}(g(t,\cdot)),
	\end{align*}
	which implies the first condition to hold. 
	Using the fact that 
	\[\frac{\d\eta(t)}{\d t} = -\partial^+(V_\alpha\circ(\textup{id},x))(t;-1) +\int_U r(x(t),\bm{u})g(t,\bm{u})\d\bm{u}-\alpha \mathrm{H}(g(t,\cdot)),\]
	we deduce that the second condition holds. 
	
We now use Lemma \ref{L3.4} to obtain
	\[\partial^-V_\alpha(t,x(t);1, \dot{x}(t)) \ge \bm{p}_t + \bm{p}_x\cdot \dot{x}(t)=\bm{p}_t+\int_U \bm{p}_x\cdot f(x(t),\bm{u})g(t,\bm{u})\d \bm{u} \]
	for all $(\bm{p}_t,\bm{p}_x)\in D^-V_\alpha(t,x(t))$. Together with the first condition, we deduce that for all $(\bm{p}_t,\bm{p}_x)\in D^-V_\alpha(t,x(t))$,
	\[\bm{p}_t+\int_U \left(\bm{p}_x\cdot f(x(t),\bm{u})+r(x(t),\bm{u})\right)g(t,\bm{u})\d\bm{u}  -\alpha \mathrm{H}(g(t,\cdot))= 0.\]
	Similarly, the assertion for $D^+V_\alpha$ in Lemma \ref{L3.4} and the second condition  impliy that for all $(\bm{p}_t,\bm{p}_x)\in D^+V_\alpha(t,x(t))$,
	\[\bm{p}_t+\int_U \left(\bm{p}_x\cdot f(x(t),\bm{u})+r(x(t),\bm{u})\right)g(t,\bm{u})\d\bm{u}  -\alpha \mathrm{H}(g(t,\cdot))= 0.\]
	Therefore, the third condition holds. 
	
Lastly, since $V_\alpha$ is a viscosity solution of \eqref{HJB_soft}, we have for all $(\bm{p}_t,\bm{p}_x)\in D^+V_\alpha(t,x(t))$, 
	\begin{equation}\label{est-p}
	0\le \bm{p}_t-H_{\alpha}(x(t),\bm{p}_x)\le\bm{p}_t+\int_U \left(\bm{p}_x\cdot f(x(t),\bm{u})+r(x(t),\bm{u})\right)g(t,\bm{u})\d\bm{u}  -\alpha\mathrm{H}(g(t,\cdot))= 0,
	\end{equation}
	where the second inequality comes from Lemma \ref{L3.2}. Therefore, all the inequalities above should hold with equality. By Lemma \ref{L3.2}, we conclude that 
	\[g(t,\bm{u}) = \frac{\exp\left(-\frac{1}{\alpha}(\bm{p}_x\cdot f(x(t),\bm{u})+r(x(t),\bm{u}))\right)}{\int_U \exp\left(-\frac{1}{\alpha}(\bm{p}_x\cdot f(x(t),\bm{u})+r(x(t),\bm{u}))\right)\d \bm{u}}.\]
	For $(\bm{p}_t,\bm{p}_x)\in D^-V_\alpha(t,x(t))$, we note that the first inequality in \eqref{est-p} holds with equality by Proposition \ref{prop:bilateral}. Therefore, the same conclusion holds for $(\bm{p}_t,\bm{p}_x)\in D^-V_\alpha(t,x(t))$.
\end{proof}


 The first and second conditions   in Proposition \ref{P3.3} indicate that the sum of the infinitesimal change of $V_\alpha$ along the trajectory and the infinitesimal cost should be less than 0, which  implies the quantity $\eta$ is non-increasing when the optimal control $g$ is implied. When the set $D^\pm V_\alpha(t,x(t))$ is non-empty, the third condition implies that the quantity $\eta$ is a constant function along the trajectory. 

Unlike the standard optimal control case, we have the following improved regularity of the value function $V_\alpha$ as a useful byproduct of the necessary conditions for optimality.
\begin{corollary}\label{C3.1}
	Suppose that Assumption~\ref{assumption} holds,  the control set $U$ is compact, and condition \eqref{control} holds. 
	We further assume that the control $g$ is an optimal solution to the maximum entropy control problem~\eqref{opt_me} with initial data $\bm{x}$, and let $x(t)$ be the system trajectory governed by $g$ with $x(0) = \bm{x}$. Then, the set $D^{\pm}V_\alpha(t,x(t))$ has at most one element. 
\end{corollary}

\begin{proof}
	If $D^\pm V_\alpha(t,x(t))=\emptyset$, then the assertion clearly holds. 
	We now assume that there exist two elements $(\bm{p}^1_t,\bm{p}^1_x),(\bm{p}^2_t,\bm{p}^2_x)\in D^{\pm}V_\alpha (t,x(t))$. By the fourth condition in Proposition \ref{P3.3}, 
	\[g(t,\bm{u}) = \frac{\exp\left(-\frac{1}{\alpha}(\bm{p}_x\cdot f(x(t),\bm{u})+r(x(t),\bm{u}))\right)}{\int_U \exp\left(-\frac{1}{\alpha}(\bm{p}_x\cdot f(x(t),\bm{u})+r(x(t),\bm{u}))\right)\d \bm{u}}\]
	for all $(\bm{p}_t,\bm{p}_x)\in D^{\pm}V_\alpha (t,x(t))$. 
	Then, we have
	\[\bm{p}^1_x\cdot f(x(t),\bm{u}) = \bm{p}^2_x\cdot f(x(t),\bm{u})\quad \forall \bm{u}\in U.\]
By condition \eqref{control}, we conclude that $\bm{p}_x^1=\bm{p}_x^2$. Then, it follows from the third condition  in Proposition \ref{P3.3} that $\bm{p}_t^1=\bm{p}_t^2$. Therefore, $D^{\pm}V_\alpha(t,x(t))$ has at most one element.
\end{proof}

Note that the improved regularity of $V_\alpha$ is due to the explicit representation~\eqref{optimal_synthesis} of optimal control $g$. Therefore, we deduce that the maximum entropy formulation enhances the regularity of the value function. 

We now provide sufficient conditions for optimality, which are extensions of those in the standard optimal control case~\cite[Theorem 3.38, Section 3]{Bardi97}:

\begin{proposition}\label{P3.4}
	Suppose that Assumption \ref{assumption} holds and   the control set $U$ is compact. Let $x(t)$ be the system trajectory governed  by some control $g$ with $x(0) = \bm{x}$. We assume that $V_\alpha$ is locally Lipschitz in a neighborhood of $\{x(t) \mid 0\le t\le T\}$. Then, $g$ is optimal if any of the following conditions holds:
	\begin{enumerate}
		\item for a.e. $0\le t\le T$,
		\[\partial^-V_\alpha(t,x(t);1, \dot{x}(t))+\int_U r(x(t),\bm{u})g(t,\bm{u})\d\bm{u}-\alpha \mathrm{H}(g(t,\cdot))\le 0;\]
		\item for a.e. $0\le t\le T$,
		\[-\partial^+V_\alpha(t,x(t);-1,-\dot{x} (t))+\int_U r(x(t),\bm{u})g(t,\bm{u})\d\bm{u}-\alpha \mathrm{H}(g(t,\cdot))\le 0;\]
		\item for a.e. $0\le t\le T$, there exists $(\bm{p}_t,\bm{p}_x)\in D^{\pm} V(t,x(t))$ such that
		\[\bm{p}_t+\int_U \left(\bm{p}_x\cdot f(x(t),\bm{u})+r(x(t),\bm{u})\right)g(t,\bm{u})\d\bm{u}  -\alpha \mathrm{H}(g(t,\cdot)) \le 0.\]
	\end{enumerate}
\end{proposition}
\begin{proof}
Recall that the optimality of $g$ is equivalent to the fact that $\eta$ defined by \eqref{eta} is a constant function. Since we already observe that $\eta$ is a non-decreasing function, it suffice to show that $\eta$ is a non-increasing function. 

Suppose  the first condition holds. Then, it follows from Lemma \ref{L3.3} that
	\begin{align*}
	\frac{\d\eta(t)}{\d t} &= \partial^-(V_\alpha\circ (\textup{id},x))(t;1)+\int_U r(x(t),\bm{u})g(t,\bm{u})\d\bm{u}-\alpha\mathrm{H}(g(t,\cdot)) \\
	&= \partial^-V_\alpha(t,x(t);1, \dot{x}(t))+\int_U r(x(t),\bm{u})g(t,\bm{u})\d\bm{u}-\alpha \mathrm{H}(g(t,\cdot))\le0.
	\end{align*}
	Thus, $\eta$ is a non-increasing function and therefore, $g$ is optimal. The optimality of $g$ under the second condition can be shown in a similar manner. 
	
Lastly, we assume that the third condition holds. Without loss of generality, we choose $(\bm{p}_t,\bm{p}_x)\in D^+V(t,x(t))$ satisfying the third condition. Then, 
	\begin{align*}
	\frac{\d\eta(t)}{\d t}&= \partial^-V_\alpha(t,x(t);1,\dot{x}(t))+\int_U r(x(t),\bm{u})g(t,\bm{u})\d\bm{u}-\alpha \mathrm{H}(g(t,\cdot))\\
	&\le \partial^+V_\alpha(t,x(t);1, \dot{x}(t))+\int_U r(x(t),\bm{u})g(t,\bm{u})\d\bm{u}-\alpha \mathrm{H}(g(t,\cdot))\\
	&\le \bm{p}_t + \bm{p}_x \cdot \dot{x}(t)+\int_U r(x(t),\bm{u})g(t,\bm{u})\d\bm{u}-\alpha \mathrm{H}(g(t,\cdot))\le0.
	\end{align*}
This implies that $\eta$ is non-increasing, and therefore $g$ is optimal.
\end{proof}
Finally, we provide a necessary and sufficient condition of optimality, which is a corollary of the preceding two propositions.

\begin{corollary}\label{C3.2}
	Suppose that Assumption \ref{assumption} holds and  the control set $U$ is compact. Let $x(t)$ be the system trajectory governed  by some control $g$ with $x(0) = \bm{x}$. We assume that $V_\alpha$ is locally Lipschitz in a neighborhood of $\{x(t) \mid 0\le t\le T\}$. Then, $g$ is optimal if and only if any of the following conditions holds:
	\begin{enumerate}
		\item For a.e. $0\le t\le T$,
		\[\partial^-V_\alpha(t,x(t);1, \dot{x}(t))+\int_U r(x(t),\bm{u})g(t,\bm{u})\d\bm{u}-\alpha \mathrm{H}(g(t,\cdot))\le 0.\]
		\item For a.e. $0\le t\le T$,
		\[-\partial^+V_\alpha(t,x(t);-1,-\dot{x}(t))+\int_U r(x(t),\bm{u})g(t,\bm{u})\d\bm{u}-\alpha \mathrm{H}(g(t,\cdot))\le 0.\]
	\end{enumerate}
	We further assume that $D^{\pm}V(t,x(t))\neq \emptyset$. Then, $g$ is optimal if and only if there exists $(\bm{p}_t,\bm{p}_x)\in D^{\pm} V(t,x(t))$ for a.e. $0\le t\le T$ such that
	\[\bm{p}_t+\int_U \left(\bm{p}_x\cdot f(x(t),\bm{u})+r(x(t),\bm{u})\right)g(t,\bm{u})\d\bm{u}  -\alpha \mathrm{H}(g(t,\cdot)) \le 0.\]
	In this case, the optimal control $g$ can be represented as 
	\begin{equation}\label{gtu}
	g(t,\bm{u}) = \frac{\exp\left(-\frac{1}{\alpha}(\bm{p}_x\cdot f(x(t),\bm{u})+r(x(t),\bm{u}))\right)}{\int_U \exp\left(-\frac{1}{\alpha}(\bm{p}_x\cdot f(x(t),\bm{u})+r(x(t),\bm{u}))\right)\d \bm{u}}.
	\end{equation}
\end{corollary}
\begin{proof}
	The proof directly follows from Proposition \ref{P3.3} and Proposition \ref{P3.4}.
\end{proof}

We now discuss how to synthesize an optimal control using the conditions for optimality. 
When the value function $V_\alpha$ is differentiable at every points $(t,\bm{x})\in[0,T]\times\bbr^n$, then it follows from Corollary \ref{C3.2} that an optimal control is uniquely characterized as a feedback map $\Phi:[0,T]\times \bbr^n\to L^1_{+,1}(U)$, defined by
\[(\Phi(t,\bm{x}))(\bm{u}) = \Phi(\bm{u};t,\bm{x})= \frac{\exp\left(-\frac{1}{\alpha}(\nabla_{\bm{x}}V_\alpha(t,\bm{x}) \cdot f(\bm{x},\bm{u})+r(\bm{x},\bm{u}))\right)}{\int_U \exp\left(-\frac{1}{\alpha}(\nabla_{\bm{x}}V_\alpha(t,\bm{x}) \cdot f(\bm{x},\bm{u})+r(\bm{x},\bm{u}))\right)\d \bm{u}}.\]
On the other hand, when the value function $V_\alpha$ is merely continuous but not differentiable,  we introduce the following set of controls  satisfying Condition 1 in Corollary \ref{C3.2}:
\[ S_1(t,\bm{x}):=\Bigg\{\phi\in L^1_{+,1}(U)~:~\partial^-V_\alpha\left(t,\bm{x};1,\int_{U}f(\bm{x},\bm{u})\phi(\bm{u})\d \bm{u}\right)+\int_U r(\bm{x},\bm{u})\phi(\bm{u})\d\bm{u}-\alpha \mathrm{H}(\phi(\cdot))\le 0 \Bigg\}.\]
If $D^\pm V_\alpha(t,\bm{x})\neq \emptyset$ for all $(t,\bm{x})\in[0,T]\times\bbr^n$, we define $S_2$ as the set  of densities that can be expressed as the form \eqref{gtu} in Corollary \ref{C3.2}:
\[S_2(t,\bm{x}):=\Bigg\{\phi\in L^1_{+,1}(U)~:~\phi(\bm{u})=\frac{\exp\left(-\frac{1}{\alpha}(\bm{p}_x \cdot f(\bm{x},\bm{u})+r(\bm{x},\bm{u}))\right)}{\int_U \exp\left(-\frac{1}{\alpha}(\bm{p}_x \cdot f(\bm{x},\bm{u})+r(\bm{x},\bm{u}))\right)\d \bm{u}} \mbox{ for } (\bm{p}_t, \bm{p}_x)\in D^\pm V_\alpha(t,\bm{x})\Bigg\}.\]
By Corollary \ref{C3.1}, the set $D^\pm V_{\alpha}(t,\bm{x})$ is a singleton whenever it is non-empty. Therefore, we note that the set $S_2(t,\bm{x})$ is also a singleton.

Consider feedback controls $\Phi_i$ such that
\[\Phi_i(t,\bm{x})\in S_i(t,\bm{x}),  \quad i=1,2.\]
Then, by Corollary \ref{C3.2}, $\Phi_i$'s are optimal under the same conditions as those in Corollary \ref{C3.2}.
\begin{corollary}
	Suppose that Assumption \ref{assumption} holds and   the control set $U$ is compact. Moreover, we assume that $V_\alpha$ is locally Lipschitz. Then, the feedback control $\Phi_1(t,\bm{x})$ is optimal. We further assume that $D^\pm V_\alpha(t,\bm{x})\neq\emptyset$ for all $(t,\bm{x})$. Then, the feedback control $\Phi_2(t,\bm{x})$ is also optimal. 
\end{corollary}

\subsection{Asymptotic Consistency}

It seems reasonable to expect that, as $\alpha \to 0$,
 $V_\alpha$ converges to the value function of the standard optimal control problem~\eqref{opt0}. 
This subsection is devoted to showing this convergence property.

We first formally describe the convergence result.
Recall the Laplace principle \cite{dembo2010largedeviation}: for any measurable function $\phi:U\to\bbr$,
\[\lim_{\alpha\to 0} \left(-\alpha \log\int_{U}\exp\left(-\frac{\phi(\bm{u})}{\alpha}\right)\,\d\bm{u}\right)=\inf_{u\in U}\phi(\bm{u}).\]
Therefore, the Hamiltonian $H_{\alpha}(\bm{x},\bm{p}):= \alpha \log\int_{U}\exp\left(-\frac{\bm{p}\cdot f(\bm{x},\bm{u})+r(\bm{x},\bm{u})}{\alpha}\right)\,\d\bm{u}$  converges pointwisely to
\[H_0(\bm{x},\bm{p}):=-\inf_{\bm{u}\in U}\left\{\bm{p}\cdot f(\bm{x},\bm{u})+r(\bm{x},\bm{u})\right\}.\]
Thus, at the formal level, the soft HJB equation \eqref{HJB_soft} converges to the following HJB equation:
\begin{equation}\label{HJB-0}
\partial_t V_0 -H_0(\bm{x},\nabla_{\bm{x}} V_0):=\partial_t V_0 +\inf_{\bm{u}\in U} \left\{\nabla_{\bm{x}} V_0\cdot f(\bm{x},\bm{u})+r(\bm{x},\bm{u})\right\}=0
\end{equation}
as $\alpha \to 0$. 
It is well-known that this HJB equation admits the unique viscosity solution, which coincides with the value function of the standard optimal control problem~\eqref{opt0}. 
Thus, it is natural to use the HJB equations to establish the desired convergence result regarding the value functions. 

We begin by showing  the following uniform convergence of the Hamiltonian $H_\alpha$ to $H_0$ as $\alpha$ tends to $0$.

\begin{lemma}\label{lem-hamiltonian}
	Suppose that Assumption \ref{assumption} holds and  the control set $U$ is compact. Then, the soft Hamiltonian $H_\alpha:\bbr^{2n}\to \bbr$  converges uniformly to the standard Hamiltonian $H_0:\bbr^{2n}\to\bbr$ on any compact subset of $\bbr^{2n}$ as $\alpha \to0$.
\end{lemma}

\begin{proof}
Fix an arbitrary $(\bm{x},\bm{p})\in\bbr^{2n}$. We first notice that
	\begin{align*}
	H_{\alpha}(\bm{x},\bm{p}) &= \alpha \log\int_U \exp\left(-\frac{\bm{p}\cdot f(\bm{x},\bm{u})+r(\bm{x},\bm{u})}{\alpha}\right)\d\bm{u} \\
	&=\alpha \log\int_U \exp\left(-\frac{\bm{p}\cdot f(\bm{x},\bm{u})+r(\bm{x},\bm{u})}{\alpha}\right)\mathscr{U}(\d \bm{u})+\alpha\log|U|\\
	&=:\tilde{H}_\alpha(\bm{x},\bm{u})+\alpha\log|U|,
	\end{align*}
	where $\mathscr{U}$ denotes the uniform probability measure defined by $\mathscr{U}(\d \bm{u}) = \frac{\d\bm{u}}{|U|}$.
Differentiating $\tilde{H}_\alpha$ with respect to $\alpha$ yields
	\begin{align*}
	\frac{\partial}{\partial \alpha} \tilde{H}_\alpha(\bm{x},\bm{p})&=\log \int_{U}\exp\left(-\frac{\bm{p}\cdot f(\bm{x},\bm{u})+r(\bm{x},\bm{u})}{\alpha}\right)\mathscr{U}(\d\bm{u})\\
	&\quad +\alpha\frac{\int_{U}\exp\left(-\frac{\bm{p}\cdot f(\bm{x},\bm{u})+r(\bm{x},\bm{u})}{\alpha}\right)\frac{\bm{p}\cdot f(\bm{x},\bm{u})+r(\bm{x},\bm{u})}{\alpha^2}\mathscr{U}(\d\bm{u})}{\int_{U}\exp\left(-\frac{\bm{p}\cdot f(\bm{x},\bm{u})+r(\bm{x},\bm{u})}{\alpha}\right)\mathscr{U}(\d\bm{u})}\\
	&=\log \int_{U}\exp\left(-\frac{\bm{p}\cdot f(\bm{x},\bm{u})+r(\bm{x},\bm{u})}{\alpha}\right)\mathscr{U}(\d\bm{u})\\
	&\quad +\frac{1}{\alpha} \frac{\int_{U}\exp\left(-\frac{\bm{p}\cdot f(\bm{x},\bm{u})+r(\bm{x},\bm{u})}{\alpha}\right)(\bm{p}\cdot f(\bm{x},\bm{u})+r(\bm{x},\bm{u}))\mathscr{U}(\d\bm{u})}{\int_{U}\exp\left(-\frac{\bm{p}\cdot f(\bm{x},\bm{u})+r(\bm{x},\bm{u})}{\alpha}\right)\mathscr{U}(\d\bm{u})}.
	\end{align*}
	For simplicity, we let $F(\bm{u}) = F(\bm{u};\bm{x},\bm{p}):=\exp\left(-\frac{\bm{p}\cdot f(\bm{x},\bm{u})+r(\bm{x},\bm{u})}{\alpha}\right)$ and consider it as a function of $\bm{u}$ since $(\bm{x}, \bm{p})$ is fixed. Then, we have
	\[\frac{\partial}{\partial \alpha}\tilde{H}_\alpha(\bm{x},\bm{p})=\log \left(\int_{U}F(\bm{u})\mathscr{U}(\d\bm{u})\right)-\frac{\int_{U}F(\bm{u})\log F(\bm{u})\mathscr{U}(\d\bm{u})}{\int_{U}F(\bm{u})\,\mathscr{U}(\d\bm{u})}.\]
Since the function $\phi(r):=r\log r$ is convex and $\mathscr{U}(\d\bm{u})$ is a probability measure on $U$,  the Jensen's inequality gives
	\[\left(\int_{U}F(\bm{u})\mathscr{U}(\d\bm{u})\right)\log \left(\int_{U}F(\bm{u})\mathscr{U}(\d\bm{u})\right)\le \int_{U}F(\bm{u})\log F(\bm{u})\,\mathscr{U}(\d\bm{u}).\]
Thus, $\frac{\partial}{\partial \alpha} \tilde{H}_\alpha (\bm{x},\bm{p})\le 0$ for all $(\bm{x},\bm{p})\in\bbr^{2n}$. Moreover, we already observe that the Laplace principle implies the pointwise convergence $\lim_{\alpha\to 0}\tilde{H}_{\alpha}(\bm{x},\bm{p}) = H_0(\bm{x},\bm{p})$. By the monotonic pointwise convergence of $\tilde{H}_\alpha$ to $H_0$, Dini's theorem \cite{rudin1976PMA} implies that $\tilde{H}_{\alpha}$ converges uniformly to $H_0$ as $\alpha\to 0$ on any compact subset of $\bbr^{2n}$. Finally,  the constant $\alpha\log|U|$ converges uniformly to 0 as $\alpha \to 0$. Therefore, we conclude that $H_\alpha $ converges locally uniformly to $H_0$ as $\alpha\to 0$.
\end{proof}

Next, we show that the value function $V_\alpha$ is bounded and Lipschitz continuous, uniformly in $\alpha$.

\begin{lemma}\label{lem-unif_equi}
	Suppose that Assumption \ref{assumption} holds, $U$ is compact, and $0< \alpha\le 1$. Moreover, we assume that $r$ and $q$ are locally Lipschitz continuous in $\bm{x}$, i.e., for every $R>0$ and $\bm{x},\bm{y}\in B(0,R)$, we have
	\[|r(\bm{x},\bm{u})-r(\bm{y},\bm{u})|\le C(R)|\bm{x}-\bm{y}|,\quad |q(\bm{x})-q(\bm{y})|\le C(R)|\bm{x}-\bm{y}|.\] 
	Let $V_\alpha$ be the unique viscosity solution  to \eqref{HJB_soft}. Then, for any compact subset $K$ of $\bbr^n$, 
	\[|V_\alpha(t,\bm{x})|\le C \quad  \forall (t,\bm{x})\in [0,T]\times K,\]
	\[|V_{\alpha}(t,\bm{x})-V_\alpha(s,\bm{y})|\le C(|t-s|+|\bm{x}-\bm{y}|) \quad \forall (t,\bm{x}),(s,\bm{y})\in[0,T]\times K,\]
	where the constant $C$ does not depend on $\alpha$.
\end{lemma}
\begin{proof}
Our proof is   similar to that for the standard optimal control case~\cite{Evans2010}. \\
	
	\noindent $\bullet$ (Step 1): We first prove the uniform boundedness of $V_\alpha$.
Choose the uniformly distributed constant control $\bar{g}(t,\bm{u})\equiv\frac{1}{|U|}$. Then, by the definition of $V_\alpha$,
	\begin{align}
	\begin{aligned}\label{ineq1}
	V_\alpha(t,\bm{x})&\le J^\alpha_{\bm{x},t}(\bar{g}) = \int_t^T \int_{U} r(x(s),\bm{u})\frac{1}{|U|}\d\bm{u}\,\d s -\alpha \int_t^T \mathrm{H}(\bar{g}(s,\cdot))+q(x(T))\\
	&\le TC_r+C_q+\alpha T\log|U|\le TC_r+C_q+\alpha T |\log|U||,
	\end{aligned}
	\end{align}
	where we use
	\[\mathrm{H}(\bar{g}(s,\cdot)) = -\int_U \frac{1}{|U|} \log \frac{1}{|U|} \d\bm{u} = \log|U|.\]
	On the other hand, for any control $g \in \mathcal{G}$, 
	\[\int_t^T\left(\int_{U} r(x(s),\bm{u})g(s,\bm{u})\,\d\bm{u} -\alpha \mathrm{H}(g(s,\cdot))\right)\,\d s+q(x(T))\ge -(T-t)C_r-C_q -\alpha\int_t^T \mathrm{H}(g(s,\cdot))\,\d s.\]
We also notice that
	\[\mathrm{H}(g(t,\cdot)) = -\int_U g(t,\bm{u})\log g(t,\bm{u})\d\bm{u} = -\kl(g\d\bm{u}||\mathscr{U})+\log|U|\le\log |U|.
	\]
	Therefore,
	\[\int_t^T\left(\int_{U} r(x(s),\bm{u})g(\bm{u})\,\d\bm{u} -\alpha \mathrm{H}(g(s,\cdot))\right)\,\d s+q(x(T))\ge -(T-t)C_r-C_q -\alpha (T-t)\log |U| .\]
Taking infimum of both sides with respect to $g\in\mathcal{G}$ yields
	\begin{equation}\label{ineq2}
	V_\alpha(t,\bm{x})\ge -(T-t)C_r-C_q -\alpha(T-t)\log|U| \ge -TC_r-C_q-\alpha T|\log|U||.
	\end{equation}
By \eqref{ineq1} and \eqref{ineq2}, we obtain that for $0<\alpha\le 1$,
	\[|V_\alpha(t,\bm{x})|\le TC_r+C_q+\alpha T|\log|U||\le TC_r+C_q+T|\log|U||.\]
	Note that the bound $TC_r+C_q+T|\log|U||$ does not depend on the temperature parameter $\alpha$.\\
	
	\noindent $\bullet$ (Step 2): We now prove the Lipschitz continuity of $V_\alpha$ in $\bm{x}$. Choose any $\e>0$ and $\bm{x},\bm{y}\in K$. Then, there exists $\bar{g}\in\mathcal{G}$ such that
	\[V_\alpha(t,\bm{y}) \ge\int_t^T \left(\int_{U}r(\bar{y}(s),\bm{u})\bar{g}(\bm{u})\,\d\bm{u} -\alpha \mathrm{H}(\bar{g}(s,\cdot))\right)\,\d s+q(\bar{y}(T))-\e,\]
	where $\bar{y}(s)$ is a solution to \eqref{dyn2} with control $\bar{g}$ and initial condition $\bar{y}(t)=\bm{y}$. We also let $\bar{x}(s)$ denote the solution to \eqref{dyn2} with the same control $\bar{g}$ and initial condition $\bar{x}(t)=\bm{x}$. 
By Lemma \ref{stability},	both $\bar{x}(t)$ and $\bar{y}(t)$ are bounded. Therefore, we have
	\begin{align*}
	V_{\alpha}(t,\bm{x})-V_{\alpha}(t,\bm{y})&\le \int_t^T\left(\int_{U}r(\bar{x}(s),\bm{u})\bar{g}(\bm{u})\,\d\bm{u} -\alpha \mathrm{H}(\bar{g}(s,\cdot))\right)\,\d s+q(\bar{x}(T))\\
	&\quad-\int_t^T\left(\int_{U}r(\bar{y}(s),\bm{u})\bar{g}(\bm{u})\,\d\bm{u} -\alpha \mathrm{H}(\bar{g}(s,\cdot))\right)\,\d s-q(\bar{y}(T))+\e\\
	&=\int_t^T\int_{U}(r(\bar{x}(s),\bm{u})-r(\bar{y}(s),\bm{u}))\bar{g}(\bm{u})\,\d\bm{u}\,\d s + (q(\bar{x}(T))-q(\bar{y}(T)))+\e\\
	&\le C\int_t^T |\bar{x}(s)-\bar{y}(s)|\,\d s+|\bar{x}(T)-\bar{y}(T)|+\e,
	\end{align*} 
	where the constant $C$ in the last inequality only depends on the local Lipschtiz constant of $r$ and $q$. On the other hand, again thanks to the stability estimate in Lemma \ref{stability}, we have
	\[|\bar{x}(s)-\bar{y}(s)|\le C|\bar{x}(t)-\bar{y}(t)|=C|\bm{x}-\bm{y}|,\quad t\le s\le T.\]
	Therefore, we have
	\[V_\alpha(t,\bm{x})-V_\alpha(t,\bm{y})\le C|\bm{x}-\bm{y}|+\e,\]
	where the constant $C$ depends on $f,r,q$ and $T$ but it is independent of $\alpha$. We now change the role of $\bm{x}$ and $\bm{y}$ to obtain
	\[|V_\alpha(t,\bm{x})-V_\alpha(t,\bm{y})|\le C|\bm{x}-\bm{y}|+\e.\]
	Since $\e$ was arbitrarily chosen, we conclude that $V_\alpha$ is Lipschitz continuous in $\bm{x}$.\\
	
	\noindent $\bullet$ (Step 3): Lastly, we show the Lipschitz continuity of $V_\alpha$ in $t$. Fix any $\e>0$ and $0\le t<\bar{t}\le T$. Choose $g\in \mathcal{G}$ such that
	\[V_\alpha(t,\bm{x}) \ge \int_t^T\left(\int_{U}r(x(s),\bm{u})g(s,\bm{u})\,\d\bm{u} -\alpha \mathrm{H}(g(s,\cdot))\right)\,\d s+q(x(T))-\e,\]
	where $x(s)$ is the solution to \eqref{dyn2} with $x(t)=\bm{x}$ and control $g$. We define a delayed control $\bar{g}$ as $\bar{g}(s)=g(s+t-\bar{t})$, where $\bar{t}\le s\le T$. Let $\bar{x}(s)$ be the solution to \eqref{dyn2} satisfying $\bar{x}(\bar{t})=\bm{x}$. Note that $\bar{x}(s)=x(s+t-\bar{t})$ for $\bar{t}\le s\le T$. 
We then have
	\begin{align*}
	V_\alpha&(\bar{t},\bm{x})-V_\alpha(t,\bm{x})\\
	&\le\int_{\bar{t}}^T \left(\int_{U}r(\bar{x}(s),\bm{u})\bar{g}(s,\bm{u})\,\d\bm{u} - \alpha \mathrm{H}(\bar{g}(s,\cdot))\right)\,\d s+q(\bar{x}(T))\\
	&\quad -\int_t^T\left(\int_{U}r(x(s),\bm{u})g(s,\bm{u})\,\d\bm{u} -\alpha \mathrm{H}(g(s,\cdot))\right)\,\d s-q(x(T))+\e\\
	&=\int_{\bar{t}}^T \left(\int_{U}r(x(s+t-\bar{t}),\bm{u})\bar{g}(s+t-\bar{t},\bm{u})\,\d\bm{u} - \alpha \mathrm{H}(g(s+t-\bar{t},\cdot))\right)\,\d s+q(x(T+t-\bar{t}))\\
	&\quad -\int_t^T\left(\int_{U}r(x(s),\bm{u})g(s,\bm{u})\,\d\bm{u} -\alpha \mathrm{H}(g(s,\cdot))\right)\,\d s-q(x(T))+\e\\
	&=\int_{t}^{T+t-\bar{t}} \left(\int_{U}r(x(s),\bm{u})g(s,\bm{u})\,\d\bm{u} - \alpha \mathrm{H}(g(s,\cdot))\right)\,\d s+q(x(T+t-\bar{t}))\\
	&\quad -\int_t^T\left(\int_{U}r(x(s),\bm{u})g(s,\bm{u})\,\d\bm{u} -\alpha \mathrm{H}(g(s,\cdot))\right)\,\d s-q(x(T))+\e\\
	&=-\int_{T+t-\bar{t}}^T \left(\int_{U}r(x(s),\bm{u})g(s,\bm{u})\,\d\bm{u} -\alpha \mathrm{H}(g(s,\cdot))\right)\,\d s+q(x(T+t-\bar{t}))-q(x(T))+\e.
	\end{align*}
Recall that
	\[\mathrm{H}(g(s,\cdot)) = -\kl(g\d\bm{u}||\mathscr{U})+\log|U|\le \log|U|.\] 
By the boundedness of   $x(t)$ and the local Lipschitz continuity of $r$ and  $q$, we have
	\[V_\alpha(t,\bm{x})-V_\alpha(\bar{t},\bm{x})\le C|t-\bar{t}|+\e,\]
	where the constant $C$ is independent of $\alpha$. Since  the role of $t$ and $\bar{t}$ can be switched and $\e$ was arbitrarily chosen, we conclude that
	\[|V_\alpha(t,\bm{x})-V_\alpha(\bar{t},\bm{x})|\le C|t-\bar{t}|,\]
	which implies the Lipschitz continuity of $V_\alpha$ in $t$.
\end{proof}

The previous lemma provides the boundedness and the equicontinuity of $V_\alpha$. This leads to the following local uniform convergence result for $V_\alpha$.

\begin{theorem}\label{thm:asymptotic}
	Suppose that Assumption \ref{assumption} holds, $U$ is compact and $0<\alpha\le 1$. Moreover, we assume that $r$ and $q$ are locally Lipschitz continuous in $\bm{x}$. Let $V_\alpha$ and $V_0$ be the unique viscosity solutions to \eqref{HJB_soft} and \eqref{HJB-0} respectively. Then, for any compact subset $K$ of $\bbr^n$, 
	\[V_\alpha \to V_0 \quad \mbox{uniformly on $[0,T]\times K$}.\]
\end{theorem}
\begin{proof}
Lemma \ref{lem-unif_equi} implies that the value functions $\{V_\alpha\}_{0<\alpha\le 1}$ are uniformly bounded and equicontinuous on $[0,T]\times K$ for every compact subset $K$ of $\bbr^n$. By the Arzel\`a-Ascoli theorem \cite{rudin1976PMA}, there exists a subsequence $\{\alpha_n\}$ and a limit function $V_0$ such that
	\[V_{\alpha_n}\to V_0,\quad \mbox{uniformly on $[0,T]\times K$}.\]
	Therefore, combining this with
	the uniform convergence of $H_\alpha$ to $H_0$ (Lemma~\ref{lem-hamiltonian}), we conclude that $V_0$ is  a viscosity solution to \eqref{HJB-0} \cite[Proposition 2.2 in Section II]{Bardi97}. Since  \eqref{HJB-0} has a unique viscosity solution, we further obtain that the entire sequence $\{V_\alpha\}_{0< \alpha\le 1}$ converges uniformly  to $V_0$ on $[0,T]\times K$ as $\alpha\to0$.
\end{proof}

\subsection{Infinite-Horizon Case}

In this subsection, we briefly discuss the infinite-horizon case. Consider a cost functional of the form
\[J^\alpha_{\bm x}(\mu):=\int_0^\infty e^{-\lambda t} \left(\int_{U}r(x(t),\bm{u})\d\mu-\alpha \mathrm{H}(\mu(t;\cdot))\right)\,\d t, \]
where $\lambda > 0$ is a discount factor, and $x(t)$ is the solution to \eqref{dyn2} with control $\mu$ and initial condition $x(0)={\bm x}$. We choose the set of admissible controls for the infinite-horizon problem as follows.
\begin{definition}
The set of admissible  controls $\mathcal{M}_\infty$  is defined as the set of time-dependent probability measures $\mu:[0,+\infty)\to\mathcal{P}(U)$ that satisfy the following conditions:
	\begin{enumerate}
		\item For all $(t,\bm{x})\in[0,\infty)\times\bbr^n$, we have
		\[\int_U |f(\bm{x},\bm{u})|\mu(t;\d \bm{u})<+\infty,\quad\int_U |r(\bm{x},\bm{u})|\mu(t;\d \bm{u})<+\infty.\]
		\item The  map
		\[t\mapsto \int_U |\bm{u}|\mu(t;\d {\bm u})\]
		is locally integrable, i.e., integrable over any finite time interval $[0,T]$ with $T<+\infty$.
		\item The  maps
		\[t\mapsto e^{-\lambda t}\int_U r(x(t),\bm{u})\mu(t;\d\bm{u}),\quad t\mapsto e^{-\lambda t}\mathrm{H}(\mu(t;\cdot))\]
		are integrable on $[0,+\infty)$.
	\end{enumerate}
\end{definition}
If an admissible  control $\mu\in\mathcal{M}_\infty$ is executed, the solution $x(t)$ is globally well-posed and the cost functional $J_{\bm x}^\alpha(\mu)$ is well-defined.

We first show that the infinite-horizon maximum entropy control problem has an optimal solution under some conditions similar to those in Theorem \ref{thm:exist_finite}.
 The idea of proof is almost the same as that for Theorem \ref{thm:exist_finite}, with modifications to probability measures. However, since the time horizon is now infinite, we need the boundedness of $f$ and $r$.

	\begin{theorem}\label{thm:exist_infite}
		Suppose that Assumption~\ref{assumption} holds  and the control set $U$ is compact. Moreover, we assume that $f$ and $r$ are bounded and Lipschitz continuous in $\bm{x}$.
		Then, for each $\bm{x} \in \mathbb{R}^n$,	 there exists $\mu^\star\in\mathcal{M}_\infty$ such that
		\[J^{\alpha}_{\bm{x}} (\mu^\star) = \inf_{\mu\in\mathcal{M}_\infty}J^{\alpha}_{\bm{x}} (\mu).\]
	\end{theorem}
	\begin{proof}
Fix $\bm{x}\in\bbr^n$ and let $\{\mu_i\}_{i=1}^\infty\subset \mathcal{M}_\infty$ be a sequence of admissible controls such that
		\[\lim_{i\to\infty}J_{\bm{x}}^{\alpha}(\mu_i)=\inf_{\mu\in\mathcal{M}_\infty}J_{\bm{x}}^{\alpha}(\mu).\]
		Thus, the costs $J_{\bm{x}}^{\alpha}(\mu_i)$ are bounded.
Using a cost reformulation  similar to \eqref{cost_kl}, we deduce that
		\[\left|\int_0^\infty e^{-\lambda t}\left(\int_{U}r(x_i(t),\bm{u})\,\mu_i(t;\d \bm{u}) +\alpha \kl(\mu_i(t,\cdot)||\mathscr{U})\right)\,\d t-\frac{\alpha}{\lambda}\log|U|\right|<C\]
		for some constant $C$ independent of $i$, 
		where $x_i$ denotes the solution to \eqref{dyn2} with control $\mu_i$. Then, we can uniformly bound the following integral of the KL-divergence term:
		\[\left|\int_0^\infty e^{-\lambda t}\kl(\mu_i(t;\cdot)||\mathscr{U})\,\d t\right|< \frac{1}{\alpha}\left(C+\frac{\alpha}{\lambda}|\log|U||+\frac{\alpha C_r}{\lambda}\right).\]
As in the proof of Theorem \ref{thm:exist_finite}, we introduce the following probability measure on $[0,+\infty)\times U$ for each $\mu_i$:
		\[\nu_i(\d t,\d \bm{u}):= \frac{e^{-\lambda t}}{\lambda} \mu_i(t;\d \bm{u})\d t\]
		and let $\mathscr{U}_\infty(\d t,\d \bm{u}):=\frac{e^{-\lambda t}}{\lambda|U|}\d t \d\bm{u}$. Then, the KL-divergence from $\mathscr{U}_\infty$ to $\nu_i$ is uniformly bounded by a constant $M$:
		\begin{align*}
		\left|\kl(\nu_i||\mathscr{U}_\infty)\right|&=\left|\int_0^\infty \int_U \frac{\frac{e^{-\lambda t}}{\lambda}\mu_i(t;\d \bm{u})}{\frac{e^{-\lambda t}}{\lambda|U|}\d\bm{u}}\log\left(\frac{\frac{e^{-\lambda t}}{\lambda}\mu_i(t;\d \bm{u})}{\frac{e^{-\lambda t}}{\lambda|U|}\d\bm{u}}\right)\frac{e^{-\lambda t}}{\lambda|U|}\d\bm{u} \d t\right|\\
		&=\left|\int_0^\infty\int_U \frac{\mu_i(t;\d \bm{u})}{\frac{1}{|U|}\d\bm{u}}\log\left(\frac{\mu_i(t;\d \bm{u})}{\frac{1}{|U|}\d\bm{u}}\right)\frac{e^{-\lambda t}}{\lambda|U|}\d\bm{u} \d t\right|\\
		&=\frac{1}{\lambda}\left|\int_0^\infty e^{-\lambda t} \kl(\mu_i(t;\cdot)||\mathscr{U})\d t\right|\le \frac{1}{\alpha \lambda}\left(C+\frac{\alpha}{\lambda}|\log|U||+\frac{\alpha C_r}{\lambda}\right)=:M.
		\end{align*}
		Hence, by the argument in the proof of Theorem \ref{thm:exist_finite}, there exists a subsequence $\{\nu_{i_k}\}_{k=1}^\infty$ such that $\nu_{i_k}\xrightharpoonup{*} \nu\in \mathcal{P}([0,+\infty)\times U)$, and there exists $\rho:[0,\infty)\to\mathcal{P}(U)$ such that $\nu(\d t,\d \bm{u}) = \frac{e^{-\lambda t}}{\lambda}\rho(t;\d \bm{u})\,\d t$. To avoid overload in notation, we let $\{\nu_i\}$ denote the subsequence $\{\nu_{i_k}\}$ from now on.
		
		We now show that $\rho$ minimizes $J_{\bm{x}}^\alpha$ over $\mathcal{M}_\infty$. We notice that
		\begin{align*}
		\inf_{\mu\in\mathcal{M}}J_{\bm{x}}^{\alpha}(\mu)&=\liminf_{i\to \infty}\left\{\int_0^\infty e^{-\lambda t} \int_{U} r(x_i(t),\bm{u})\,\mu_i(t;\d \bm{u})\d t +\frac{\alpha}{\lambda} \kl(\nu_i||\mathscr{U}_\infty)\right\}-\frac{\alpha}{\lambda}\log|U|\\
		&\ge\liminf_{i\to\infty}\left\{\int_0^\infty e^{-\lambda t}\int_{U} r(x_i(t),\bm{u})\mu_i(t;\d \bm{u})\d t\right\} +\frac{\alpha}{\lambda}\liminf_{i\to\infty} \kl(\nu_i||\mathscr{U}_\infty)-\frac{\alpha}{\lambda}\log|U|\\
		&\ge \liminf_{i\to\infty}\left\{\int_0^\infty e^{-\lambda t}\int_{U} r(x_i(t),\bm{u})\mu_i(t;\d \bm{u})\d t\right\} +\frac{\alpha}{\lambda} \kl(\nu||\mathscr{U}_\infty)-\frac{\alpha}{\lambda}\log|U|,
		\end{align*}
		where the lower semi-continuity of the KL-divergence is used. Thus, it suffice to show that
		\[\left|\int_0^\infty e^{-\lambda t} \int_U r(x_i(t),\bm{u})\mu_i(t;\d \bm{u})\d t-\int_0^\infty e^{-\lambda t} \int_U r(x(t),\bm{u})\rho(t;\d \bm{u})\d t\right|\to 0\quad \mbox{as $i\to +\infty$},\]
		where $x(t)$ denotes the system state of \eqref{dyn2} when control $\rho$ is employed. By the triangle inequality, 
		\begin{align*}
		&\left|\int_0^\infty e^{-\lambda t} \int_U r(x_i(t),\bm{u})\mu_i(t;\d \bm{u})\d t-\int_0^\infty e^{-\lambda t} \int_U r(x(t),\bm{u})\rho(t;\d \bm{u})\d t\right|\\
		&\qquad \le \left|\int_0^T e^{-\lambda t} \int_U r(x_i(t),\bm{u})\mu_i(t;\d \bm{u})\d t-\int_0^T e^{-\lambda t} \int_U r(x(t),\bm{u})\rho(t;\d \bm{u})\d t\right|\\
		&\qquad\quad +\left|\int_T^\infty e^{-\lambda t} \int_U r(x_i(t),\bm{u})\mu_i(t;\d \bm{u})\d t-\int_T^\infty e^{-\lambda t} \int_U r(x(t),\bm{u})\rho(t;\d \bm{u})\d t\right|=:\mathcal{I}_1+\mathcal{I}_2.
		\end{align*}
As in Theorem \ref{thm:exist_finite} for any fixed $T>0$,  $\mathcal{I}_1\to0$ as $i\to\infty$. On the other hand, $\mathcal{I}_2$ can be estimated as
		\[\mathcal{I}_2 \le 2C_r\int_T^\infty e^{-\lambda t} \d t=\frac{2C_r}{\lambda}e^{-\lambda T}.\]
		Therefore, for any $\e>0$, there exists $T=T(\e)$ such that $\frac{2C_r}{\lambda}e^{-\lambda T}<\frac{\e}{2}$. For such fixed $T(\e)$, there exists an index $N(\e)$ such that
		\[\left|\int_0^{T(\e)} e^{-\lambda t} \int_U r(x_i(t),\bm{u})\mu_i(t;\d \bm{u})\d t-\int_0^{T(\e)} e^{-\lambda t} \int_U r(x(t),\bm{u})\rho(t;\d \bm{u})\d t\right|<\frac{\e}{2}\quad \mbox{for}\quad i>N(\e).\]
		In conclusion, for any $\e>0$, there exists an index $N(\e)$ such that for $i>N(\e)$,
		\[\left|\int_0^\infty e^{-\lambda t} \int_U r(x_i(t),\bm{u})\mu_i(t;\d \bm{u})\d t-\int_0^\infty e^{-\lambda t} \int_U r(x(t),\bm{u})\rho(t;\d \bm{u})\d t\right|<\e,\]
		which implies the desired convergence property. Thus, we finally have
		\begin{align*}
		\inf_{\mu\in\mathcal{M}} J_{\bm{x}}^\alpha(\mu) &\ge \int_0^\infty e^{-\lambda t}\int_U r(x(t),\bm{u})\rho(t;\d \bm{u})\d t +\frac{\alpha}{\lambda}\kl(\nu||\mathscr{U}_{\infty})-\frac{\alpha}{\lambda}\log|U|\\
		&=\int_0^\infty e^{-\lambda t} \left(\int_U r(x(s),\bm{u})\rho(t;\d \bm{u})+\alpha \kl(\rho(t;\cdot)||\mathscr{U})\right)\d t-\frac{\alpha}{\lambda}\log|U|\\
		&=J_{\bm{x}}^\alpha(\rho).
		\end{align*}
	\end{proof}

As emphasized in Remark~\ref{rem:den}, we can consider  density functions $g$ as control variables instead of measures $\mu$.
Let the set $\mathcal{G}_\infty$ be defined by
\[\mathcal{G}_\infty:=\left\{g:[0,+\infty)\to L^1_{+,1}(U)~\mid~ g\d\bm{u} \in \mathcal{M}_\infty\right\}\]
and consider  the  cost functional 
\[J^\alpha_{\bm{x}}(g):=\int_0^\infty e^{-\lambda t}\left(\int_{U}r(x(t), \bm{u})\,g(t,\bm{u})\,\d\bm{u}-\alpha \mathrm{H}(g(t, \cdot))\right)\,\d t,\quad x(0)= \bm{x}.\]
The corresponding value function is given by
 $V_\alpha(\bm{x}):= \inf_{g\in \mathcal{G}_\infty} J^\alpha_{\bm x}(g)$. 
 By the dynamic programming principle, we  obtain the following lemma.
\begin{lemma}
	The value function $V_{\alpha}$ for the infinite horizon problem satisfies
	\[V_{\alpha}(\bm{x}) = \inf_{g\in \mathcal{G}_\infty} \left\{\int_0^h \left(\int_U r(x(t),\bm{u})g(t,\bm{u})\d\bm{u}-\alpha \mathrm{H}(g(t,\cdot))\right)\,\d t +e^{-\lambda h} V_{\alpha}(x(h))\right\},\]
	where $x(t)$ is the solution to \eqref{dyn2} with initial condition $x(0)=\bm{x}$.
\end{lemma}
Since the proof of this lemma is almost identical to that of Lemma \ref{dpp}, we have omitted the proof.

Using this lemma, we can formally derive the following soft HJB equation for the infinite-horizon problem:
\begin{equation}\label{HJB-inf}
\lambda V_\alpha+H_\alpha(\bm{x},\nabla_{\bm{x}}V_\alpha)=0 \quad \mbox{on } \mathbb{R}^n,
\end{equation}
where the soft Hamiltonian, given by
\begin{equation}\nonumber
 H_\alpha(\bm{x},\bm{p})=\alpha\log\int_{U}\exp\left(-\frac{\bm{p}\cdot f(\bm{x},\bm{u})+r(\bm{x},\bm{u})}{\alpha}\right)\,\d\bm{u},
\end{equation}
is identical to that in the finite-horizon case. 
Recall the following standard definition of viscosity solutions to stationary HJ equations: 
\begin{definition}
	A continuous function $V: \mathbb{R}^n \to \mathbb{R}$ is a viscosity solution of \eqref{HJB-inf} provided that
	\begin{enumerate}
		\item (Subsolution) For any $\phi\in C^1(\bbr^n)$ such that $V-\phi$ has a local maximum at $\bm{x}_0$,
		\[\lambda V(\bm{x}_0) +H_{\alpha}(\bm{x}_0,\nabla_{\bm{x}}\phi(\bm{x}_0))\le 0.\]
		\item (Supersolution) For any $\phi\in C^1(\bbr^n)$ such that $V-\phi$ has a local minimum at $\bm{x}_0$,
		\[\lambda V(\bm{x}_0) +H_{\alpha}(\bm{x}_0,\nabla_{\bm{x}}\phi(\bm{x}_0))\ge 0.\]
	\end{enumerate}
\end{definition}

We can show that the soft HJB equation \eqref{HJB-inf} has a unique viscosity solution, which coincides with the value function $V_\alpha$ of the infinite-horizon problem. Since the proof is almost the same as that for the finite-horizon case, we have omitted the proof.
\begin{theorem}
	Suppose that Assumption \ref{assumption} holds  and  the control set $U$ is compact. Then, the value function $V_\alpha$ is the unique viscosity solution of the HJB equation \eqref{HJB-inf}.
\end{theorem}

The conditions for optimality  can be obtained using almost the same argument as in the finite horizon case. In particular, optimal controls can be constructed as $\Phi_i(\bm{x})\in S_i(\bm{x})$,\footnote{When using $S_2$, we further assume that $D^\pm V_\alpha(\bm{x})\neq\emptyset$ for all $\bm{x} \in \mathbb{R}^n$.} 
where
\[ S_1(\bm{x}):=\Bigg\{\phi\in L^1_{+,1}(U)~:~\partial^-V_\alpha\left(\bm{x};\int_{U}f(\bm{x},\bm{u})\phi(\bm{u})\d \bm{u}\right)+\int_U r(\bm{x},\bm{u})\phi(\bm{u})\d\bm{u}-\alpha \mathrm{H}(\phi(\cdot))\le\lambda V_\alpha(\bm{x}) \Bigg\},\] and  \[S_2(\bm{x}):=\Bigg\{\phi\in L^1_{+,1}(U)~:~\phi(\bm{u})=\frac{\exp\left(-\frac{1}{\alpha}(\bm{p}_x \cdot f(\bm{x},\bm{u})+r(\bm{x},\bm{u}))\right)}{\int_U \exp\left(-\frac{1}{\alpha}(\bm{p}_x \cdot f(\bm{x},\bm{u})+r(\bm{x},\bm{u}))\right)\d \bm{u}} \mbox{ for } \bm{p}_x\in D^\pm V_\alpha(\bm{x})\Bigg\}.\]

\section{Tractable Methods for Maximum Entropy Optimal Control}\label{sec:4}

A salient feature of the maximum entropy  control problem 
is its tractability.
In this section, we discuss tractable methods based on the soft HJB equation.

\subsection{Control-Affine Systems}

We consider a control-affine system with 
$f(\bm{x},\bm{u}) = f_1(\bm{x})+f_2(\bm{x})\bm{u}$
and a cost function of the form $r(\bm{x},\bm{u}) = r_1(\bm{x}) +\frac{1}{2}\bm{u}^\top R \bm{u}$. Here, we assume that $f_1(0)=0$, $R$ is a symmetric positive definite matrix, and
$r_1(\bm{x})$ is a  positive definite function, i.e., $r_1(\bm{x})\ge 0$ for all $\bm{x} \in \mathbb{R}^n \setminus \{0\}$ and $r_1(0)=0$.\footnote{The positive definiteness of $r_1$ is needed for the asymptotical stability of the closed-loop system with optimal $g_\alpha^\star$ in Theorem~\ref{T4.1}.}
 As in the standard optimal control problem for control-affine systems, we set $U=\bbr^{m}$. 
The maximum entropy   control problem for  control-affine systems can then be formulated as
 \begin{equation}\label{max_ent_CA}
 \min_{g\in \mathcal{G}_\infty}\int_0^\infty e^{-\lambda t}\left(\int_U \left(r_1(x(t)) +\frac{1}{2}\bm{u}^\top R \bm{u}\right)g(t,\bm{u})\d\bm{u}-\alpha \mathrm{H}(g(t,\cdot))\d t\right),
 \end{equation}
 where $x(t)$ is the solution to
 \[\dot{x}(t) = \int_{U} (f_1(x(t))+f_2(x(t))\bm{u})g(t,\bm{u})\d \bm{u},\quad x(0)=\bm{x}.\]
 Since the control set $U$ is not compact in this case, we cannot directly use the theory in Section \ref{sec:3}. 
 Instead, we consider a different method to obtain an optimal control and compare it with the theory of soft HJB equations in Section \ref{sec:3}. 
The following theorem indicates that the optimal control is uniquely characterized as a normal distribution.
Furthermore, its mean corresponds to the optimal control for the standard problem without an entropy term.

 \begin{theorem}\label{T4.1}
 	Suppose that $V_0$ is a   unique $C^1$-positive definite solution  to the following HJB equation:
 	 \begin{equation}\label{HJB0_CA}
 	 \lambda V_0+ \frac{1}{2}(\nabla_{\bm{x}} V_0)^\top f_2(\bm{x})R^{-1}f_2({\bm{x}})^\top(\nabla_{\bm{x}} V_0) -r_1(\bm{x}) -(\nabla_{\bm{x}}V_0)^\top f_1(\bm{x})=0.
 	 \end{equation}
 	 Then, the optimal control $g^\star_\alpha\in \mathcal{G}_\infty$ for the maximum entropy optimal control problem for the control-affine system \eqref{max_ent_CA} is uniquely given as the probability density function of the normal distribution $\mathcal{N}(-R^{-1}f_2(\bm{x})^\top \nabla_{\bm{x}}V_0(\bm{x}),\alpha R^{-1})$, i.e.,
 	 \[g^\star_\alpha(\bm{x},\bm{u}):=\sqrt{\frac{\det{R}}{(2\pi \alpha)^m}}\exp\left(-\frac{1}{2\alpha}(\bm{u}+R^{-1}f_2(\bm{x})^\top \nabla_{\bm{x}}V_0(\bm{x}))^\top R(\bm{u}+R^{-1}f_2(\bm{x})^\top \nabla_{\bm{x}}V_0(\bm{x}))\right).\]
 	 Note that $V_0$ is the (optimal) value function of the standard optimal control problem.
 \end{theorem}

 \begin{proof}
 	We split the proof into two steps. First, we show that   $g_\alpha^\star$ is given as the probability density of a normal distribution. Then, we confirm that the mean and the covariance of the normal distribution have the desired form.\\	
 \noindent $\bullet$ (Step 1) : We first highlight that, in the control-affine case, the system \eqref{dyn2} (with $\mu(t;\d\bm{u}) = g(t,\bm{u})\d \bm{u}$) can be simplified as
 \begin{align*}
 \dot{x}(t) &= \int_U (f_1(x(t))+f_2(x(t))\bm{u})g(t,\bm{u}) \d \bm{u} = f_1(x(t)) +f_2(x(t)) \int_U \bm{u}g(t,\bm{u})\d \bm{u} \\
 &= f_1(x(t))+f_2(x(t))\mathbb{E}[g],
 \end{align*}
 where $\mathbb{E}[g]$ denotes the mean of the probability density $g$. Therefore, the system dynamics is completely determined by the mean of $g$. Moreover, the running cost is expressed as
 \begin{align*}
 \int_U r(\bm{x},\bm{u})g(t,\bm{u})\,\d \bm{u} &= \int_U (r_1(\bm{x})+\frac{1}{2}\bm{u}^\top R \bm{u})g(t,\bm{u}) \d u=r_1(\bm{x}) +\frac{1}{2}\int_U (\bm{u}^\top R \bm{u}) g(t,\bm{u})\d \bm{u}\\
 &=r_1(\bm{x}) +\frac{1}{2}\int_U ((\bm{u}-\mathbb{E}[g])^\top R (\bm{u}-\mathbb{E}[g]))g(t,\bm{u})\d \bm{u}+\frac{1}{2}\mathbb{E}[g]^\top R\mathbb{E}[g]\\
 &=r_1(\bm{x}) +\frac{1}{2}\int_U \mbox{tr}((\bm{u}-\mathbb{E}[g])^\top R (\bm{u}-\mathbb{E}[g]))g(t,\bm{u}) \d \bm{u} +\frac{1}{2}\mathbb{E}[g]^\top R\mathbb{E}[g]\\
 &=r_1(\bm{x}) +\frac{1}{2}\int_U \mbox{tr}(R (\bm{u}-\mathbb{E}[g])(\bm{u}-\mathbb{E}[g])^\top) g(t,\bm{u}) \d \bm{u}+\frac{1}{2}\mathbb{E}[g]^\top R\mathbb{E}[g]\\
 &=r_1(\bm{x}) +\frac{1}{2}\mbox{tr}(R\mbox{Cov}[g])+\frac{1}{2}\mathbb{E}[g]^\top R\mathbb{E}[g],
 \end{align*}
 where $\mbox{Cov}[g]$ denotes the covariance of the probability density $g$. Hence, for the control-affine system with the quadratic control cost, the system dynamics as well as the running cost are determined by the mean and the covariance  of the probability density. Therefore, to minimize the total cost, the optimal control $g_\alpha^\star$ should maximize the entropy when the mean and the covariance are fixed. On the other hand, it is well-known that the normal distribution has the maximum differential entropy among all probability distributions with the same mean and covariance. Therefore, the optimal control should be given as a normal distribution. \\
 \noindent $\bullet$ (Step 2): Now, suppose that the optimal control $g^\star_\alpha(t)$ follows the normal distribution $\mathcal{N}(c(t),\Sigma(t))$. Since the entropy of the normal distribution $\mathcal{N}(c,\Sigma)$ is given by $\frac{1}{2}\log\det(2\pi e\Sigma)$, the cost can be explicitly calculated as
 \[\int_U r(\bm{x},\bm{u})g^\star_\alpha(t,\bm{u}) \d \bm{u} -\alpha \mathrm{H}(g(t,\cdot)) = r_1(\bm{x})+\frac{1}{2}\mbox{tr}(R\Sigma(t))+\frac{1}{2}c(t)^\top R c(t) -\frac{\alpha}{2}\log\det(\Sigma(t))-\frac{\alpha m}{2}\log(2\pi e).\]
Note that the system trajectory $x(t)$ does not depend on the covariance matrix. Thus, an optimal covariance matrix $\Sigma(t)$ minimizes
 \[\frac{1}{2}\mbox{tr}(R\Sigma(t))-\frac{\alpha}{2}\log\det(\Sigma(t))\]
 for all $t > 0$.
 To find such $\Sigma$, let $\tilde{\Sigma} = R\Sigma$. Then, we need to find $\tilde{\Sigma}$ that minimizes
 \[  \mbox{tr}(\tilde{\Sigma})-\alpha\log\det(R^{-1}\tilde{\Sigma})= \mbox{tr}(\tilde{\Sigma})-\alpha\log\det(\tilde{\Sigma}) .\]
 Since the trace and the determinant of $\tilde{\Sigma}$ solely depend on the set of eigenvalues, the problem is equivalent to find the set of eigenvalues $\{\lambda_1,\ldots, \lambda_m\}$ of $\tilde{\Sigma}$ that minimizes
 \[\mbox{tr}(\tilde{\Sigma})-\alpha\log\det(\tilde{\Sigma})=\sum_{i=1}^m \lambda_i -\alpha \log \prod_{i=1}^m\lambda_i=\sum_{i=1}^m(\lambda_i-\alpha \log \lambda_i).\]
 This quantity is minimized when $\lambda_i=\alpha$ for all $i=1,2,\ldots, m$. Therefore, $\tilde{\Sigma}$ should be equal to $\alpha I$, and thus $\Sigma = \alpha R^{-1}$. Now, the optimal control problem is reduced to find $c(t)$ that minimizes
 \[\int_0^{\infty} e^{-\lambda t}\left(r_1(x(t)) +\frac{1}{2}c(t)^\top Rc(t) +C\right)\d t,\]
 where $C$ is a constant and $x(t)$ satisfies $\dot{x}(t) = f_1(x(t))+f_2(x(t))c(t)$. However, this is equivalent to the standard optimal control problem with the control-affine system and the quadratic cost. Therefore, if $V_0$ is a $C^1$-positive definite solution to the HJB equation \eqref{HJB0_CA}, it is  unique in the same class~\cite{jiang2014robust}. Furthermore, $c(t)$ should be given as the optimal control of the standard version, which is uniquely given as $-R^{-1}f_2(x(t))^\top \nabla_{\bm{x}} V_0$~\cite{lewis2012optimal}. 
 Note that the control $c(t)$ stabilizes the system  by the positive definiteness of $r_1$.
 In conclusion, the optimal control $g^\star_\alpha\in\mathcal{G}_\infty$ is uniquely given as the normal distribution with mean 
 $c(t) = -R^{-1}f_2(x(t))^\top \nabla_{\bm{x}}V_0(x(t))$ and covariance matrix $\Sigma(t)=\alpha R^{-1}$.
 \end{proof}

Using the optima control $g^\star_\alpha$ identified in Theorem~\ref{T4.1}, we can directly derive the soft HJB equation for the control-affine case. 
 \begin{proposition}\label{P:ca}
 	Suppose that $V_0\in C^1(\bbr^n)$ is a solution of the  HJB equation \eqref{HJB0_CA}. Then, the value function $V_\alpha$ for the maximum entropy optimal control problem \eqref{max_ent_CA} is given by 
 	\[V_\alpha(\bm{x}) = V_0(\bm{x})-\frac{\alpha}{2\lambda}\log\frac{(2\pi\alpha)^m}{\det R}.\]
 	Moreover, it satisfies the following HJB equation in the classical sense (and hence in the  sense of viscosity solutions):
 	 \[\lambda V_\alpha +\frac{1}{2}(\nabla_{\bm{x}}V_\alpha)^\top f_2(\bm{x}) R^{-1} f_2(\bm{x})^\top(\nabla_{\bm{x}}V_\alpha)-r_1(\bm{x})-(\nabla_{\bm{x}}V_\alpha)^\top f_1(\bm{x})+\frac{\alpha}{2}\log\frac{(2\pi\alpha)^m}{\det R}=0.\]
 	 We note that this is exactly the same as the soft HJB equation \eqref{HJB-inf} for the control-affine case.
 \end{proposition}

 \begin{proof}
 First, we simplify the cost term as follows:
 \begin{align*}
 \int_U &r(x(t),\bm{u})g^\star_\alpha(t,\bm{u})\d \bm{u}-\alpha \mathrm{H}(g(t,\cdot))\\
  &= r_1(x(t))+\frac{1}{2}c(t)^\top Rc(t)-\frac{m\alpha }{2}\log(2\pi e) +\frac{1}{2}\mbox{tr}(\alpha I) -\frac{\alpha}{2}\log \det(\alpha R^{-1})\\
 &=r_1(x(t))+\frac{1}{2}c(t)^\top Rc(t)-\frac{m\alpha }{2}\log(2\pi e) +\frac{m\alpha}{2}-\frac{\alpha}{2}\log\frac{\alpha^m}{\det R}\\
 &= r_1(x(t))+\frac{1}{2}c(t)^\top Rc(t)-\frac{\alpha}{2}\log\frac{(2\pi\alpha)^m}{\det R},
 \end{align*}
 where $c(t)=-R^{-1}f_2(x(t))^\top\nabla_{\bm{x}}V_0(x(t))$ is the same as in the proof of Theorem \ref{T4.1}. Therefore, the value function $V_\alpha$ can be explicitly calculated as
  \begin{align*}
 V_{\alpha}(\bm{x}) &= J_{\bm{x}}^\alpha(g^\star_\alpha) = \int_0^\infty e^{-\lambda t}\left(\int_U r(x(t),\bm{u})g^\star_\alpha(t,\bm{u})\d\bm{u}-\alpha \mathrm{H}(g^\star_\alpha(t,\cdot))\right)\d t \\
 &=\int_0^\infty e^{-\lambda t} \left(r_1(x(t))+\frac{1}{2}c(t)^\top Rc(t)-\frac{\alpha}{2}\log\frac{(2\pi\alpha)^m}{\det R}\right)\d t\\
 &=\int_0^\infty e^{-\lambda t} \left(r_1(x(t))+\frac{1}{2}c(t)^\top R c(t)\right)\d t -\frac{\alpha}{2\lambda}\log\frac{(2\pi\alpha)^m}{\det R}=V_0(\bm{x}) -\frac{\alpha}{2\lambda}\log\frac{(2\pi\alpha)^m}{\det R},
 \end{align*} 
 where the last equality comes from the fact that $c(t)$ is an optimal control for the standard optimal control problem without an entropy term.  Thus, if $V_0$ is a $C^1$-solution of the HJB equation \eqref{HJB0_CA}, then $V_\alpha$ is a $C^1$-solution of the following HJB equation:
 \[\lambda V_\alpha +\frac{1}{2}(\nabla_{\bm{x}}V_\alpha)^\top f_2(\bm{x}) R^{-1} f_2(\bm{x})^\top(\nabla_{\bm{x}}V_\alpha)-r_1(\bm{x})-(\nabla_{\bm{x}}V_\alpha)^\top f_1(\bm{x})+\frac{\alpha}{2}\log\frac{(2\pi\alpha)^m}{\det R}=0.\]
 
 \end{proof}
 
We now provide several quantitative comparisons between the standard and the maximum entropy optimal control problems for the control-affine case. We can interpret the optimal control $R^{-1}f_2(\bm{x})^\top \nabla_{\bm{x}}V_0(\bm{x})$ for the standard version as a Dirac measure $\mu_0^\star =\delta(\bm{u}+R^{-1}f_2(\bm{x})^\top \nabla_{\bm{x}}V_0(\bm{x}))$.  
Considering the Dirac measure as a normal distribution with zero covariance matrix, the difference between $\mu^\star_\alpha=g^\star_\alpha \d\bm{u}$ and $\mu^\star_0$ can be measured by the 2-Wasserstein distance as
 \[W_2(\mu_\alpha^\star,\mu_0^\star)^2 = \mbox{tr}(\alpha R^{-1}) = \alpha\mbox{tr}(R^{-1}).\]
 Therefore, as $\alpha \to 0$, the convergence of $\mu^\star_\alpha$ to $\mu_0^\star$ is of order $\alpha^{\frac{1}{2}}$. Next, we quantify the effect of the entropy term. Recall that the entropy of the normal distribution $\mathcal{N}(c,\Sigma)$ is $\frac{1}{2}\log\det(2\pi e\Sigma)$, we have
 \[\mathrm{H}(g_\alpha^\star(t,\cdot)) =\frac{1}{2}\log\det(2\pi e \alpha R^{-1}) = \frac{1}{2}\log\left(\frac{(2\pi e\alpha)^m}{\det R}\right)=\frac{1}{2}\log\left(\frac{(2\pi\alpha)^m}{\det R}\right)+\frac{m}{2}.\]
 Therefore, the total entropy in the cost functional $J^\alpha_{\bm{x}}(g_\alpha^\star)$ is equal to
 \[\int_0^\infty e^{-\lambda t}\alpha \mathrm{H}(g_\alpha^\star(t,\cdot))\,\d t=\frac{\alpha}{\lambda}\left(\frac{1}{2}\log\left(\frac{(2\pi\alpha)^m}{\det R}\right)+\frac{m}{2}\right).\]
The pure running cost without the entropy is  then given by
 \[V_\alpha(\bm{x}) +\alpha \int_0^\infty e^{-\lambda t} \mathrm{H}(g_\alpha^\star)\d t = V_0(\bm{x}) -\frac{\alpha}{2\lambda}\log\frac{(2\pi\alpha)^m}{\det R}+\frac{\alpha}{\lambda}\left(\frac{1}{2}\log\left(\frac{(2\pi\alpha)^m}{\det R}\right)+\frac{m}{2}\right)=V_0(\bm{x}) +\frac{m\alpha}{2\lambda}. \]
 Therefore, when using the maximum entropy method,
  the pure optimal running cost is increased by $\frac{m\alpha}{2\lambda}$ compared to the standard optimal cost  $V_0(\bm{x})$. 
  The cost difference is proportional to the temperature parameter $\alpha$ and the input dimension $m$.

\subsection{Linear-Quadratic Problems}\label{sec:lq}
As a special case of the previous subsection, we consider a linear-quadratic problem with $f(\bm{x},\bm{u}) = A\bm{x}+B\bm{u}$ and $r(\bm{x},\bm{u})= \frac{1}{2}\bm{x}^\top Q\bm{x} +\frac{1}{2}\bm{u}^\top R \bm{u}$. It is well known that the value function $V_0$ of the standard linear-quadratic problem without an entropy term can be expressed as $V_0(\bm{x}) = \frac{1}{2}\bm{x}^\top P\bm{x}$, where $P$ is a symmetric positive definite solution to the following algebraic Riccati equation (ARE):
\begin{equation}\label{ARE-0}
	\lambda P+P BR^{-1}B^\top P-Q-PA-A^\top P=0.
\end{equation}
Using Proposition~\ref{P:ca}, the value function $V_\alpha$ of the maximum entropy linear quadratic problem is directly obtained as
\[V_\alpha(\bm{x}) = \frac{1}{2}\bm{x}^\top P \bm{x}-\frac{\alpha}{2\lambda}\log\frac{(2\pi\alpha)^m}{\det R}\]
and the optimal   control is uniquely given as
\[g_\alpha^\star(\bm{x},\bm{u}):=\sqrt{\frac{\det{R}}{(2\pi \alpha)^m}}\exp\left(-\frac{1}{2\alpha}(\bm{u}+R^{-1}B^\top P\bm{x})^\top R(\bm{u}+R^{-1}B^\top P\bm{x})\right).\]

We compare our results for the LQ problem with the results for its stochastic counterpart in \cite{wang2019exploration}. When all the coefficients for the stochastic terms in \cite{wang2019exploration} are zero and the cost function is purely quadratic in state and control variables, the value function and the optimal control in \cite[Theorem 4]{wang2019exploration} are equivalent to $V_\alpha$ and $g_\alpha^\star$, respectively.

We further observe that the pure running cost is calculated as
\[\int_0^\infty e^{-\lambda t} r(x(t),\bm{u}) g_\alpha^\star(x(t),\bm{u})\d\bm{u} = V_{\alpha}(\bm{x})+\alpha\int_0^\infty e^{-\lambda t}\mathrm{H}(g_{\alpha}^\star)\d t =V_0(\bm{x})+\frac{m\alpha}{2\lambda} = \frac{1}{2}\bm{x}^\top P\bm{x} +\frac{m\alpha}{2\lambda},\]
which is consistent with the previous result for stochastic systems \cite[Theorem 8]{wang2019exploration}. Therefore, our results on the maximum entropy linear-quadratic problem are consistent with the results regarding its stochastic counterpart in \cite{wang2019exploration}. Furthermore, we extend those results to the more general setting of control-affine systems.

\subsection{Generalized Hopf--Lax Formula}\label{sec:4.3}

As discussed in Section~\ref{sec:opt_con},
constructing an optimal control $g^\star_\alpha$ requires the viscosity solution of the soft HJB equation \eqref{HJB_soft}. 
Since a general HJ equation does not have a closed-form solution,
it is typical to use numerical methods such as finite-difference methods~\cite{crandall1984two,osher1991high}. However, the computational cost for grid-based methods  grows exponentially with the dimension of the state space, making them impractical even with six- or seven-dimensional state spaces. 
Therefore, there have been some efforts to find  simple representations of the viscosity solution of HJ equations. 
When the Hamiltonian does not depend on the state variable, there is a well-known Hopf--Lax formula \cite{hopf1965generalized}, which can be used for computing the solution without discretizing the state space. In a series of recent works \cite{chow2017algorithm,chow2019algorithm,darbon2016algorithms}, the Hopf--Lax formula has been generalized to handle a larger class of HJ equations. 
 We use the generalized Hopf--Lax formula for the state-dependent Hamiltonian proposed in \cite{chow2019algorithm} to solve the soft HJB equation \eqref{HJB_soft} in a grid-free manner. 

To begin with, we  reformulate the terminal value problem \eqref{HJB_soft} into the initial value problem by letting $W_\alpha(t,\bm{x}) := V_\alpha(T-t,\bm{x})$. Then, $W_\alpha$ is the unique viscosity solution of the initial value problem
\begin{equation}\label{HJB-convert}
	\partial_tW_\alpha +H_\alpha(\bm{x},\nabla_{\bm{x}} W)=0,\quad W(0,\bm{x})=q(\bm{x}).
\end{equation}
The solution $W_\alpha(t,\bm{x})$  can be written as one of the following representations \cite{chow2019algorithm}:
\begin{align}
	\begin{aligned}\label{grid_free1}
		W_\alpha(t,\bm{x}) &= \min_{\bm{v}\in\bbr^n} \bigg\{q(\gamma(\bm{x},\bm{v},0))\\
		&\hspace{1cm}+\int_0^t \left[p(\bm{x},\bm{v},s)\cdot \nabla_{\bm{p}}H_\alpha(\gamma(\bm{x},\bm{v},s),p(\bm{x},\bm{v},s))-H_\alpha(\gamma(\bm{x},\bm{v},s),p(\bm{x},\bm{v},s))\right]\,\d s\bigg\},
	\end{aligned}
\end{align}
or
\begin{align}
	\begin{aligned}\label{grid_free2}
		W_\alpha(t,\bm{x}) &= \max_{\bm{v}\in\bbr^n} \bigg\{\bm{x}\cdot \bm{v}-q^*(p(\bm{x},\bm{v},0))\\
		&\hspace{1cm}-\int_0^t \left[H_\alpha(\gamma(\bm{x},\bm{v},s),p(\bm{x},\bm{v},s))-\gamma(\bm{x},\bm{v},s)\cdot \nabla_{\bm{x}}H_\alpha(\gamma(\bm{x},\bm{v},s),p(\bm{x},\bm{v},s))\right]\,\d s\bigg\},
	\end{aligned}
\end{align}
where $\gamma(\bm{x},\bm{v},s)$ and $p(\bm{x},\bm{v},s)$ are given as the solution to the following characteristic ODEs:\footnote{The bi-characteristic curves $(\gamma,p)$ in \eqref{bichar} may exist only local-in-time for a Hamiltonian which is not Lipschitz continuous.   The generalized Hopf--Lax formula has a fundamental limitation  if  the global existence of the bi-characteristic curves is not guaranteed.
This issue of bi-characteristic curves may occur regardless of entropy regularization.
}
\begin{align}
	\begin{aligned}\label{bichar}
		\dot{\gamma}(\bm{x},\bm{v},s) &= \nabla_{\bm{p}} H_\alpha(\gamma(\bm{x},\bm{v},s),p(\bm{x},\bm{v},s)),\quad \gamma(\bm{x},\bm{v},t)=\bm{x} \\
		\dot{p}(\bm{x},\bm{v},s) &= -\nabla_{\bm{x}}H_\alpha(\gamma(\bm{x},\bm{v},s),p(\bm{x},\bm{v},s)), \quad p(\bm{x},\bm{v},t) = \bm{v}
	\end{aligned}
\end{align}
and $q^*$ is the Legendre-Fenchel transformation of $q$, defined by $q^*(\bm{v}) := \max_{\bm{x}\in\bbr^n} \left\{\bm{x}\cdot \bm{v} - q(\bm{x})\right\}$. In particular, when $f(\bm{x},\bm{u}) = f(\bm{u})$ and $r(\bm{x},\bm{u}) = r(\bm{u})$, the soft HJB equation becomes
\[\partial_t W_\alpha +H_{\alpha} (\nabla_{\bm{x}} W_\alpha) = 0,\quad H_{\alpha}(\bm{p}):=\alpha \log\int_{U}\exp\left(-\frac{\bm{p}\cdot f(\bm{u})+r(\bm{u})}{\alpha}\right)\,\d\bm{u}.\]
In this state-independent case, we have $\nabla_{\bm{x}} H_\alpha \equiv0$, and therefore $p(\bm{x},\bm{v},s)=p(\bm{x},\bm{v},t)=\bm{v}$.
Then, the formula \eqref{grid_free2} is reduced to
\[W_\alpha(t,\bm{x}) = \max_{\bm{v}\in\bbr^n} \left\{\bm{x}\cdot \bm{v} -q^*(\bm{v}) -tH_\alpha(\bm{v})\right\},\]
which can be shown to be equivalent to the classical Hopf--Lax formula.

We note that either  \eqref{grid_free1} or \eqref{grid_free2} can be used only if the Hamiltonian $H_\alpha(\bm{x},\bm{p})$ can be explicitly evaluated for any given $(\bm{x},\bm{p})$. 
As mentioned, in the maximum entropy control problem, we are able to explicitly calculate the Hamiltonian as
\[H_{\alpha}(\bm{x},\bm{p}) = \alpha \log \int_U \exp\left(-\frac{\bm{p}\cdot f(\bm{x},\bm{u})+r(\bm{x},\bm{u})}{\alpha}\right)\,\d\bm{u}.\]
However, in the standard optimal control problem without an entropy term, the Hamiltonian is given by
\begin{equation}\label{ham_st}
H_0(\bm{x},\bm{p})=-\inf_{\bm{u}\in U}\{\bm{p}\cdot f(\bm{x},\bm{u})+r(\bm{x},\bm{u})\},
\end{equation}
which  has no explicit representation in terms of $(\bm{x},\bm{p})$ in general. Even worse, it is challenging to compute the standard Hamiltonian $H_0(\bm{x},\bm{p})$ when the optimization problem above is nonconvex. 
Unlike such standard optimal control cases,
grid-free methods based on generalized Hopf--Lax formulas are applicable to a large class of maximum entropy control problems even when it is impossible to evaluate the standard Hamiltonian. 
Therefore, it is more tractable to use generalized Hopf--Lax formulas in the maximum entropy control case than in the standard case. 
It is also worth emphasizing that, unlike the standard Hamiltonian $H_0$,
the soft Hamiltonian $H_\alpha$ is differentiable with respect to $\bm{p}$ and $\bm{x}$ if $\bm{x} \mapsto f(\bm{x}, \bm{u})$ and $\bm{x} \mapsto r(\bm{x}, \bm{u})$ are differentiable. 
This additional regularity of $H_\alpha$ allows us to use the characteristic curve formulation~\eqref{bichar} in maximum entropy control even when the standard Hamiltonian is not differentiable.\footnote{When the standard Hamiltonian is not differentiable, the remark in \cite{chow2019algorithm} suggests to use its subdifferentials in the characteristic formula \eqref{bichar}, regarding the ODE as a differential inclusion. However, although using subdifferentials is theoretically reasonable, it is challenging to explicitly compute the subdifferential of $H_0$, thereby making the  differential inclusion approach impractical.}

\section{Linear-Quadratic Control with Unknown Model Parameters}\label{sec:5}

Recall that one of important motivations for using the maximum entropy formulation 
is to enhance the exploration capabilities of RL agents, thereby better balancing the exploitation-exploration tradeoff
 when system models are not fully known. 
In discrete-time settings, there have been a number of empirical evidences in the effectiveness of maximum entropy RL methods~\cite{Peters2010, Fox2016, Haarnoja2017, Haarnoja2018, Hazan2019}. 
However, to our knowledge, RL methods for continuous-time dynamical systems adopt heuristic exploration mechanisms, such as $\epsilon$-greedy and injecting an artificial noise signal~(e.g., \cite{Doya2000, Munos2000, Palanisamy2015, Vamvoudakis2017, jiang2014robust, yang2017hamiltonian, Bian2019, kim2020hamilton}).

In this section, we claim that the idea of maximum entropy RL can be extended to the continuous-time setting  using our soft HJB framework.
Specifically, we consider the following maximum entropy linear-quadratic control problem with unknown parameters:
 \[\min_{g \in \mathcal{G}_\infty} \; J^\alpha_{\bm{x}}(g):=\int_0^\infty e^{-\lambda t} \left(\int_U(x(t)^\top Qx(t)+\bm{u}^\top R \bm{u})g(t,\bm{u})\d\bm{u}-\alpha H(g(t,\cdot))\right)\d t,\]
 where $x(t)$ is the state trajectory of 
\begin{equation}\label{linear_relax}
	\dot{x}(t) = Ax(t) + B\int_U \bm{u}g(t,\bm{u})\d \bm{u}, \quad x(0) = \bm{x}.
\end{equation}
In Section~\ref{sec:lq}, we have already shown that the value function is given as a quadratic function $V_\alpha(\bm{x}) = \frac{1}{2}\bm{x}^\top P\bm{x}+c$, where $P$ is a symmetric positive definite solution to the  ARE \eqref{ARE-0}
and the optimal control is given as a normal distribution $g_\alpha^\star = \mathcal{N}(-R^{-1}B^\top P\bm{x},\alpha R^{-1})=:\mathcal{N}(-K\bm{x},\alpha R^{-1})$, where  $K = R^{-1}B^\top P$. Therefore, to obtain the optimal control $g_\alpha^\star$, we need to find a pair of matrices $(P,K)$ by solving the ARE \eqref{ARE-0}, which can be reformulated as
\[(A-BK)^\top P+P(A-BK)-\lambda P +Q+K^\top R K=0.\]

Suppose first that the system matrices $A$ and $B$ are fully known.  Throughout this section, we   assume that the pair $(A-\frac{\lambda}{2}I,B)$ is stabilizable\footnote{This condition automatically holds when the pair $(A,B)$ is stabilizable.} and that the pair $(A-\frac{\lambda}{2}I,Q^{\frac{1}{2}})$ is observable\footnote{By the Popov-Belevitch-Hautus rank test, this is equivalent to the observability of the pair $(A,Q^{\frac{1}{2}})$.}
The policy  iteration method in~\cite{kleinman1968iterative} can be used to numerically solve the ARE. Starting from an arbitrary gain matrix $K_0$ such that $A-\frac{\lambda}{2}I-BK_0$ is Hurwitz, we set $P_k$ as a symmetric positive definite solution to 
\begin{equation}\label{iter-ARE}
	(A-BK_k)^\top P_k + P_k(A-BK_k)-\lambda P_k +Q+K^\top_kRK_{k}=0,\quad k\ge0,
\end{equation}
and update the gain matrix as $K_{k+1}:=R^{-1}B^\top P_k$. Then, 
it directly follows from \cite{kleinman1968iterative} that
the sequence of  $(P_k,K_{k+1})$ converges to $(P,K)$, where $P$ is the unique positive definite solution to the ARE \eqref{ARE-0} and $K=R^{-1}B^\top P$.
This property is summarized as the following convergence result for maximum entropy policy iteration:

\begin{proposition}	Suppose that the initial gain matrix $K_0$ is chosen so that  $A-\frac{\lambda}{2}I-BK_0$ is Hurwitz. Let $\{(P_k,K_{k+1})\}$ be a sequence of matrix pairs constructed by \eqref{iter-ARE}. Then, the matrix $A-\frac{\lambda}{2}I-BK_k$ is Hurwitz, $P \preceq P_{k+1}\preceq P_k$, and $\lim_{k\to \infty} (P_k,K_k) = (P,K)$.
\end{proposition}

However, when the system matrices $A$ and $B$ are unknown, we cannot directly solve \eqref{iter-ARE}. 
Instead, data-driven methods can be used to indirectly perform the iterative procedure above. 
In the following subsections, we present on-policy and off-policy methods   to learn the optimal pair $(P,K)$ using system trajectory data.
The on-policy method uses  sample data generated using the most recent $(P,K)$ pair at every iteration. 
Therefore, one cannot recycle  samples produced by old pairs. 
However, in the off-policy method,  sample data generated by previous estimates of $(P, K)$ are reusable. 
The two methods use the adaptive dynamic programming approach~\cite{jiang2014robust} in our maximum entropy setting. 

\subsection{On-Policy Method}

Let  $\pi_k: \mathbb{R}^n \to L_+^1(U)$ be the Markov policy, at iteration $k$, that maps system state $x(t)$ to input $g_k(t, \cdot)$.
With this policy, the closed-loop system of \eqref{linear_relax} at iteration $k$ can be written as
\begin{align*}
	\dot{x}(t) &= Ax(t) +B\int_U \bm{u} \pi_k(x(t); \bm{u})\,\d \bm{u} = (A-BK_k)x(t) + B\int_{U} \bm{u} \left[\pi_k(x(t); \bm{u})\d \bm{u} -\delta_{-K_kx(t)}(\d \bm{u})\right] \\
	&=:(A-BK_k)x(t) +B\int_U \bm{u} \e_k(t;\d \bm{u}),
\end{align*}
where   the measure $\e_k(t;\d\bm{u})$ is defined as the difference between the Dirac measure $\delta_{-K_kx(t)}(\d\bm{u})$ and $g_k(x(t),\bm{u})\d\bm{u}$. 
Differentiating $e^{-\lambda t}x^\top(t)P_k x(t)$ with respect to $t$ and using \eqref{iter-ARE}, we obtain
\begin{align}
	\begin{aligned}\label{est-0}
		\frac{\d}{\d t}&(e^{-\lambda t}x^\top(t) P_kx(t))\\
		&=-\lambda e^{-\lambda t}x^\top(t) P_k x(t)+e^{-\lambda t}\left((A-BK_k)x(t)+B\int_U\bm{u}\e_k(t;\d \bm{u})\right)^\top P_kx(t) \\
		&\quad+e^{-\lambda t}x^\top(t) P_k\left((A-BK_k)x(t)+B\int_U\bm{u}\e_k(t;\d \bm{u})\right)\\
		& = e^{-\lambda t}x^\top(t) ((A-BK_k)^\top P_k +P_k(A-BK_k)-\lambda P_k)x(t) +2e^{-\lambda t}\left(\int_U\bm{u}\e_k(t;\d\bm{u})\right)^\top B^\top P_kx(t)\\
		&=-e^{-\lambda t}x^\top(t) (Q+K_k^\top RK_k)x(t)+2e^{-\lambda t}\left(\int_U\bm{u}\e_k(t;\d \bm{u})\right)^\top RK_{k+1}x(t).
	\end{aligned}
\end{align}
For given $l$ time intervals $[t_i,t_i+\delta t]$, $i=1,2\ldots, l$, we integrate \eqref{est-0} from $t_i$ to $t_i+\delta t$ to derive the following set of $l$ equations that the pair $(P_k, K_{k+1})$ should satisfy:
\begin{align}
	\begin{aligned}\label{est_0_adp}
		e^{-\lambda(t_i+\delta t)}x^\top(t_i+\delta t) P_k x(t_i+\delta t)& -e^{-\lambda t_i}x(t_i)^\top P_k x(t_i)   -2\int_{t_i}^{t_i+\delta t}e^{-\lambda s} \left(\int_U\bm{u}\e_k(s;\d \bm{u})\right)^\top RK_{k+1}x(s)\d s\\
		&=-\int_{t_i}^{t_i+\delta t} e^{-\lambda s}x^\top(s)(Q+K_k^\top R K_k)x(s)\d s.
	\end{aligned}
\end{align}
These equations can be compactly expressed as a single matrix equation using a vectorization operator \cite{jiang2014robust}. 
For a given $m\times n$ matrix $A=(a_{ij})$, the vectorization $\mbox{vec}(A)$ of the matrix $A$ is defined as
\[\mbox{vec}(A):=(a_{11},\ldots,a_{m1},a_{12},\ldots,a_{m2},\ldots,a_{1n},\ldots,a_{mn})^\top.\]
By the well-known identity $\mbox{vec}(ABC) = (C^\top\otimes A)\mbox{vec}(B)$, \eqref{est_0_adp} can be written as
\begin{equation}\label{adp_0}
	\Theta_k\begin{pmatrix}
		\mbox{vec}(P_k)\\
		\mbox{vec}(K_{k+1})\end{pmatrix} = \Xi_k,
\end{equation}
where
\begin{equation}\label{thetak}
\Theta_k = \begin{pmatrix}
e^{-\lambda t}x^\top(t)\otimes x^\top(t)\Big|_{t_1}^{t_1+\delta t}& -2\int_{t_1}^{t_1+\delta t} e^{-\lambda s}\left(x^\top \otimes \left(\int_U \bm{u} \e_k(s;\d\bm{u})\right)^\top R\right)\d s\\
\vdots& \vdots\\
e^{-\lambda t}x^\top(t)\otimes x^\top(t)\Big|_{t_l}^{t_l+\delta t}& -2\int_{t_l}^{t_l+\delta t} e^{-\lambda s}\left(x^\top \otimes \left(\int_U \bm{u} \e_k(s;\d\bm{u})\right)^\top R\right)\d s
\end{pmatrix}\in \bbr^{l\times (n^2+nm)},
\end{equation}
and
\begin{equation}\label{xik}
\Xi_k :=\begin{pmatrix}
-\int_{t_1}^{t_1+\delta t} e^{-\lambda s} x^\top(s) (Q+K_k^\top RK_k)x(s)\,\d s\\
\vdots\\
-\int_{t_l}^{t_l+\delta t} e^{-\lambda s} x^\top(s) (Q+K_k^\top RK_k)x(s)\,\d s
\end{pmatrix}\in \bbr^{l}. 
\end{equation}
Note that the matrices $\Theta_k$ and $\Xi_k$ can be computed using the system state and input trajectory data even when $A$ and $B$ are unknown. 
Thus, $(P_k, K_{k+1})$ can be found   by solving the linear matrix equation without  knowing $A$ and $B$, under a suitable rank condition on $\Theta_k$.
The following proposition follows directly from \cite[Theorem 2.3.6]{jiang2014robust}.

\begin{proposition}[Convergence of on-policy learning]
	Under the condition that
	\begin{equation}\label{rank_on}
	\textup{rank}\left(\Theta_k\right) = \frac{n(n+1)}{2}+mn,
	\end{equation}
	there exists a unique pair $(P_k, K_{k+1})$ with $P_k=P_k^\top$ satisfying \eqref{adp_0}.  If, in addition, the initial gain matrix $K_0$ is chosen so that  $A-\frac{\lambda}{2}I-BK_0$ is Hurwitz,\footnote{In general, finding such a stabilizing gain matrix $K_0$ may be nontrivial. To address this issue, one may use the value iteration method proposed in \cite{Bian2016}.} then 
	\begin{itemize}
	\item 
	$A-\frac{\lambda}{2}I-BK_{k}$ is Hurwitz;
	
	\item
	the  pair $(P_k, K_{k+1})$ converges to the optimal pair $(P,K)$ as $k\to\infty$.
	\end{itemize}
\end{proposition}

\begin{algorithm}[tb]
	\caption{On-policy maximum entropy method for data-driven  LQ control}
	\label{alg:on-policy}
 Initialize the gain matrix $K_0$ so that $A-BK_0$ is Hurwitz;\\
		\For{$k=0,1,2,\ldots$}{
 Initialize $t_1=0$;\\
		\While{the rank condition \eqref{rank_on} is not satisfied}{
Execute the control $\mathcal{N}(-K_kx(t_i),\alpha R^{-1})$ and collect controlled trajectory data;\\
Construct the $i$th row of $\Theta_k$ and $\Xi_k$ in \eqref{thetak} and \eqref{xik}; \\
Set $t_{i+1} = t_i+\delta t$;
}
Find the pair $(P_k,K_{k+1})$ by solving \eqref{adp_0};\\
 Stop if $|P_k-P_{k+1}|<\varepsilon$, where $\varepsilon$ is a predefined threshold;
}
\end{algorithm}

This proposition indicates the convergence property that the optimal pair $(P, K)$ can be learned using our method. Furthermore, the gain matrices $K_k$ constructed at any intermediate steps guarantee the exponential stability of $e^{-\frac{\lambda t}{2}} x(t)$. 

All the steps in the data-driven control method  are summarized in
Algorithm \ref{alg:on-policy}, which is an unapproximated version. 
Practical systems may not take an input in the form of probability density. If that is the case, we can sample a control input from the distribution $\mathcal{N}(-K_kx(t_i),\alpha R^{-1})$ and exert it to the system. The integrals in \eqref{thetak} and \eqref{xik} can then be approximated accordingly. 
The sampling approach is supported by the discrete-time approximation result in Section~\ref{sec:dt}.

When  constructing the matrices $\Theta_k$ and $\Xi_k$, we use the system state controlled with the most recent gain matrix $K_k$. Thus, in every iteration, new trajectory data must be collected using the current gain matrix $K_k$. 
This implies that Algorithm \ref{alg:on-policy} is an \emph{on-policy} method.
 
We now discuss the difference between  the adaptive DP method in \cite{jiang2014robust} and our method.
The adaptive DP method adopts sinusoidal signals as a heuristic exploration mechanism, and thus uses the input $u = -K_kx +e$, where $e$ is an exploration noise constructed as the sum of sinusoidal signals. 
However, our method uses a principled information theoretic exploration mechanism. 
As a result, our control $\mathcal{N}(-K_k x,\alpha R^{-1})$ itself has an exploration capability and our method does not need to inject a separate artificial noise, which may degrade  the overall performance. 
Our result also confirms that the common practice of using  Gaussian noise in RL is effective in the sense of maximum entropy, if the mean and the covariance matrix are carefully chosen,
 when considering linear-quadratic problems.

\subsection{Off-Policy Method}

The on-policy method in the previous subsection needs trajectory data newly generated using the most recent gain matrix in every iteration. 
To improve sample efficiency, 
we now present an \emph{off-policy} variant of Algorithm~\ref{alg:on-policy}.
Unlike the on-policy method in the previous subsection, we fix a control $g_0$ and use the trajectory data generated under $g_0$ in all iterations.
The closed-loop system with $g_0$ is given by
\begin{equation}\label{linear_off}
	\dot{x}(t)=Ax(t)+B\int_U\bm{u}g_0(t,\bm{u})\d \bm{u}.
\end{equation}
Differentiating $\frac{\d}{\d t}(e^{-\lambda t}x^\top(t) P_k x(t))$ with respect to $t$ and using \eqref{iter-ARE} and \eqref{linear_off}, we obtain
\begin{align}
	\begin{aligned}\label{est}
		&\frac{\d}{\d t}(e^{-\lambda t}x^\top(t) P_kx(t)) \\
		&=-\lambda e^{-\lambda t}x^\top(t) P_k x(t)+e^{-\lambda t}\left(Ax(t)+B\int_U\bm{u}g_0(t,\bm{u})\d \bm{u}\right)^\top P_kx(t) \\
		&\quad +e^{-\lambda t}x^\top(t) P_k\left(Ax(t)+B\int_U \bm{u}g_0(t,\bm{u})\d \bm{u}\right)\\
		& = e^{-\lambda t}x^\top(t) (A^\top P_k +P_kA-\lambda P_k)x(t) +2e^{-\lambda t}\left(\int_U\bm{u}g_0(t,\bm{u})\d \bm{u}\right)^\top B^\top P_kx(t)\\
		&=e^{-\lambda t}x^\top(t) (-Q-K_k^\top RK_k +K_k^\top B^\top P_k + P_kBK_k)x(t)+2e^{-\lambda t}\left(\int_U\bm{u}g_0(t,\bm{u})\d \bm{u}\right)^\top RK_{k+1}x(t)\\
		&=e^{-\lambda t}x^\top(t) (-Q-K_k^\top RK_k +K_k^\top RK_{k+1} +K_{k+1}^\top RK_k)x(t)+2e^{-\lambda t}\left(\int_U\bm{u}g_0(t,\bm{u})\d \bm{u}\right)^\top RK_{k+1}x(t)\\
		&=e^{-\lambda t}x^\top(t) (-Q-K_k^\top RK_k )x(t) +2e^{-\lambda t}x^\top(t) K_k^\top RK_{k+1}x(t)+2e^{-\lambda t}\left(\int_U\bm{u}g_0(t,\bm{u})\d \bm{u}\right)^\top RK_{k+1}x(t)\\
		&=-e^{-\lambda t}x^\top(t) (Q+K_k^\top RK_k)x(t) + 2e^{-\lambda t}\left(K_kx(t) +\int_U\bm{u}g_0(t,\bm{u})\d \bm{u}\right)^\top RK_{k+1}x(t).
	\end{aligned}
\end{align}
Integrating \eqref{est} over $[t_i,t_i+\delta t]$, $i=1,2,\ldots, l$, yields
\begin{align}
	\begin{aligned}\label{est_adp}
		&e^{-\lambda(t_i+\delta t)}x^\top(t_i+\delta t) P_k x(t_i+\delta t) -e^{-\lambda t_i}x(t_i)^\top P_k x(t_i)  \\
		&\qquad -2\int_{t_i}^{t_i+\delta t}e^{-\lambda s} \left(K_kx+\int_U\bm{u} g_0(s,\bm{u})\d \bm{u}\right)^\top RK_{k+1}x(s)\d s\\
		&=-\int_{t_i}^{t_i+\delta t} e^{-\lambda s}x^\top(s)(Q+K_k^\top R K_k)x(s)\d s.
	\end{aligned}
\end{align}
We have $l$ equations of \eqref{est_adp} to calculate the matrices $(P_k, K_{k+1})$ using trajectory data. 
To transform \eqref{est_adp} into a single matrix equation, we introduce the following matrices:
\begin{align}
	\begin{aligned}\label{data}
		&\Delta = \begin{pmatrix}
			e^{-\lambda t}x(t)\otimes x(t)\Big|_{t_1}^{t_1+\delta t},\quad e^{-\lambda t}x(t)\otimes x(t)\Big|_{t_2}^{t_2+\delta t},\quad\ldots,\quad e^{-\lambda t}x(t)\otimes x(t)\Big|_{t_l}^{t_l+\delta t}
		\end{pmatrix}^\top \in \bbr^{l\times n^2},\\
		&I_1 = \begin{pmatrix}
			\int_{t_1}^{t_1+\delta t} e^{-\lambda s} x\otimes x \d s,\quad \int_{t_2}^{t_2+\delta t} e^{-\lambda s} x\otimes x \d s,\quad\ldots,\quad \int_{t_l}^{t_l+\delta t} e^{-\lambda s} x\otimes x \d s
		\end{pmatrix}^\top\in \bbr^{l\times n^2},\\
		&I_2 = \begin{pmatrix}
			\int_{t_1}^{t_1+\delta t} e^{-\lambda s} x\otimes \left(\int_U\bm{u} g_0(s,\bm{u})\d \bm{u}\right) \d s,\quad \ldots,\quad \int_{t_l}^{t_l+\delta t} e^{-\lambda s} x\otimes \left(\int_U\bm{u}g_0(s,\bm{u})\d \bm{u}\right) \d s
		\end{pmatrix}^\top\in \bbr^{l\times mn}.
	\end{aligned}
\end{align}
Then, the $l$ equations in \eqref{est_adp} can be combined to the following linear matrix equation:
\begin{equation}\label{adp}
	\begin{pmatrix}
		\Delta , -2I_1(I_n\otimes K_k^\top R)-2I_2(I_n\otimes R)
	\end{pmatrix}\begin{pmatrix}
		\mbox{vec}(P_k)\\
		\mbox{vec}(K_{k+1})\end{pmatrix} = -I_1\mbox{vec}(Q+K_k^\top RK).
\end{equation}
Again, under a suitable rank condition on the matrices $I_1$ and $I_2$, the matrices $(P_k,K_{k+1})$ satisfying \eqref{adp} can be obtained as a unique solution to the linear equation~\eqref{adp}.
The following proposition follows directly from \cite[Theorem 2.3.12]{jiang2014robust}.

\begin{proposition}[Convergence of off-policy learning]
	Under the condition that
	\begin{equation}\label{rank_off}
	\textup{rank}\left(\begin{pmatrix}
	I_1,~ I_2
	\end{pmatrix}\right) = \frac{n(n+1)}{2}+mn,
	\end{equation}
	there exists a unique pair $(P_k, K_{k+1})$ with $P_k=P_k^\top$ satisfying \eqref{adp}.  
	If, in addition, the initial gain matrix $K_0$ is chosen so that  $A-\frac{\lambda}{2}I-BK_0$ is Hurwitz, then 
	\begin{itemize}
	\item 
	$A-\frac{\lambda}{2}I-BK_{k}$ is Hurwitz;
	
	\item
	the   pair $(P_k, K_{k+1})$  converges to the optimal pair $(P,K)$ as $k\to\infty$.
	\end{itemize}
\end{proposition}

\begin{algorithm}[tb]
	\caption{Off-policy maximum entropy method for data-driven  LQ control}
	\label{alg:off-policy}
 Initialize the gain matrix $K_0$ so that $A-BK_0$ is Hurwitz.\\
  Initialize $t_1=0$;\\
	\While {the rank condition \eqref{rank_off} is not satisfied}{
 Execute the control $\mathcal{N}(-K_kx(t_i),\alpha R^{-1})$ and collect controlled trajectory data;\\
Construct the $i$th row of $\Delta$, $I_1$ and $I_2$ in \eqref{data};\\
Set $t_{i+1} = t_i+\delta t$;
}
	\For {$k=0,1,2,\ldots$}{
Find the pair $(P_k,K_{k+1})$ by solving \eqref{adp};\\
Stop if $|P_k-P_{k+1}|<\varepsilon$, where $\varepsilon$ is a predefined threshold;
}
\end{algorithm}

The off-policy method also generates gain matrices $K_k$, stabilizing the system (up to the factor $e^{-\frac{\lambda t}{2}}$), and guarantees convergence to the optimal pair $(P, K)$.
Algorithm~\ref{alg:off-policy} describes the off-policy version of our maximum entropy data-driven control.
As in the on-policy case, in practice, we may sample a control input from the normal distribution and numerically approximate the integrals in  \eqref{data}.
While the on-policy method in the previous subsection needs to collect  new trajectory data in every iteration, Algorithm \ref{alg:off-policy} only uses the data collected in the beginning to construct the data matrices $\Delta$, $I_1$ and $I_2$.  Then, it repeatedly uses the same data matrix during the learning process. Thus,  Algorithm \ref{alg:off-policy} is an off-policy method.

\section{Numerical Examples}\label{sec:6}

We provide  numerical examples to demonstrate the performance and the utility of our maximum entropy optimal control methods. 
The second numerical experiment concerns  the effectiveness of the generalized Hopf--Lax formula in solving soft HJB equations.
In the second case study, a nonlinear system is controlled using the maximum entropy method with known model information. 
In the third set of experiments, we demonstrate the performance of our data-driven method in the linear-quadratic setting when the model parameters are unknown.

\begin{figure}[h!]
\centering
	\subfigure[]{
		\includegraphics[width=0.5\columnwidth]{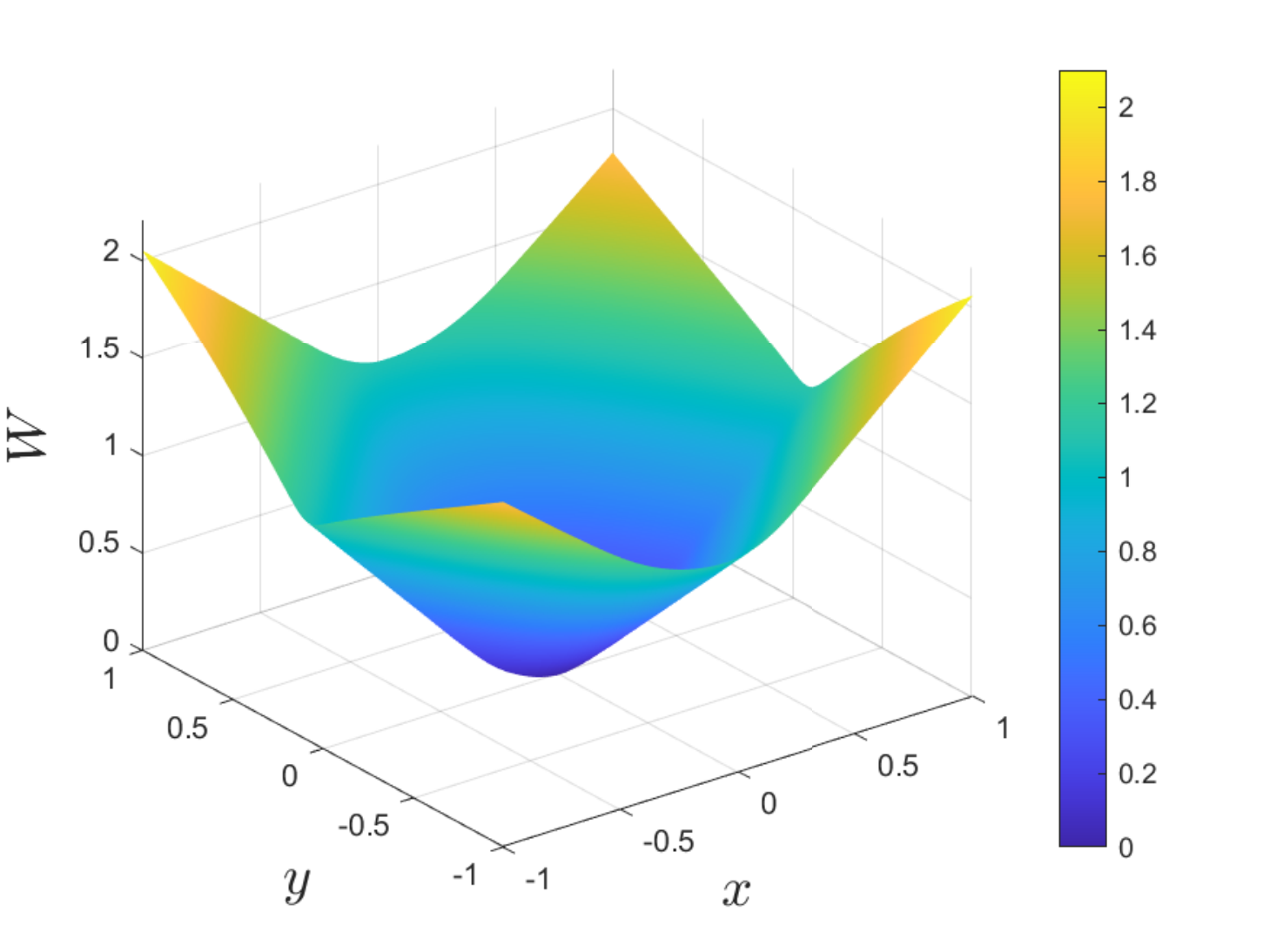}
		\includegraphics[width=0.5\columnwidth]{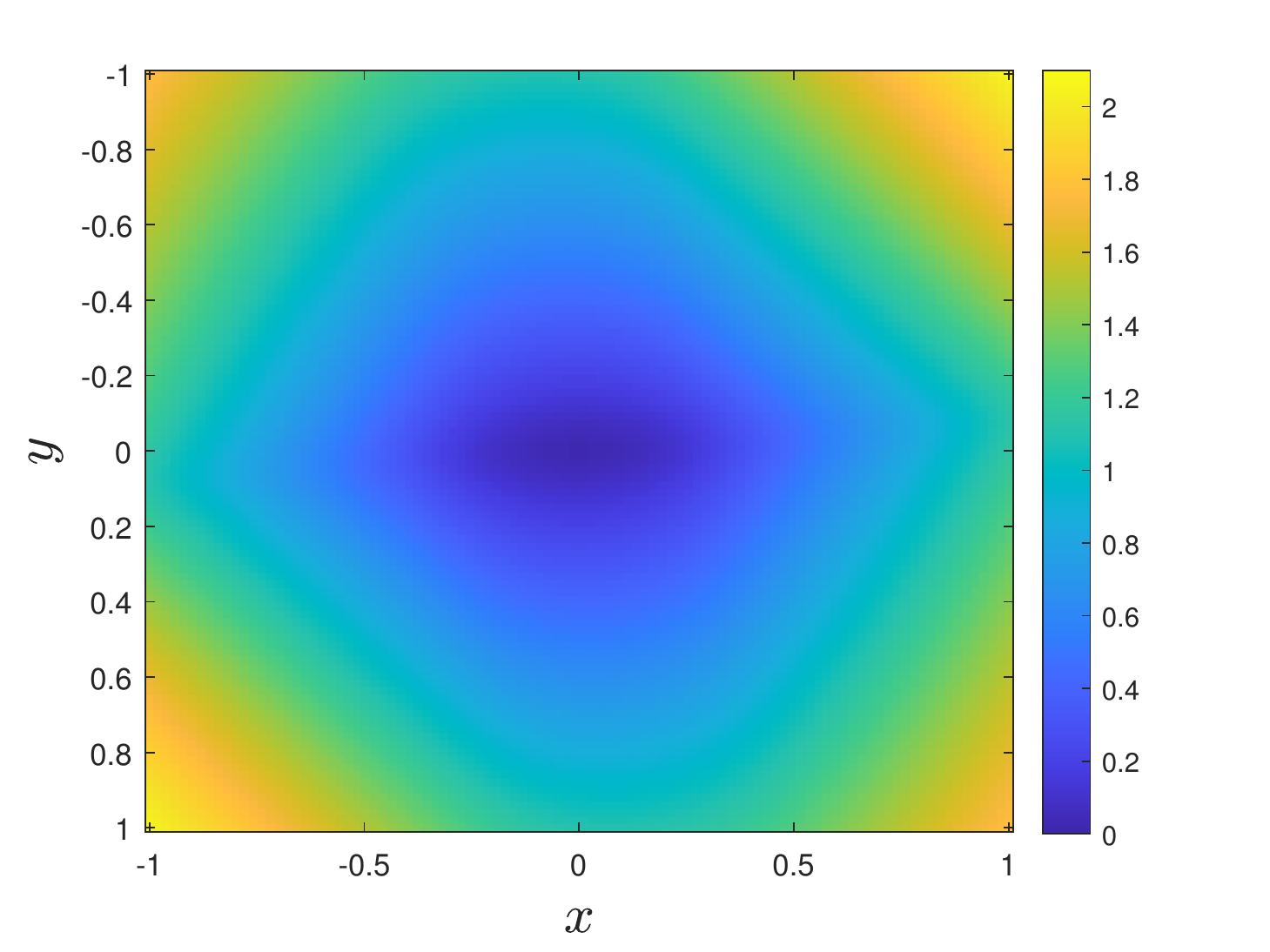}
	}
	\subfigure[]{
		\includegraphics[width=0.5\columnwidth]{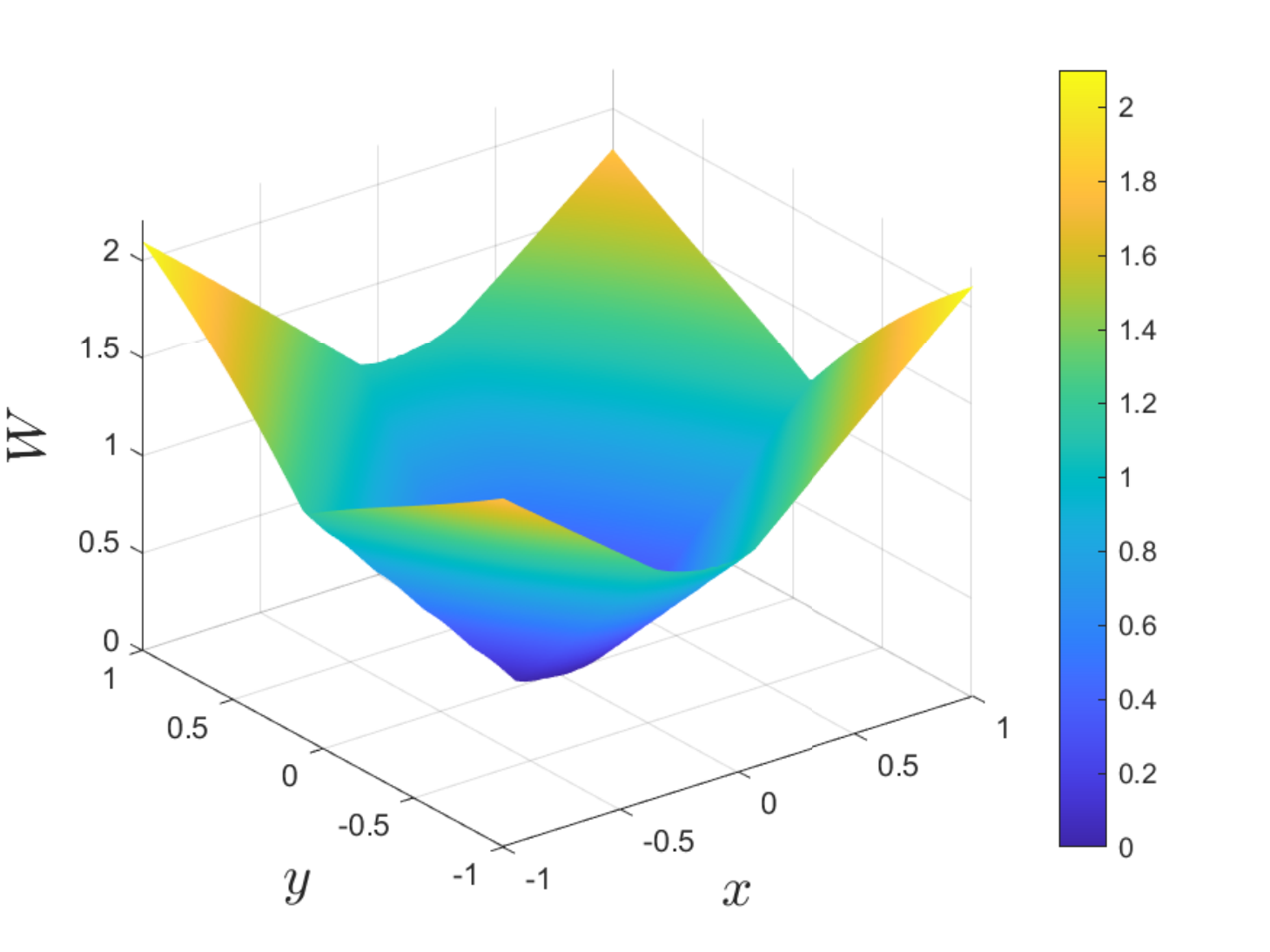}
		\includegraphics[width=0.5\columnwidth]{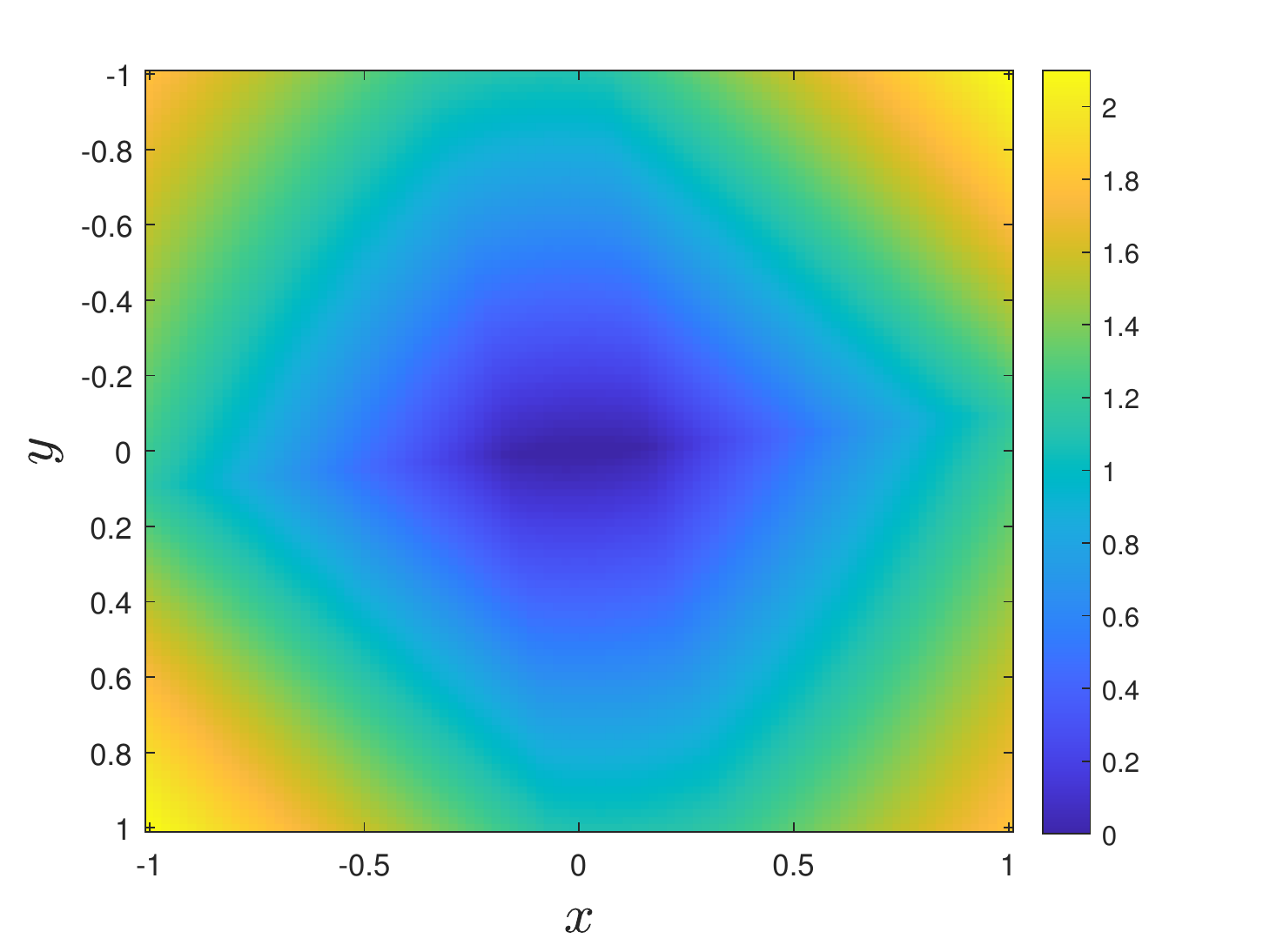}
	}
	\caption{The numerical viscosity solution of the soft HJB equation at $t=0.1$ obtained by (a) the Godunov scheme, and  (b) the grid-free scheme.}
	\label{fig3}
\end{figure}

\begin{figure}[h!]
\centering
	\includegraphics[width=0.6\columnwidth]{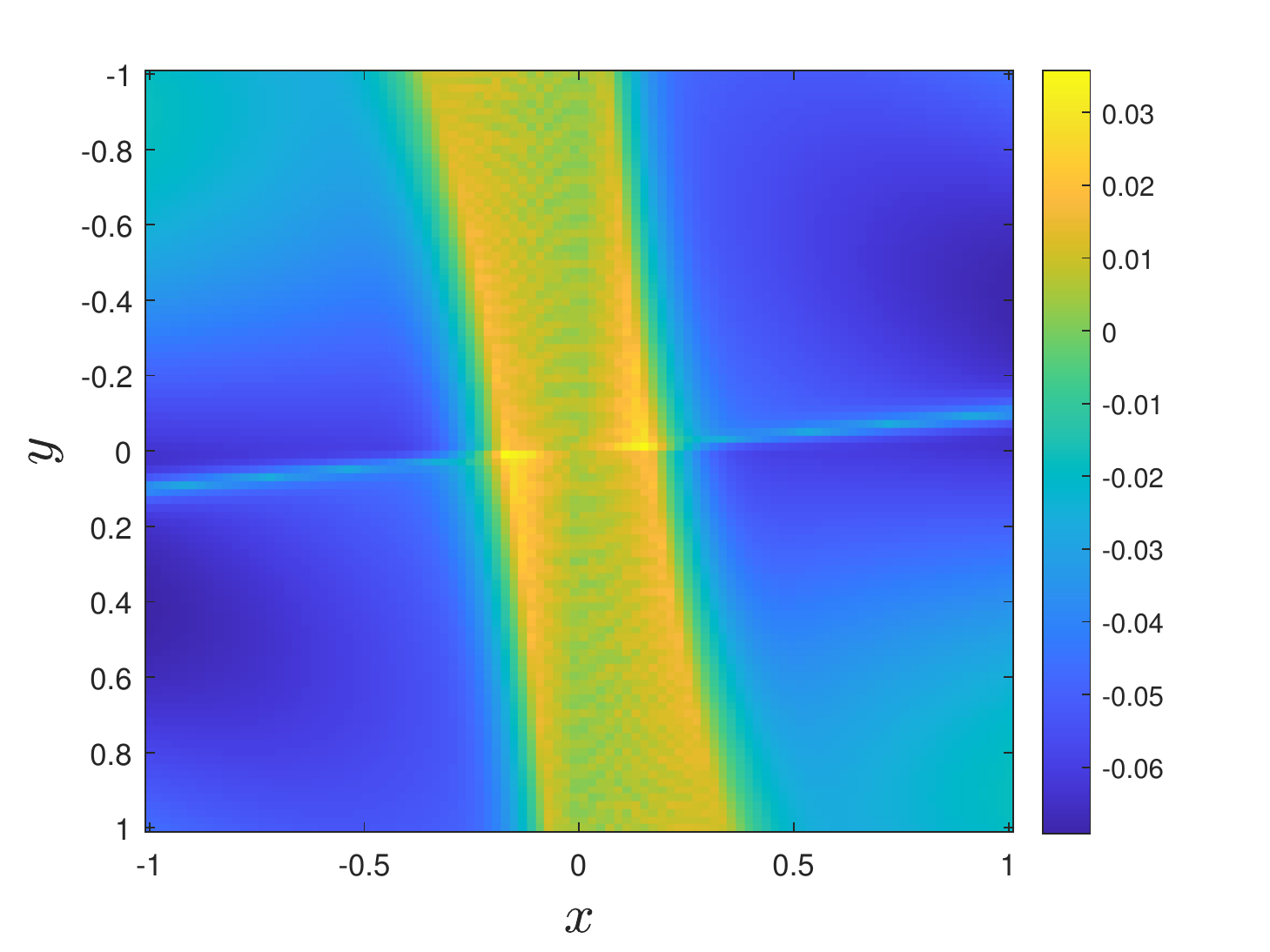}
	\caption{The difference between the two solutions; one obtained by the Godunov scheme, and another obtained by the grid-free scheme.}
	\label{fig5}
\end{figure}

\subsection{Solution of Soft HJB Equations via a Generalized Hopf--Lax Formula}

We first demonstrate that the soft HJB equation \eqref{HJB_soft} can be effectively solved using the generalized Hopf--Lax formula in Section \ref{sec:4.3}. 
We consider a Van der Pole oscillator with
 the following nonlinear vector field and running cost function:
\begin{equation}\label{fxu}
\begin{split}
&f(\bm{x},\bm{u}) = \left(x_2,-2(x_1^2-1)x_2-x_1+(2+\sin(x_1x_2))\left(\bm{u}+\frac{1}{3}\bm{u}^3+\sin \bm{u}\right)\right) \\
&r(\bm{x},\bm{u}) = |\bm{x}|+|\bm{u}|
\end{split}
\end{equation}
together with the terminal cost function $q(\bm{x}) = \|\bm{x}\|_{1}$. The set of available control $U$ is chosen as $[-1,1]$. Recall that the soft HJB equation can be written as the following initial value problem:
\begin{align}
\begin{aligned}\label{HJB_init}
	&\partial_t W_\alpha+H_{\alpha}(\bm{x},\nabla_{\bm{x}} W_\alpha) = 0,\quad H_{\alpha}(\bm{x},\bm{p}) = \alpha\log\int_U \exp\left(-\frac{\bm{p}\cdot f(\bm{x},\bm{u})+r(\bm{x},\bm{u})}{\alpha}\right)\,\d\bm{u},\\
	&W_\alpha(0,\bm{x})=q(\bm{x}).
\end{aligned}
\end{align}
We compare the numerical viscosity solution obtained by using the grid-free method in \cite{chow2019algorithm} and that computed using  the Godunov monotone scheme~\cite{bardi1991nonconvex,osher1991high}. 
Figure \ref{fig3} shows the results obtained by the two methods;
the overall solution shapes are almost identical. 
The difference between the two solutions is shown in
Figure \ref{fig5}. More precisely, the figure shows the value of $W_\alpha^1-W_\alpha^2$, where $W_\alpha^1$ is the numerical solution obtained by  the Godunov scheme and $W_\alpha^2$ is the solution constructed by the grid-free scheme. The difference is reasonably small;  thus, the grid-free method successfully solves the HJB \eqref{HJB_init}. The maximum of the absolute difference between two solutions is 0.0692, which is 3.29\% of the $\|W_\alpha^2\|_{L^\infty}$. 
As a remark, we note that there are two regions where the difference is relatively large. The first region is the boundary of the domain. In the Godunov scheme,  the extrapolating boundary condition is used, and it may introduce numerical errors. 
The other region is the center of the domain. This is due to the numerical dissipation of the Godunov scheme, where the solution is non-smooth, although it is less diffusive than other monotone schemes (e.g., \cite{osher1984riemann}). This dissipation indicates that the Godunov solution  at this non-smooth region is smoothing out, thereby causing an undesirable overestimation of the numerical solution.

\subsection{Nonlinear Systems with Known Model Parameters}

We now use the grid-free scheme based on the Hopf--Lax formula to solve a nonlinear maximum entropy optimal control problem.
Consider the following modified Van der Pole oscillator~\cite{yang2007adaptive}:
\begin{align}
	\begin{aligned}\label{vdp}
		&\dot{x}_1 = x_2,\\
		&\dot{x}_2 = -2(x_1^2-1)x_2-x_1+(2+\sin(x_1x_2))\left(u+\frac{1}{3}u^3+\sin(u)\right),\\
		&\dot{x}_3 = x_4,\\
		&\dot{x}_4 = -x_3-0.2x_4+x_1
	\end{aligned}
\end{align}
with  initial data $x(0)=(0.05,0.25,0,0.02)$.
The running cost $r(\bm{x},\bm{u})$ and the terminal cost $q(\bm{x})$ are chosen as
\[r(\bm{x},\bm{u}) = \|\bm{x}\|_{1}+|\bm{u}|,\quad q(\bm{x}) = \|\bm{x}\|_1.\]
Again, the set of available control $U$ is $[-1,1]$.  The standard HJB equation is given by
\[\partial_t V -H_0(\bm{x},\nabla_{\bm{x}}V)=0,\quad H_0(\bm{x},\bm{p}):=-\inf_{\bm{u}\in U}\left\{p \cdot f(\bm{x},\bm{u})+r(\bm{x},\bm{u})\right\},\]
where $f(\bm{x},\bm{u})$ is the vector field of \eqref{vdp}. 
Note that the minimization problem in the Hamiltonian $H_0$ is nonconvex due to the nonlinearity of $\bm{u} \to f(\bm{x}, \bm{u})$. 
Thus,  evaluating $H_0$ is computationally challenging.
However, the soft Hamiltonian~\eqref{Ham_soft} is explicitly represented as an integral, which can be computed using existing numerical methods. 
Thus, it is computationally tractable to use the generalized Hopf--Lax formula-based method for solving the corresponding soft HJB equation \eqref{HJB_soft} for maximum entropy control.
 In the experiment,   the temperature parameter was chosen as $\alpha=1$.
\begin{figure}[t]
\centering
	\includegraphics[width=0.6\textwidth]{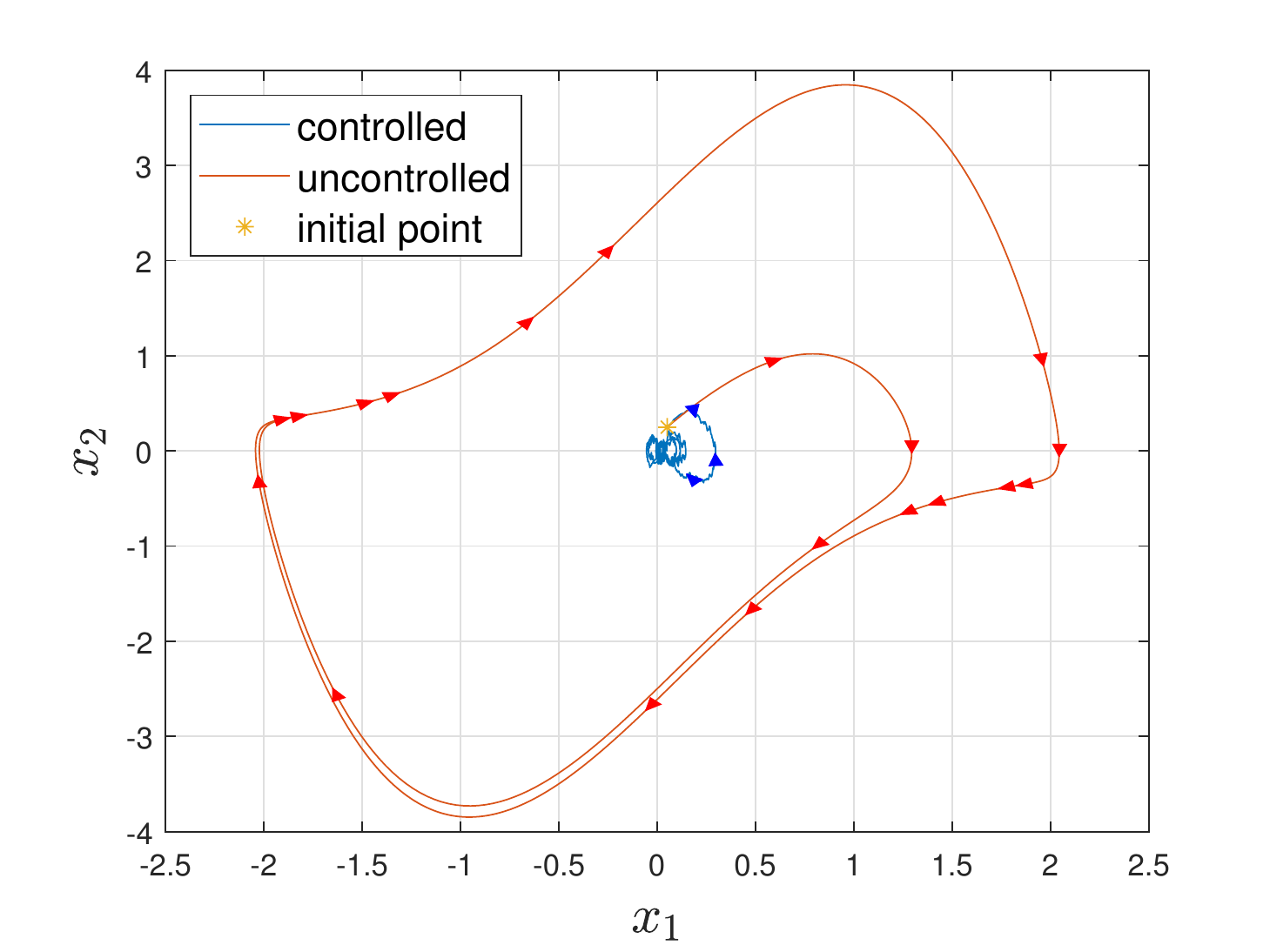}
	\caption{Uncontrolled and controlled $(x_1, x_2)$ trajectories of the Van der Pole oscillator.}
	\label{phase_diag}
\end{figure}

We construct the optimal control policy $g^\star_\alpha$ by solving the converted HJB equation \eqref{HJB-convert} with the generalized Hopf--Lax formula introduced in Section \ref{sec:4.3}. 
As emphasized in Footnote 7,   the generalized Hopf--Lax formula-based method can be unstable when solving HJB equations
for a long period of time due to the issue in the global existence of bi-characteristic curves, regardless of entropy regularization.
To control the nonlinear system \eqref{vdp} for a long period of time, say $T=20$, we construct  maximum entropy suboptimal controls by successively solving the subproblems with  $T' =2.5$ for 8 times. 
Figure \ref{phase_diag} shows the controlled and the uncontrolled trajectories of the first two states. 
As shown in the result, the maximum entropy controller successfully drives the nonlinear system near the origin,
 while the uncontrolled system converges to a limit cycle far from the origin.

\subsection{Linear Systems with Unknown Model Parameters}

In this subsection, we use the data-driven methods in Section~\ref{sec:5}  
to solve linear-quadratic control problems with unknown model parameters.

\subsubsection{On-Policy Method}\label{sec:exp_on}

Consider a linear system of the form:\footnote{The matrices $A$ and $B$ are randomly generated using the internal function \texttt{rss} in MATLAB, which  produces an arbitrary linear system model. The matrix $A$ is then modified by adding a constant multiplication of the identity matrix so that each eigenvalue of $A$ has a real part no greater than $-0.01$. The matrix $B$ is then multiplied by $0.1$. The system matrices used in Section~\ref{sec:exp_on} and~\ref{sec:exp_off} can be downloaded from the following link: \url{http://coregroup.snu.ac.kr/DB/sys_matrix.mat}.}
\[\dot{x}(t) = Ax(t)+Bu(t),\quad x(t)\in \bbr^{10},\quad u(t)\in \bbr^{10}.\]
The matrix $A$ is chosen to be  Hurwitz. Thus, with the initial gain matrix $K_0=0$,  $A-\frac{\lambda}{2}I-BK_0$ is Hurwitz for any $\lambda>0$.
Our specific choice of $A$ has  the eigenvalues at $-9.9067$, $-4.8468$, $-2.4977$, $-2.2825$, $-1.597$ $-1.4836\pm1.1164\textup{i}$, $-0.7143$, $-0.3318$ and $-0.01$.
As shown in  Figure~\ref{traj_on} (a), the system converges to 0  very slowly.
The running cost function is chosen as $r(\bm{x},\bm{u}) = 0.01 |\bm{x}|^2+|\bm{u}|^2$ and the discount factor is set to be $\lambda = 10^{-10}$.

\begin{figure}[tb]
	\centering
	\subfigure[]{
		\includegraphics[width=0.32\columnwidth]{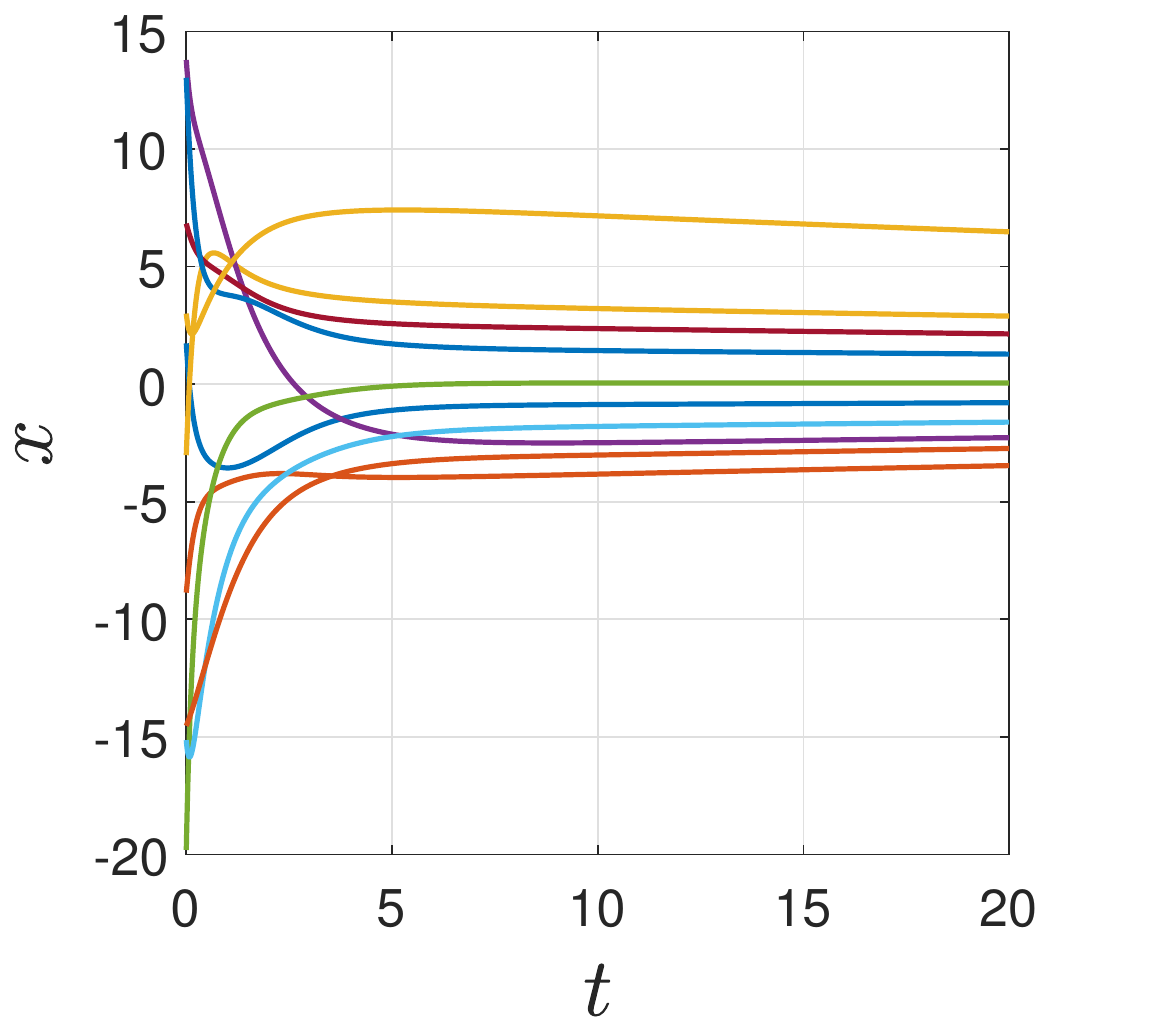}}
	\subfigure[]{
		\includegraphics[width=0.32\columnwidth]{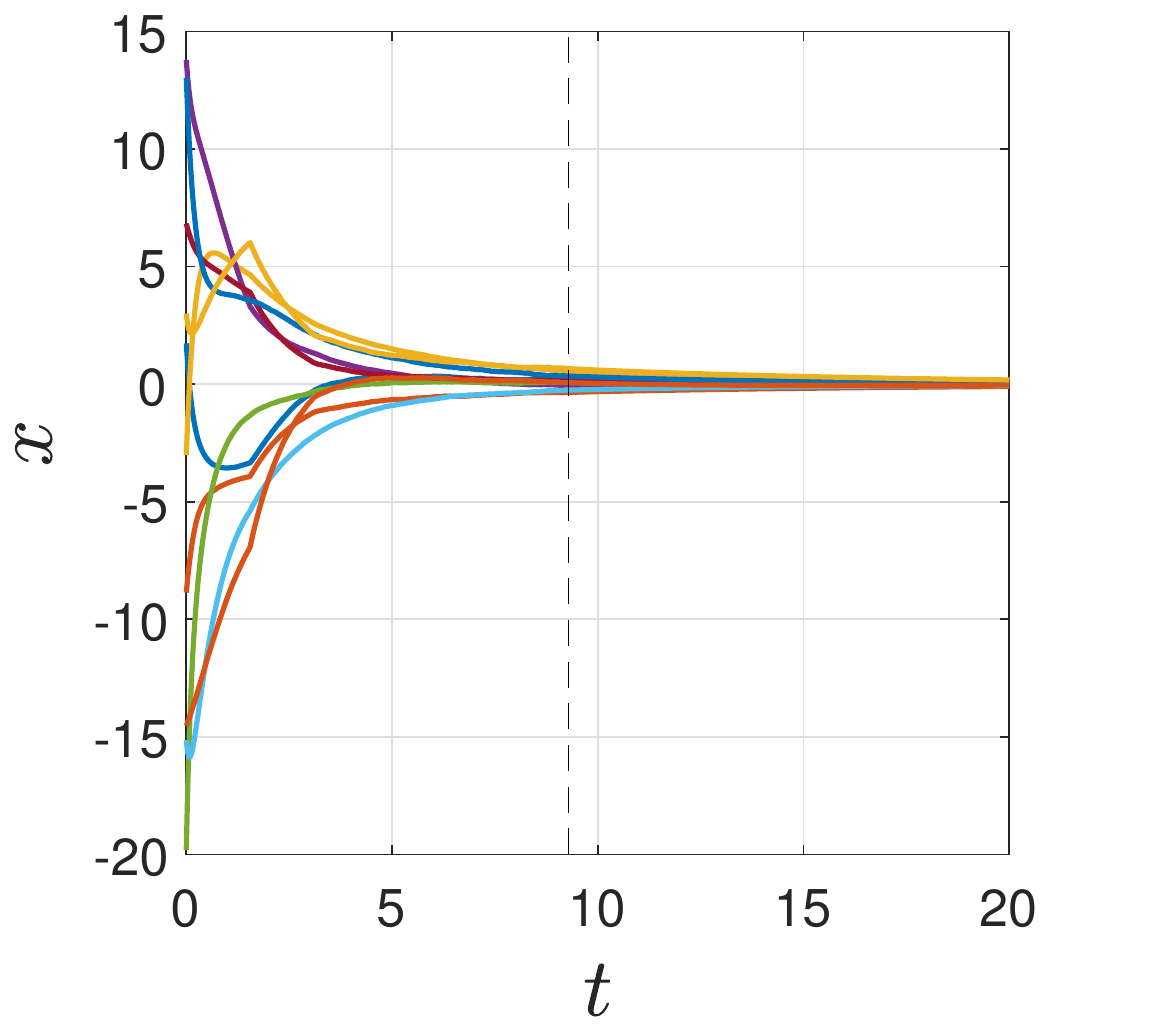}}
	\subfigure[]{
		\includegraphics[width=0.32\columnwidth]{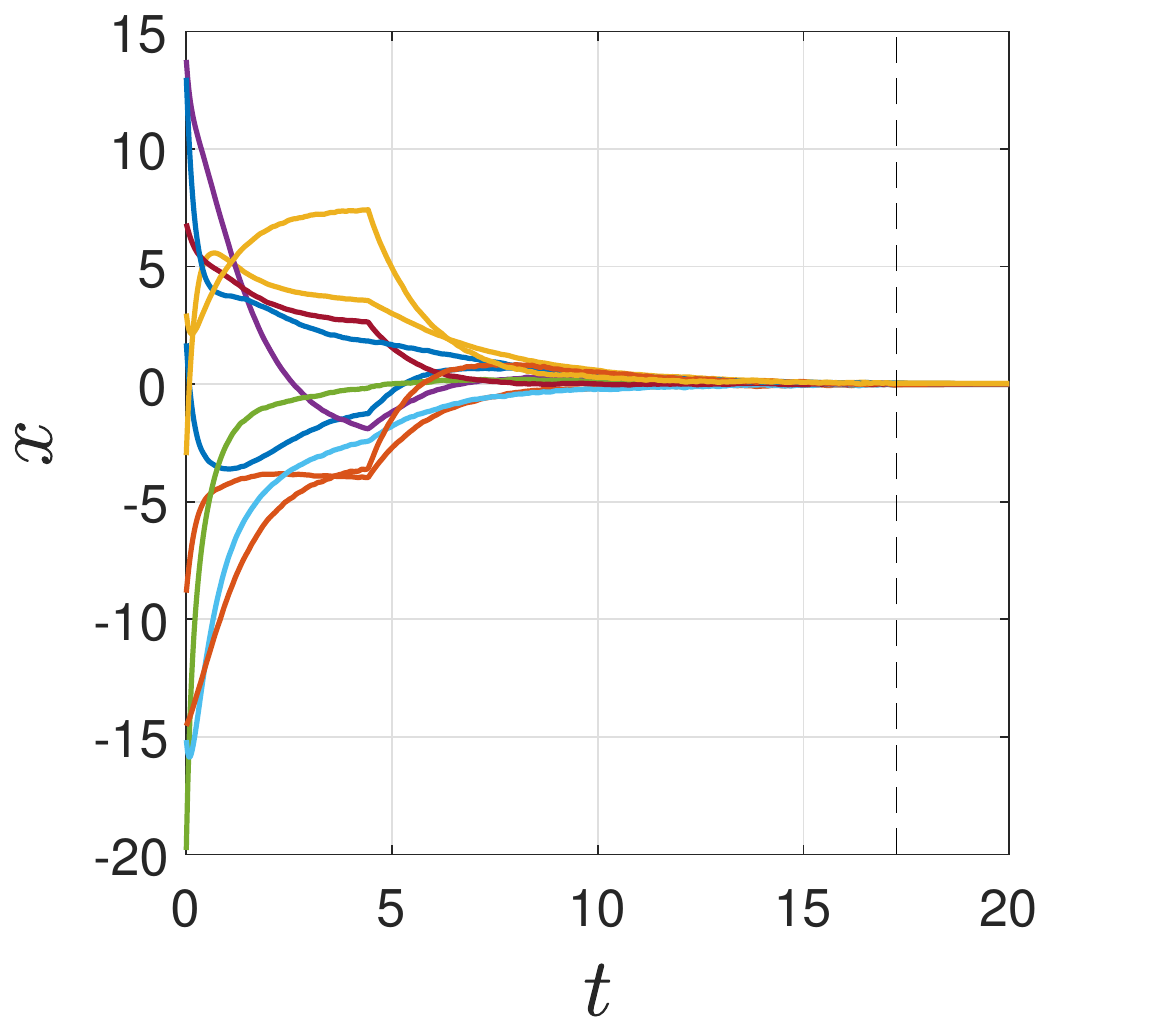}}
	\caption{On-policy case: System trajectories controlled by (a) $u \equiv 0$ (uncontrolled), (b) the maximum entropy method with $\alpha = 1$, and (c) the standard method. The vertical line indicates  when learning is completed.}
	\label{traj_on}
\end{figure}

\begin{table}[tb]
	\caption{On-policy case: Quantitative comparisons between the maximum entropy method and the standard adaptive DP method in \cite{jiang2014robust}.}
	\centering
	\begin{tabular}{ c |  >{\centering\arraybackslash}p{2.9cm} >{\centering\arraybackslash}p{2.9cm} >{\centering\arraybackslash}p{2.9cm}  }
	\hline
		 & \centering Uncontrolled & Max entropy $(\alpha=1)$ & Standard   \\ \hline\hline
Total running cost   & 686.92  & 467.77 & 496.48 \\
		Settling time 	     & $t=206.88$ & $t=6.60$ & $t=8.36$    \\
		Avg. \# of data for rank condition & N/A &155 & 431.5 \\
		Total \# of data  &  N/A  & 930 & 1726         \\
		Learning time   & N/A& $t = 9.30$ & $t = 17.26$ \\
		Computation time (sec)  & N/A& 1.36& 4.84 \\\hline
	\end{tabular}
	\label{tbl}
\end{table}

We first use the on-policy method in Section~\ref{sec:5}  
with $\delta t = 0.01$
to learn the optimal gain matrix $K$. 
We compare our method and the adaptive DP algorithm in \cite{jiang2014robust}, which is also data-driven, with the following sinusoidal exploration noise:
\[
e =  a \sum_{k=1}^{100}\sin(\omega_k t),\quad \omega_k\sim U(-\bar{\omega}, \bar{\omega}),
\]
where $a$ and $\bar{\omega}$ denote the amplitude and the frequency of the sinusoidal exploration, respectively.
We use $\alpha= 1$   for the maximum entropy method and $a=0.5$ and $\bar{\omega}=100$ for the standard adaptive DP method. The  threshold for convergence is set to be $\varepsilon=5\times10^{-1}$. 
Figure~\ref{traj_on} (b) and (c) show 
the system trajectories controlled by 
the two on-policy learning methods.  Both methods successfully learn the optimal gain matrix $K$ after several iterations. 
However, the learning speed of the maximum entropy method is much faster than that of the standard method.
To be precise, our method finishes learning at $t=9.30$ (with the actual total computation time of $1.36$ seconds), while its standard counter part finishes learning at $t=17.26$ (with the actual total computation time of $4.84$ seconds). The dashed vertical lines in Figure \ref{traj_on} indicate the times at which learning is completed. 
 Note also that the trajectories controlled by the on-policy methods are not smooth at which the gain matrices $K_k$ are updated.

\begin{table}[tb]
	\caption{Effect of $\alpha$ on the on-policy method.}
	\centering
	\begin{tabular}{ c | >{\centering\arraybackslash}p{1.5cm} >{\centering\arraybackslash}p{1.5cm} >{\centering\arraybackslash}p{1.5cm} >{\centering\arraybackslash}p{1.5cm} >{\centering\arraybackslash}p{1.5cm}  }
\hline
		$\alpha$   & $1$ & $0.5$ & $0.1$ & $0.05$ & $0.01$ \\ \hline
		Avg. \# of data   & 155 & 155 & 155 & 155& 155.5\\
Total running cost   & 467.77 & 387.49 & 371.59 & 370.72 & 370.93\\\hline
	\hline		
	$\alpha$ & $5\times 10^{-3}$ & $10^{-3}$ & $8\times10^{-4}$ &$6\times 10^{-4}$ & $4\times 10^{-4}$  \\ \hline
		Avg. \# of data   & 158.6 & 190.25 & 197.50 & 277.25 & 548.5 \\
Total running cost & 371.86 & 381.78& 382.93 & 405.07 & 458.22 \\\hline
\end{tabular}
	\label{tbl_alpha_on}
\end{table}

Table \ref{tbl} provides quantitative comparisons of the two methods.
First, our method significantly reduces the total running cost without the entropy term,  accumulated over $[0, 500]$. 
Together with the improved learning speed, this implies that our method better balances the exploration-exploitation tradeoff compared to the standard method.
To see how fast the two methods stabilize the system, 
we also compute the settling time, defined as the earliest time after which the trajectory stays in the interval $[-1, 1]$:
\[\mbox{Settling time}:=\min\left\{t\ge0 ~:~ \max_{1\le i\le 10} |x_i(s)|\le 1,\quad\mbox{for}\quad s\ge t\right\}.\]
As reported in Table~ \ref{tbl},
the settling time of the maximum entropy method is $t=6.60$, while that of the standard method is $t=8.36$.  
This result indicates that our method better stabilizes the system during the learning process compared to  the standard method.

Another remarkable result is the difference in sample efficiency. 
Our method needs 155 data to satisfy the rank condition \eqref{rank_on} in each iteration, while the standard method needs 431.5 data on average, as reported in Table \ref{tbl}.
Interestingly, the smallest number of data required to meet the rank condition is 155. 
This implies that our maximum entropy method optimally performs exploration in the sense of satisfying the rank condition. 
As a result of sample efficiency, our method outperforms the standard method in terms of both learning speed and computation time.

We finally examine the effect of the temperature parameter $\alpha$ in balancing the exploitation-exploration tradeoff. As shown in Table~\ref{tbl_alpha_on},  
the average sample size required to satisfy the rank condition decreases with the temperature parameter $\alpha$. 
This result is consistent with our intuition that a control with higher entropy has a better exploration capability than that with lower entropy. 
As a result, for $\alpha \geq 0.05$, the total running cost  decreases as $\alpha$ decreases or, equivalently, as the entropy of our control diminishes.
In this range, the performance increases as the controller focuses more on exploitation. 
However, for $\alpha \leq 0.05$, the total running cost increases as $\alpha$ decreases.
In this range, the controller needs a better exploration capability to present a better performance. Therefore, there exists an appropriate range of $\alpha$  to balance the exploitation-exploration tradeoff; in our case, $\alpha \approx  0.05$ is a reasonable choice.

\subsubsection{Off-Policy Method} \label{sec:exp_off}

\begin{figure}[tb]
	\centering
	\subfigure[]{
		\includegraphics[width=0.32\columnwidth]{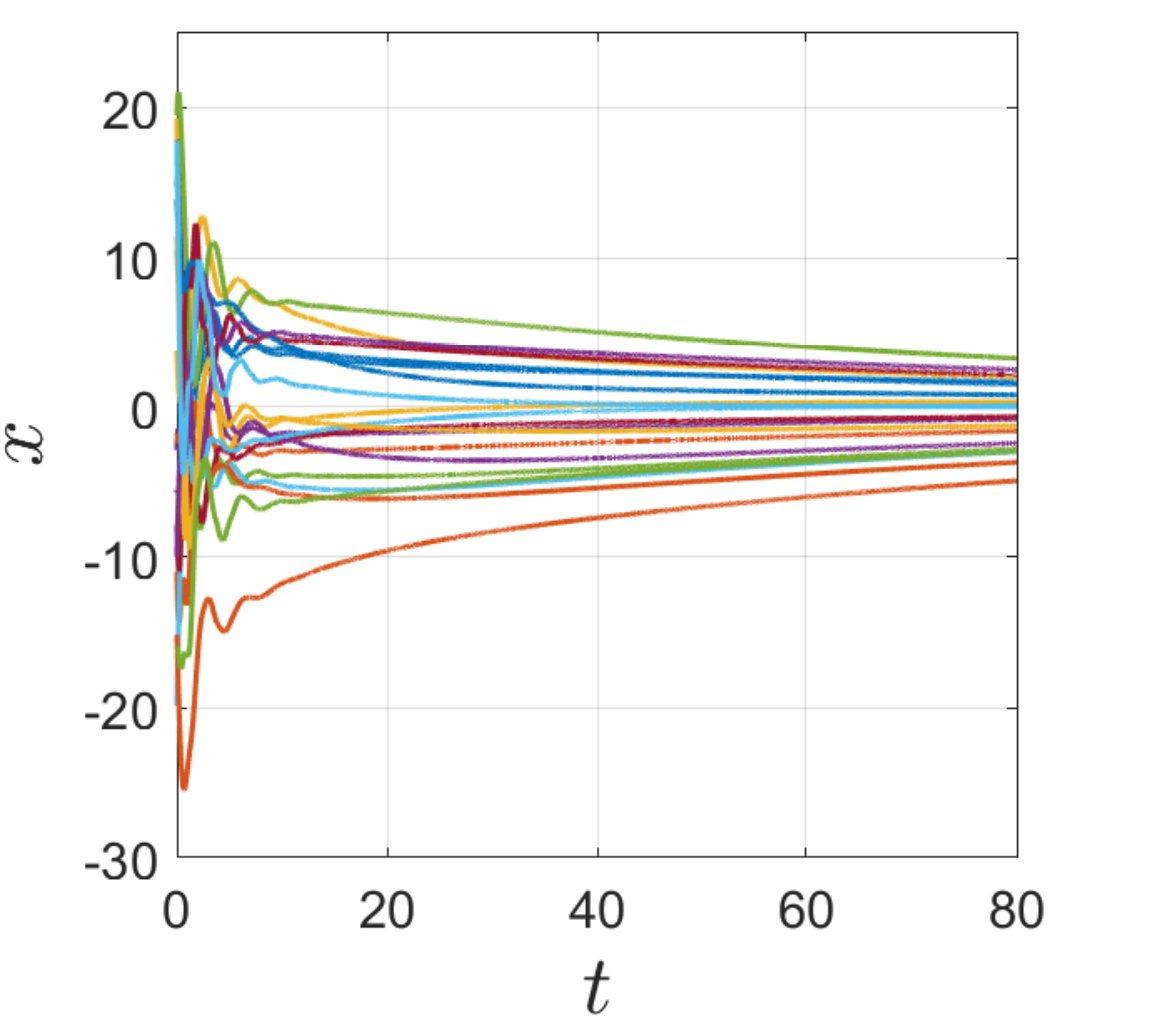}}
	\subfigure[]{
		\includegraphics[width=0.32\columnwidth]{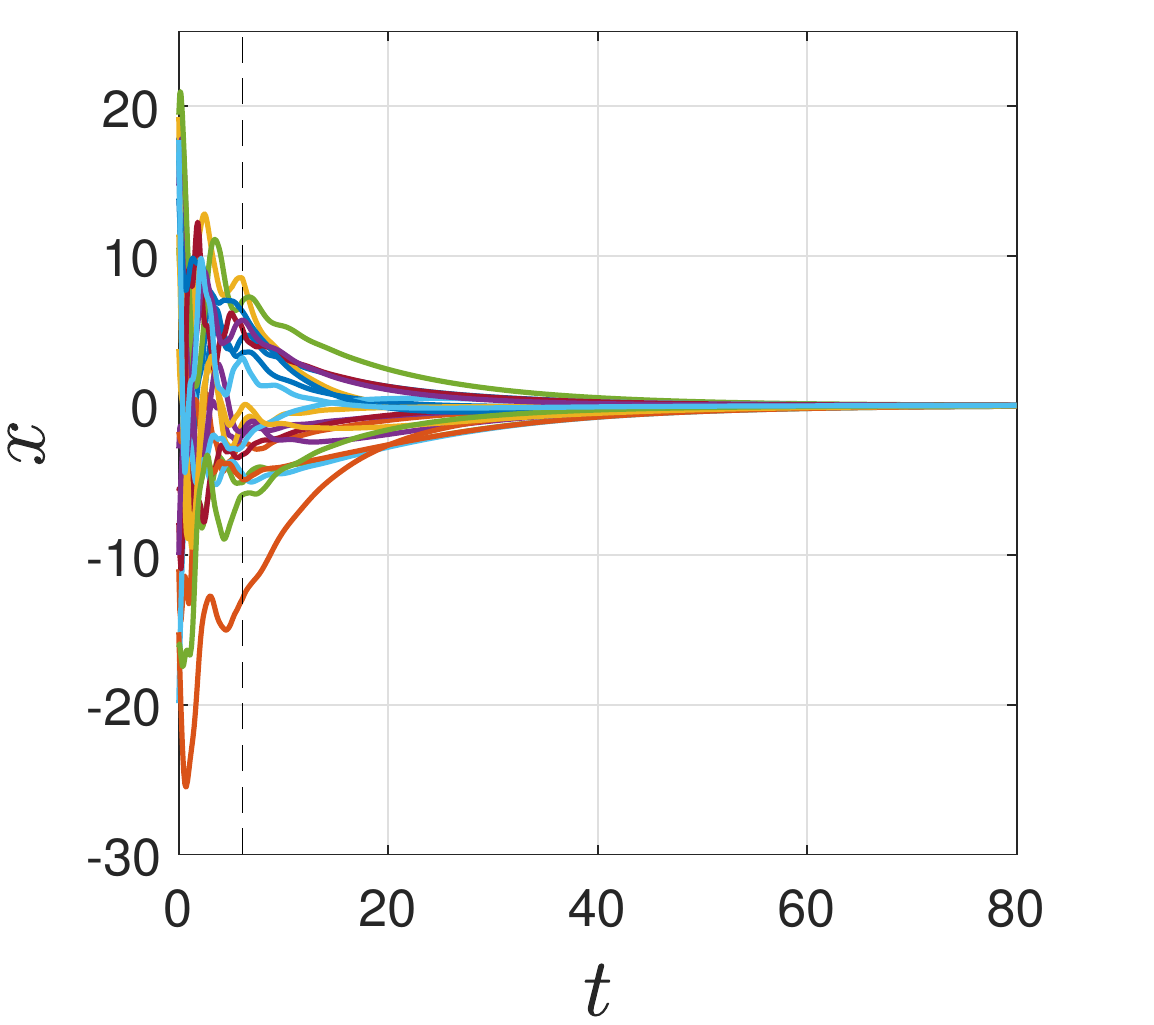}}
	\subfigure[]{
		\includegraphics[width=0.32\columnwidth]{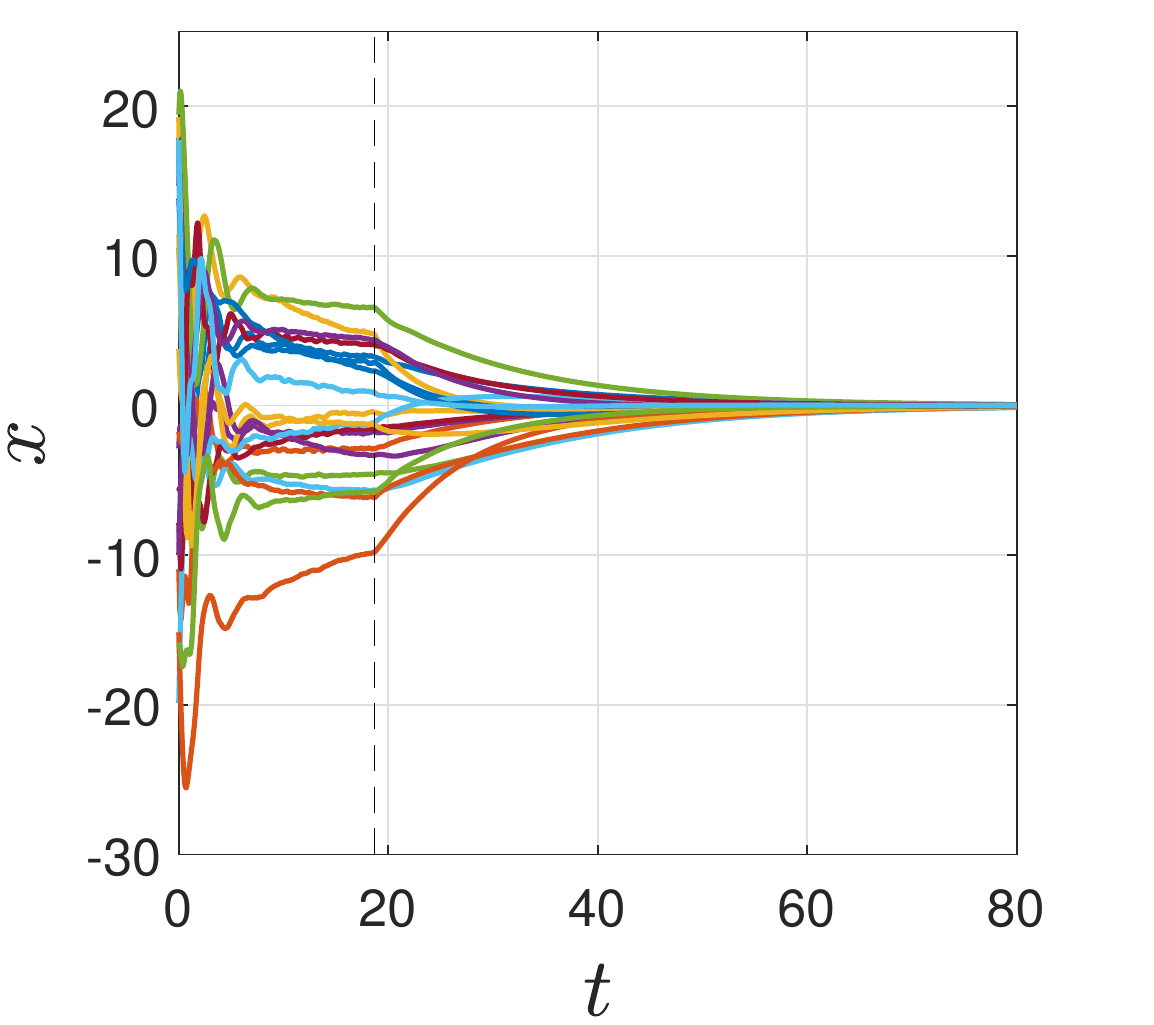}}
	\caption{Off-policy case: System trajectories controlled by (a) $u \equiv 0$ (uncontrolled), (b) the maximum entropy method with $\alpha = 1$, and (c) the standard method. The vertical line indicates  when learning is completed.}
	\label{traj_off}
\end{figure}

\begin{table}[tb]
	\caption{Off-policy case: Quantitative comparisons between the maximum entropy method and the standard adaptive DP method in \cite{jiang2014robust}.}
	\centering
	\begin{tabular}{ c |  >{\centering\arraybackslash}p{2.9cm} >{\centering\arraybackslash}p{2.9cm} >{\centering\arraybackslash}p{2.9cm}  }
	\hline
		 & \centering Uncontrolled & Max entropy $(\alpha=1)$ & Standard   \\ \hline\hline
Total running cost   & 2983.3  & 1258.5 & 1780.4 \\
		Settling time 	     & $t=194.40$ & $t=31.47$ & $t=44.14$    \\
		Total \# of data  &  N/A  & 610 & 1871         \\
		Learning time   & N/A& $t = 6.10$ & $t = 18.71$ \\
		Computation time (sec)  & N/A& 9.74& 93.66 \\\hline
	\end{tabular}
	\label{tbl_off}
\end{table}

We now consider the linear systems with 20-dimensional state and action spaces, i.e., $x(t) \in \bbr^{20}$ and $u(t)\in \bbr^{20}$.
The matrix $A$ is chosen to be Hurwitz;  thus, $K_0=0$ is a valid initial gain matrix.  The eigenvalues of $A$ are  $-3.9597$, $-3.5452$,  $-1.9443$, $-1.7884$, $-1.0196\pm4.1172\textup{i}$, $-0.8337$, $-0.5845$, $-0.4872$, $-0.4231$, $-0.4007\pm 1.9894\textup{i}$, $-0.3531$, $-0.3220$, $-0.2948$, $-0.2543$, $-0.2543$, $-0.2112\pm3.3826\textup{i}$ and $-0.01$.
We used the same running cost $r$,  discount factor $\lambda$, sample time $\delta t$, and sinusoidal noise signal $e$ as those used in the on-policy methods. The  threshold for the stopping criterion was chosen as $\varepsilon= 10^{-3}$. 
Figure~\ref{traj_off} shows the uncontrolled state trajectories and the trajectories controlled by the two off-policy methods. 
As in the on-policy case, our maximum entropy method learns the optimal gain matrix much faster than the standard method does. 
Unlike the on-policy methods, the off-policy counterparts collect data without any update on the gain matrix  until the vertical dashed-lines in Figure~\ref{traj_off}.
At this time instance, the optimal gain matrix is constructed according to Algorithm~\ref{alg:off-policy}
and then applied to the system. 
Thus, the state trajectories are not smooth only at this single time instance, whereas the on-policy methods present many non-smooth instances.

Table~\ref{tbl_off} shows the quantitative results of our experiments in the off-policy case. 
The same performance measures are used as those in the on-policy case. 
As shown in Table~\ref{tbl_off}, the maximum entropy method outperforms its standard counterpart in terms of the total running cost, response speed, sample efficiency, learning speed, and computation time.

The effect of temperature $\alpha$ on the off-policy maximum entropy method is shown in Table~\ref{tbl_alpha_off}. 
The overall tendency is similar to that in the on-policy case. 
As $\alpha$ increases, the exploration capability of our control improves.  Thus, the control with high entropy (large $\alpha$) can learn the optimal gain matrix using a small sample. 
The total running costs indicate the exploration-exploitation tradeoff in the off-policy method as well;
in this case, $\alpha \approx 0.1$ balances the tradeoff reasonably well.

\begin{table}[tb]
	\caption{Effect of $\alpha$ on the off-policy method.}
	\centering
	\begin{tabular}{ c | >{\centering\arraybackslash}p{1.4cm} >{\centering\arraybackslash}p{1.4cm} >{\centering\arraybackslash}p{1.4cm} >{\centering\arraybackslash}p{1.4cm} >{\centering\arraybackslash}p{1.4cm} >{\centering\arraybackslash}p{1.4cm} >{\centering\arraybackslash}p{1.4cm}  }
\hline
		$\alpha$   & $1$ & $0.1$ & $0.01$ & $0.005$ & $0.004$ & $0.003$ & $0.002$  \\ \hline
		Total \# of data   & 610 & 610 & 615 & 687& 742 & 801 & 1537\\
Total running cost   & 1258.5 & 1137.4 & 1141.4  &1175.1 & 1196.5  & 1225.7 & 1479.8\\\hline
	\end{tabular}
	\label{tbl_alpha_off}
\end{table}

\appendix

\section{Interpretation of Relaxed Control Systems}\label{app:relax}

We  provide another interpretation of the dynamical system \eqref{dyn2} with relaxed control in terms of differential inclusions, when the control set $U$ is compact. For any control $\mu\in \mathcal{M}$, the system \eqref{dyn2} is a solution to the following differential inclusion:
	\begin{equation}\label{dyn_DI}
	\dot{x}(t) \in F(x(t)),\quad F(\bm{x})=\overline{\mathrm{conv}[f(\bm{x},U)]},\quad f(\bm{x},U):=\{f(\bm{x},\bm{u})~\mid \bm{u}\in U\},
	\end{equation}
	where $\mathrm{conv}[A]$ denotes the convex hull of  set $A$. To see this, it suffices to show that $\bar{f}:=\int_{U}f(\bm{x},\bm{u})\mu(t,\d \bm{u})\in \overline{\mathrm{conv}[f(\bm{x},U)]}$. Suppose $\bar{f}\in (\overline{\mathrm{conv}[f(\bm{x},U)]})^c$. Note that $\overline{\mathrm{conv}[f(\bm{x},U)]}$ is  a closed convex set, and $\{\bar{f}\}$ is a singleton, which is convex and compact. Thus, there exists a hyperplane $\{\bm{x}~\mid~ a^\top \bm{x}=b\}$ such that
	\[a^\top \bar{f}>b,\quad a^\top f(\bm{x},\bm{u})<b \quad \forall \bm{u}\in U.\]
It follows from the definition of $\bar{f}$ that
	\[a^\top \bar{f} = \int_{U} a^\top f(\bm{x},\bm{u})\,\mu(t,\d \bm{u}) \le\int_{U}b\,\mu(t,\d u) =b,\]
	which is a contradiction. Therefore, $\bar{f}\in \overline{\mathrm{conv}[f(\bm{x},U)]}$.

\section{Lemmas}	\label{app:lem}

In this appendix, we provide the mathematical lemmas used in the paper. Let $X$ be a separable metric space and $\mathcal{P}(X)$ be the set of all Borel probability measure defined on $X$. 

\begin{lemma}[\cite{budhiraja2019analysis}]\label{LA.1}
Let  $\phi:X\to \bbr$ be any bounded measurable function. Then, for any $\mu, \gamma\in\mathcal{P}(X)$,
\[-\log\int_X e^{-\phi(x)}\gamma(\d x) \le \int_X \phi(x)\mu(\d x)+\kl(\mu||\gamma).\]	
Moreover, the equality holds if and only if $\mu(\d x) = \frac{e^{-\phi(x)}\gamma(\d x)}{\int_X e^{-\phi(x)}\gamma(\d x)}$.
\end{lemma}
\begin{proof}
	We define a new measure $\nu\in \mathcal{P}(X)$ as
	\[\frac{\d \nu}{\d\gamma}(x) = \frac{e^{-\phi(x)}}{\int_X e^{-\phi(x)}\gamma(\d x)}.\]
	Then, we have
	\begin{align*}
	\int_X \phi(x)\mu(\d x) +\kl(\mu||\gamma) &= \int_X \phi(x)\mu(\d x)+\int_X \log\left(\frac{\d \mu}{\d\gamma}\right)\d \mu\\
	&=\int_X \phi(x)\mu(\d x) +\int_X \log\left(\frac{\d\mu}{\d\nu}\right)\d \mu +\int_X \log\left(\frac{\d\nu}{\d\gamma}\right)\d\mu\\
	&=-\log\int_X e^{-\phi(x)}\gamma(\d x) +\kl(\mu||\nu)\ge-\log\int_X e^{-\phi(x)}\gamma(\d x).
	\end{align*}
The equality of the last inequality holds if and only if $\mu=\nu=\frac{e^{-\phi(x)}\gamma(\d x)}{\int_X e^{-\phi(x)}\gamma(\d x)}$. 
\end{proof}

\begin{lemma}[Lower semicontinuity of relative entropy~\cite{posner1975random}]\label{LA.2}
	Let $\mu, \gamma\in \mathcal{P}(X)$. If $\{\mu_n\}$ is a sequence in $\mathcal{P}(X)$ such that $\mu_n$ weakly converges to $\mu$, then
	\[\kl(\mu||\gamma)\le \liminf_{n\to\infty} \kl(\mu_n||\gamma).\]
\end{lemma}

Using the lower semicontinuity of relative entropy, we can prove that the level set of $\mathcal{P}(X)$ with respect to the relative entropy is compact.

\begin{lemma}[Compactness of level sets~\cite{budhiraja2019analysis}]\label{LB.3}
	For any $\gamma\in \mathcal{P}(X)$ and $M>0$, the level set $\mathcal{P}_M:=\{\mu\in \mathcal{P}(X) \mid \kl(\mu||\gamma)\le M\}$ is compact.
\end{lemma}
\begin{proof}
	Let $\{\mu_n\}_{n=1}^\infty$ be a sequence in $\mathcal{P}(X)$ such that $\kl(\mu_n||\gamma)\le M$. We will show that the sequence $\{\mu_n\}$ is tight. It follows from Lemma \ref{LA.1} that for any bounded measurable function $\phi:X\to \bbr$,
	\begin{equation}\label{A-1}
	\int_X \phi(x) \mu_n (\d x) -\log \int_{X}e^{\phi(x)}\d \gamma \le \kl(\mu_n||\gamma)\le M.
	\end{equation}
	Now, let $\delta>0$ be a fixed constant. We choose a sufficiently small $\e$ so that $\frac{M+\log 2}{\log\left(1+\frac{1}{\e}\right)}\le \delta$.	Since the singleton of a measure is tight, there exists a compact subset $K$ of $X$ that $\gamma(K^c)\le \e$. Now, we define a map $\phi$ as
	\[\phi(x):=\begin{cases}
	0,\quad \mbox{if}\quad x\in K,\\
	\log\left(1+\frac{1}{\e}\right)\quad \mbox{if}\quad x\in K^c.
	\end{cases}\]
By \eqref{A-1}, we have
	\[\log\left(1+\frac{1}{\e}\right)\mu_n(K^c) -\log\left(\gamma(K)+\left(1+\frac{1}{\e}\right)\gamma(K^c)\right)\le M
	\]
Rearranging the terms yields
	\begin{align*}\mu_n(K^c)&\le \frac{1}{\log\left(1+\frac{1}{\e}\right)}\left(M+\log\left(\gamma(K)+\left(1+\frac{1}{\e}\right)\gamma(K^c)\right)\right)=\frac{1}{\log\left(1+\frac{1}{\e}\right)}\left(M+\log\left(1+\frac{1}{\e}\gamma(K^c)\right)\right)\\
	&\le \frac{1}{\log\left(1+\frac{1}{\e}\right)}(M+\log 2)\le \delta.
	\end{align*}
This estimate implies that the sequence of measures $\{\mu_n\}_{n=1}^\infty$ is tight. By the Prokhorov's Theorem, there exists a subsequence $\{\mu_{n_k}\}_{k=1}^\infty$ of $\{\mu_n\}$ and $\mu\in \mathcal{P}(X)$ such that $\mu_{n_k}$ weakly converges to $\mu$. Finally, the lower semi-continuity of $\kl(\cdot||\gamma)$ (Lemma \ref{LA.2}) implies that
	\[\kl(\mu||\gamma) \le \liminf_{k\to\infty}\kl(\mu_{n_k}||\gamma)\le M.\]
	Therefore, $\mu\in \mathcal{P}_M$ and this implies that $\mathcal{P}_M$ is compact.
\end{proof}

The following is the disintegration theorem of the measure on product space.

\begin{lemma}[Disintegration theorem~\cite{villani2008optimal}]\label{LB.4} Let $X$ and $Y$ be Polish spaces. If $\pi$ is a probability measure on $X\times Y$ with marginal $\mu$ on $X$, i.e., $\pi(A\times Y) = \mu(A)$ for any measurable set $A\subset X$, then there exists a measurable map $x\mapsto \pi_x$ from $X$ to $\mathcal{P}(Y)$ such that
\[\pi = \int_X (\delta_x \otimes \pi_x)\mu(\d x),\]
or equivalently
\[\int_{X\times Y} \phi(x,y)\pi(\d x,\d y) = \int_X\left(\int_Y \phi(x,y)\pi_x(\d y)\right)\mu(\d x).\]
\end{lemma}

\begin{lemma}\label{stability}
Suppose that Assumption \ref{assumption} holds, and the control set $U$ is compact. Let $x = x(t)$ be the solution to \eqref{dyn2} with control $\mu$ and initial data $\bm{x}$. Then, there exists a constant $C$ such that $x(t)$ satisfies the following stability estimates:
\begin{align*}
	&|x(t)|\le e^{Lt}|\bm{x}| +\frac{C}{L}(e^{Lt}-1),\\
	&|x(t)-x(s)|\le C(t-s),\quad 0\le s\le  t.
\end{align*}
Moreover, let $y=y(t)$ be another solution to \eqref{dyn2} with  initial data ${\bm y}$. Then,
\[|x(t)-y(t)|\le e^{Lt}|\bm{x}-\bm{y}|.\] 
\end{lemma}

\begin{proof}
We first compute that
	\[\frac{\d |x(t)|^2}{\d t} = 2x(t)\cdot \int_U f(x(t),\bm{u})\mu(t;\d \bm{u})=2\int_U x(t)\cdot f(x(t),\bm{u})\mu(t;\d \bm{u}).\]
By Assumption \ref{assumption},
	\[x(t)\cdot f(x(t),\bm{u}) \le x(t)\cdot f(0,\bm{u})+L|x(t)|^2\le C|x(t)|+L|x(t)|^2,\]
	where $C=\sup_{\bm{u} \in U}|f(0,\bm{u})|$. Therefore, we have
	\[2|x(t)|\frac{\d |x(t)|}{\d t}=\frac{\d|x(t)|^2}{\d t} \le 2C|x(t)|+2L|x(t)|^2,\]
	which implies that
	\[\frac{\d|x(t)|}{\d t} \le C +L|x(t)|.\]
	By the Gr\"onwall--Bellman inequality, we deduce that
	\[|x(t)|\le e^{Lt}|x(0)| +\frac{C}{L}(e^{Lt}-1).\]
	
We now show the second inequality in the statement. By definition,
	\[x(t) = x(s) +\int_s^t \int_Uf(x(\tau),\bm{u})\mu(\tau;\d \bm{u})\,\d\tau.\]
	Since  $x(t)\le e^{Lt}|x(0)|+\frac{C}{L}(e^{Lt}-1)$, which is independent of $\mu$, we obtain
	\[|x(t)-x(s)|\le \int_s^t \int_U|f(x(\tau),\bm{u})|\mu(\tau;\d \bm{u})\d \tau\le C(t-s).\]
	Finally, we notice that
	\begin{align*}
	\frac{\d}{\d t}|x(t)-y(t)|^2&=2(x(t)-y(t))\cdot \int_{U} (f(x(t),\bm{u})-f(y(t),\bm{u}))\mu(t;\d u) \\
	&\le 2L\int_U |x(t)-y(t)|^2\mu(t;\d \bm{u}) = 2L|x(t)-y(t)|^2.
	\end{align*}
	Thus,
	\[\frac{\d}{\d t}|x(t)-y(t)|\le L|x(t)-y(t)|,\]
	and the result follows.
\end{proof}

\bibliographystyle{plain}

\bibliography{reference}

\begin{thebibliography}{10}

\bibitem{artstein1978relaxed}
Z.~Artstein.
\newblock Relaxed controls and the dynamics of control systems.
\newblock {\em SIAM Journal on Control and Optimization}, 16:689--701, 1978.

\bibitem{Bardi97}
M.~Bardi and I.~Capuzzo-Dolcetta.
\newblock {\em Optimal {C}ontrol and {V}iscosity {S}olutions of
  {Hamilton--Jacobi--Bellman} {E}quations}.
\newblock Birkh\"{a}user, 1997.

\bibitem{bardi1991nonconvex}
M.~Bardi and S.~Osher.
\newblock The nonconvex multidimensional {R}iemann problem for
  {H}amilton--{J}acobi equations.
\newblock {\em SIAM Journal on Mathematical Analysis}, 22:344--351, 1991.

\bibitem{Bian2016}
T.~Bian and Z.~P. Jiang.
\newblock Value iteration and adaptive dynamic programming for data-driven
  adaptive optimal control design.
\newblock {\em Automatica}, 71:348--360, 2016.

\bibitem{Bian2019}
T.~Bian and Z.~P. Jiang.
\newblock Continuous-time robust dynamic programming.
\newblock {\em SIAM Journal on Control and Optimization}, 57(6):4150--4174,
  2019.

\bibitem{brzezniak2013optimal}
Z.~Breze{\'z}niak and R.~Serrano.
\newblock Optimal relaxed control of dissipative stochastic partial
  differential equations in {B}anach spaces.
\newblock {\em SIAM Journal on Control and Optimization}, 51:2664--2703, 2013.

\bibitem{budhiraja2019analysis}
Amarjit Budhiraja and Paul Dupuis.
\newblock {\em Analysis and Approximation of Rare Events: Representations and
  Weak Convergence Methods}, volume~94.
\newblock Springer, 2019.

\bibitem{Busic2018}
A.~Busic and S.~Meyn.
\newblock Ordinary differential equation methods for {Markov} decision
  processes and application to {Kullback}--{Leibler} control cost.
\newblock {\em SIAM Journal on Control and Optimization}, 56(1):343--366, 2018.

\bibitem{chow2017algorithm}
Y.~T. Chow, J.~Darbon, S.~Osher, and W.~Yin.
\newblock Algorithm for overcoming the curse of dimensionality for
  time-dependent non-convex {Hamilton--Jacobi} equations arising from optimal
  control and differential games problems.
\newblock {\em Journal of Scientific Computing}, 73:617--643, 2017.

\bibitem{chow2019algorithm}
Y.~T. Chow, J.~Darbon, S.~Osher, and W.~Yin.
\newblock Algorithm for overcoming the curse of dimensionality for
  state-dependent {Hamilton-Jacobi} equations.
\newblock {\em Journal of Computational Physics}, 387:376--409, 2019.

\bibitem{CrandallEvansLions84}
M.~Crandall, L.~C. Evans, and P.-L. Lions.
\newblock Some properties of viscosity solutions of {Hamilton--Jacobi}
  equations.
\newblock {\em Transactions of the American Mathematical Society},
  282:487--502, 1984.

\bibitem{CrandallLions83}
M.~Crandall and P.-L. Lions.
\newblock Viscosity solutions of {Hamilton--Jacobi} equations.
\newblock {\em Transactions of the American Mathematical Society}, 277:1--42,
  1983.

\bibitem{crandall1984two}
M.~Crandall and P.-L. Lions.
\newblock Two approximations of solutions of {H}amilton--{J}acobi equations.
\newblock {\em Mathematics of Computation}, 43:1--19, 1984.

\bibitem{darbon2016algorithms}
J.~Darbon and S.~Osher.
\newblock Algorithms for overcoming the curse of dimensionality for certain
  {Hamilton--Jacobi} equations arising in control theory and elsewhere.
\newblock {\em Research in the Mathematical Sciences}, 3:19, 2016.

\bibitem{dembo2010largedeviation}
A.~Dembo and O.~Zeitouni.
\newblock {\em Large {D}eviations {T}echniques and {A}pplications}.
\newblock Springer-Verlag Berlin Heidelberg, 2010.

\bibitem{Doya2000}
K.~Doya.
\newblock Reinforcement learning in continuous time and space.
\newblock {\em Neural Computation}, 12:219--245, 2000.

\bibitem{Dvijotham2011}
K.~Dvijotham and E.~Todorov.
\newblock A unifying framework for linearly solvable control.
\newblock In {\em Proceedings of the Twenty-Seventh Conference on Uncertainty
  in Artificial Intelligence}, pages 178--186, 2011.

\bibitem{Evans2010}
L.~C. Evans.
\newblock {\em Partial Differential Equations}.
\newblock American Mathematical Society, 2010.

\bibitem{feller2008introduction}
W.~Feller.
\newblock {\em An Introduction to {P}robability {T}heory and {I}ts
  {A}pplications, Vol 2}.
\newblock John Wiley \& Sons, 2008.

\bibitem{fleming1980measure}
W.~H. Fleming.
\newblock Measure-valued processes in the control of partially-observable
  stochastic systems.
\newblock {\em Applied Mathematics and Optimization}, 6:271--285, 1980.

\bibitem{Fox2016}
R.~Fox, A.~Pakman, and N.~Tishby.
\newblock Taming the noise in reinforcement learning via soft updates.
\newblock In {\em Proceedings of the Thirty-Second Conference on Uncertainty in
  Artificial Intelligence}, pages 202--211, 2016.

\bibitem{Guan2014}
P.~Guan, M.~Raginsky, and R.~M. Willett.
\newblock Online {Markov} decision processes with {Kullback}--{Leibler} control
  cost.
\newblock {\em IEEE Transactions on Automatic Control}, 59(6):1423--1438, 2014.

\bibitem{Haarnoja2017}
T.~Haarnoja, H.~Tang, P.~Abbeel, and S.~Levine.
\newblock Reinforcement learning with deep energy-based policies.
\newblock In {\em International Conference on Machine Learning}, pages
  1352--1361, 2017.

\bibitem{Haarnoja2018}
T.~Haarnoja, A.~Zhou, K.~Hartikainen, G.~Tucker, S.~Ha, J.~Tan, V.~Kumar,
  H.~Zhu, A.~Gupta, P.~Abbeel, and S.~Levine.
\newblock Soft actor-critic algorithms and applications.
\newblock {\em arXiv preprint arXiv:1812.05905}, 2018.

\bibitem{haussmann1990existence}
U.~G. Haussmann and J.~P. Lepeltier.
\newblock On the existence of optimal controls.
\newblock {\em SIAM Journal on Control and Optimization}, 28:851--902, 1990.

\bibitem{Hazan2019}
E.~Hazan, S.~Kakade, K.~Singh, and A.~{Van Soest}.
\newblock Provably efficient maximum entropy exploration.
\newblock In {\em International Conference on Machine Learning}, pages
  2681--2691, 2019.

\bibitem{hopf1965generalized}
E.~Hopf.
\newblock Generalized solutions of non-linear equations of first order.
\newblock {\em Journal of Mathematics and Mechanics}, 14:951--973, 1965.

\bibitem{jiang2014robust}
Y.~Jiang and Z.-P. Jiang.
\newblock {\em Robust Adaptive Dynamic Programming}.
\newblock John Wiley \& Sons, 2017.

\bibitem{Kappen2005}
H.~J. Kappen.
\newblock Linear theory for control of nonlinear stochastic systems.
\newblock {\em Physical Review Letters}, 95(20):200201, 2005.

\bibitem{Kappen2005b}
H.~J. Kappen.
\newblock Path integrals and symmetry breaking for optimal control theory.
\newblock {\em Journal of Statistical Mechanics: Theory and Experiment},
  11:P11011, 2005.

\bibitem{Kappen2012}
H.J. Kappen, V.~G\'{o}mez, and M.~Opper.
\newblock Optimal control as a graphical model inference problem.
\newblock {\em Machine Learning}, 87(2):159--182, 2012.

\bibitem{kim2020hamilton}
J.~Kim and I.~Yang.
\newblock {Hamilton--Jacobi--Bellman} equations for {Q}-learning in continuous
  time.
\newblock In {\em Learning for Dynamics and Control}, pages 739--748. PMLR,
  2020.

\bibitem{kleinman1968iterative}
D.~Kleinman.
\newblock On an iterative technique for {R}iccati equation computations.
\newblock {\em IEEE Transactions on Automatic Control}, 13:114--115, 1968.

\bibitem{Lee2019}
K.~Lee, S.~Kim, S.~Lim, S.~Choi, and S.~Oh.
\newblock Tsallis reinforcement learning: A unified framework for maximum
  entropy reinforcement learning.
\newblock {\em arXiv preprint arXiv:1902.00137}, 2019.

\bibitem{lewis2012optimal}
F.~L. Lewis, D.~Vrabie, and V.~L. Syrmos.
\newblock {\em Optimal {C}ontrol}.
\newblock John Wiley \& Sons, 2012.

\bibitem{mcshane1967relaxed}
E.~J. McShane.
\newblock Relaxed controls and variational problems.
\newblock {\em SIAM Journal on Control}, 5:438--485, 1967.

\bibitem{Munos2000}
R.~Munos.
\newblock A study of reinforcement learning in the continuous case by the means
  of viscosity solutions.
\newblock {\em Machine Learning}, 40:265--299, 2000.

\bibitem{osher1984riemann}
S.~Osher.
\newblock Riemann solvers, the entropy condition, and difference.
\newblock {\em SIAM Journal on Numerical Analysis}, 21:217--235, 1984.

\bibitem{osher1991high}
S.~Osher and C.-W. Shu.
\newblock High-order essentially nonoscillatory schemes for {Hamilton--Jacobi}
  equations.
\newblock {\em SIAM Journal on Numerical Analysis}, 28:907--922, 1991.

\bibitem{Palanisamy2015}
M.~Palanisamy, H.~Modares, F.~L. Lewis, and M.~Aurangzeb.
\newblock Continuous-time {Q}-learning for infinite-horizon discounted cost
  linear quadratic regulator problems.
\newblock {\em IEEE Transactions on Cybernetics}, 45(2):165--176, 2015.

\bibitem{Peters2010}
J.~Peters, K.~M\"{u}lling, and Y.~Alt\"{u}n.
\newblock Relative entropy policy search.
\newblock In {\em Proceedings of the Twenty-Fourth AAAI Conference on
  Artificial Intelligence}, pages 1607--1612, 2010.

\bibitem{posner1975random}
E.~Posner.
\newblock Random coding strategies for minimum entropy.
\newblock {\em IEEE Transactions on Information Theory}, 21:388--391, 1975.

\bibitem{Rawlik2013}
K.~Rawlik, M.~Toussaint, and S.~Vijayakumar.
\newblock On stochastic optimal control and reinforcement learning by
  approximate inference.
\newblock In {\em Robotics: Science and Systems}, pages 353--360, 2013.

\bibitem{rudin1976PMA}
W.~Rudin.
\newblock {\em Principles of {M}athematical {A}nalysis}.
\newblock McGraw-Hill, 1976.

\bibitem{Theodorou2010}
E.~Theodorou, J.~Buchli, and S.~Schaal.
\newblock A generalized path integral control approach to reinforcement
  learning.
\newblock {\em The Journal of Machine Learning Research}, 11:3137--3181, 2010.

\bibitem{Theodorou2012}
E.~A. Theodorou and E.~Todorov.
\newblock Relative entropy and free energy dualities: Connections to path
  integral and {KL} control.
\newblock In {\em Proceedings of the 51st IEEE Conference on Decision and
  Control}, pages 1466--1473, 2012.

\bibitem{todorov2007linearly}
E.~Todorov.
\newblock Linearly-solvable {M}arkov decision problems.
\newblock In {\em Advances in Neural Information Processing Systems}, pages
  1369--1376, 2007.

\bibitem{Todorov2008}
E.~Todorov.
\newblock General duality between optimal control and estimation.
\newblock In {\em Proceedings of the 47th IEEE Conference on Decision and
  Control}, pages 4286--4292, 2008.

\bibitem{todorov2009efficient}
E.~Todorov.
\newblock Efficient computation of optimal actions.
\newblock {\em Proceedings of the National Academy of Sciences},
  106:11478--11483, 2009.

\bibitem{Vamvoudakis2017}
K.~G. Vamvoudakis.
\newblock Q-learning for continuous-time linear systems: A model-free infinite
  horizon optimal control approach.
\newblock {\em Systems \& Control Letters}, 100:14--20, 2017.

\bibitem{villani2008optimal}
C.~Villani.
\newblock {\em Optimal {T}ransport: {O}ld and {N}ew}, volume 338.
\newblock Springer Science \& Business Media, 2008.

\bibitem{wang2019exploration}
H.~Wang, T.~Zariphopoulou, and X.~Y. Zhou.
\newblock Exploration versus exploitation in reinforcement learning: {A}
  stochastic control approach.
\newblock {\em arXiv preprint arXiv:1812.01552}, 2019.

\bibitem{warga1962relaxed}
J.~Warga.
\newblock Relaxed variational problems.
\newblock {\em Journal of Mathematical Analysis and Applications}, 4:111--128,
  1962.

\bibitem{warga2014optimal}
J.~Warga.
\newblock {\em Optimal {C}ontrol of {D}ifferential and {F}unctional
  {E}quations}.
\newblock Academic press, 2014.

\bibitem{williamson1976relaxed}
L.~J. Williamson and E.~Polak.
\newblock Relaxed controls and the convergence of optimal control algorithms.
\newblock {\em SIAM Journal on Control and Optimization}, 14:737--756, 1976.

\bibitem{yang2007adaptive}
B.-J. Yang and A.~J. Calise.
\newblock Adaptive control of a class of nonaffine systems using neural
  networks.
\newblock {\em IEEE Transactions on Neural Networks}, 18:1149--1159, 2007.

\bibitem{Yang2014}
I.~Yang, M.~Morzfeld, C.~J. Tomlin, and A.~J. Chorin.
\newblock Path integral formulation of stochastic optimal control with
  generalized costs.
\newblock In {\em IFAC Proceedings Volumes}, volume~47, pages 6994--7000, 2014.

\bibitem{yang2017hamiltonian}
Y.~Yang, D.~Wunsch, and Y.~Yin.
\newblock Hamiltonian-driven adaptive dynamic programming for continuous
  nonlinear dynamical systems.
\newblock {\em IEEE Transactions on Neural Networks and Learning Systems},
  28:1929--1940, 2017.

\bibitem{young1942generalized}
L.~C. Young.
\newblock Generalized surfaces in the calculus of variations.
\newblock {\em Annals of Mathematics}, 43:84--103, 1942.

\bibitem{young1969lectures}
L.~C. Young.
\newblock {\em Lectures on the {C}alculus of {V}ariations and {O}ptimal
  {C}ontrol {T}heory}.
\newblock W. B. Saunders, Philadelphia, 1969.

\bibitem{Ziebart2010}
B.~D. Ziebart.
\newblock {\em Modeling Purposeful Adaptive Behavior with the Principle of
  Maximum Causal Entropy}.
\newblock PhD thesis, Carnegie Mellon University, 2010.

\end{thebibliography}

\end{document}